\documentclass{amsart}
\usepackage{amssymb,mathrsfs}
\usepackage{amsmath, amsthm,amsfonts,amssymb,url,color,epsfig,graphics}
\usepackage[english]{babel}
\usepackage{float}

\usepackage{hyperref}
\hypersetup{pdfborder=0 0 0}

\usepackage{color}

\usepackage{enumitem}

\usepackage{graphics}
\usepackage{comment}
\usepackage{graphicx}

\usepackage[textheight=25 cm, textwidth=16cm, headheight=10pt, headsep=20pt, footskip=10pt, left=1.8cm, right=1.8cm]{geometry}

\newtheorem{proposition}{Proposition}[section]
\newtheorem{theorem}[proposition]{Theorem}
\newtheorem{corollary}[proposition]{Corollary}
\newtheorem{lemma}[proposition]{Lemma}
\newtheorem*{bootstrap*}{Bootstrap Step}
\theoremstyle{definition}
\newtheorem{definition}[proposition]{Definition}
\newtheorem{remark}[proposition]{Remark}

\numberwithin{equation}{section}
\newtheorem{assumption}[proposition]{Assumption}

\definecolor{darklavender}{rgb}{0.56, 0.0, 1.0}
\definecolor{green}{rgb}{0.0, 0.5, 0.0}

\newcommand{\rob}[1]{\textcolor{black}{#1}}

\newcommand\eps{\varepsilon}

\def\Re{{\rm Re}}

\newcommand\de{{\partial}}

\newcommand\pp{{\upsilon}}

\newcommand{\hP}{\mathcal{P}}

\def\L{\lambda} 
\def\R{\mathbb{R}} 
\def\k{\textbf{k}} 
\def\W{\mathcal{W}} 
\def\Wbl{\mathcal{W}_{\rm{BL}}} 
\def\V{\mathcal{V}} 
\def\B{\mathcal{B}} 
\def\G{\chi}

\newcommand{\bcr}[1]{#1}
\newcommand{\bcb}[1]{#1}
\newcommand{\be}{\begin{equation}}
\newcommand{\ee}{\end{equation}}
\newcommand{{\kk}}{|\k|}

\newcommand{{\kko}}{|\k^0|}

\newcommand{\nn}{p}
\newcommand{\ta}{q}
\newcommand{\et}{\eta}
\newcommand{\bw}{{b.w.\;}}
\newcommand{\blbw}{{b.l.b.w.\;}}

\newcommand{\cha}{{\bf I}_{\mbox{\scriptsize{$y \ge 0$}}}}

\newcommand{\uu}{\textbf{u}}
\newcommand{\x}{\textbf{x}}

\usepackage{etoolbox}
\patchcmd{\subsubsection}{\itshape}{\itshape\bfseries}{}{} 

\title[]{Reflection of internal gravity waves \\in the form of quasi-axisymmetric beams}
\author[R.\ Bianchini]{Roberta Bianchini}
\address{Consiglio Nazionale delle Ricerche, Istituto per le Applicazioni del Calcolo, 00185 Rome, Italy}
\email{roberta.bianchini@cnr.it}
\author[T.\ Paul]{Thierry Paul}
\address{Sorbonne Universit\'e, Laboratoire Jacques-Louis Lions \& CNRS, F-75005 Paris, France}
\email{paul@ljll.math.upmc.fr}


\begin{document}
\maketitle

\begin{abstract} Preservation of the angle of reflection when an internal gravity wave hits a sloping boundary generates a focusing mechanism if the angle between the direction of propagation of the incident wave and the horizontal is close to the slope inclination (near-critical reflection). This paper provides an explicit description of the leading approximation of the unique Leray solution to the near-critical reflection of internal waves from a slope in the form of a beam wave. More precisely, our beam wave approach allows to construct a fully consistent and Lyapunov stable approximate solution, $L^2$ -close to the Leray solution, in the form of a beam wave, within a certain (nonlinear) time-scale. 
To the best of our knowledge, this is the first result where
a mathematical study of internal waves in terms of spatially localized beam waves is performed.\\
A beam wave is a linear superposition of rapidly oscillating plane waves, where the high frequency of oscillation is proportional to the inverse of a power of the small parameter measuring the weak amplitude of waves. \\
Being localized in the physical space thanks to rapid oscillations (and high variations of the modulus of the wavenumber), beams are physically more relevant than plane waves/packets of waves, whose wavenumber is nearly fixed (microloca\-li\-zed). At the mathematical level, this marks a strong difference between the previous plane waves/packets of waves analysis and our approach. \\
The main novelty of this work is to exploit the spatial localization of beam waves to exhibit a spatially localized, physically relevant solution and to improve the previous mathematical results from a twofold perspective: 1) our beam wave approximate solution is the sum of a finite number of terms, each of them is a consistent solution to the system and there is no artificial/non-physical corrector; 2) thanks to the absence of artificial correctors (used in the previous results) and to the special structure of the nonlinear term, we can push the expansion of our solution to next orders, so improving the accuracy and enlarging the consistency time-scale.

Finally, our results provide a set of initial conditions localized on rays, for which the Leray solution maintains approximately in $L^2$ the same localization.
\end{abstract}

\tableofcontents

\section{Introduction}
This work is dedicated to the study of \textit{beam-type} solutions to the two-dimensional Boussinesq system 
\begin{align}
\de_t u - b \sin \gamma + \de_x p -  \nu  \Delta u &= -\delta (u\de_x + w \de_y) u,\notag\\
\de_t w - b \cos \gamma + \de_y p -   \nu \Delta w & =  -\delta (u\de_x + w \de_y) w,\notag\\
\de_t b + u\sin \gamma  + w \cos \gamma -  \kappa  \Delta b & =  -\delta (u\de_x + w \de_y)b,\notag\\
\de_x u + \de_y w&=0,
\label{eq:system}
\end{align}

in the half space $\mathbb{R}^2_+$, with the parameters $\delta, \nu, \kappa>0$ {and $\gamma \in (0, \frac \pi 2)$}, endowed with the boundary conditions
\begin{align}
u_{|y=0}=w_{|y=0}=\de_y b_{|y=0}=0.\label{eq:cond-BL}
\end{align}
System \eqref{eq:system} describes, after a suitable rotation of coordinates, the phenomenon of internal gravity waves reflecting off a flat boundary inclined at an angle $\gamma$ with respect to the horizontal plane. Specifically, the coordinates $(x, y)$ in \eqref{eq:system} undergo a rotation of angle $\gamma$ from the standard Cartesian coordinates (as given in \eqref{eq:coord1}). Similarly, the variables $(u, w)$ represent a rotation of angle $\gamma$ of the horizontal and vertical components of the velocity field, as presented in \eqref{eq:variables-rotation}. Notably, system \eqref{eq:system} is an outcome of rotating the original two-dimensional Boussinesq system, presented in \eqref{eq:system-vel-nodamp} below, by an angle $\gamma$.

In \eqref{eq:system}, the viscosity/diffusion coefficients satisfy the following 
\begin{align}
\label{eq:viscosity-size}
 \nu :=  \eps \nu_0; \quad  \kappa := \eps \kappa_0,
\end{align}
where now $\eps$ is a dimensionless parameter. The Boussinesq equations \eqref{eq:system} model an important class of incompressible flows with variable density. In particular, they find a wide application in oceanography, see for instance \cite{Rieutord}, where both the incompressibility and small viscosity/diffusivity assumptions are very good approximations of the reality.\\

Before {stating} our mail results in Section \ref{sec:results}, let us dedicate the rest of this section to the presentation of the model, the motivation of our results and the comparison with the existing literature.  
\subsection*{Formal derivation of the Boussinesq equations} 
The starting point for the formal derivation of the Boussinesq system \eqref{eq:system} is represented by the non-homogeneous incompressible Navier-Stokes equations with diffusivity
\begin{equation}\label{eq:nonlin}
	\begin{aligned}
	 \rho(\de_t+\uu\cdot\nabla)\uu+\nabla  p&=-\rho \, \textbf{g} + \nu \Delta \uu,\\
	(\de_t+\uu\cdot\nabla)\rho&=\kappa \Delta \rho,\\
	\nabla\cdot \uu&=0, 
	\end{aligned}
\end{equation}
where $\textbf{x}=(x_1, x_2) \in \R^2$, the spatial gradient $\nabla=(\de_{x_1}, \de_{x_2})^T$, the unknowns $\rho=\rho(t, \textbf{x}),  \,\textbf{u}= \textbf{u}(t, \textbf{x})=(u_1, u_2 )^T,$  $p=p(t, \textbf{x})$ represent density, velocity field and scalar pressure respectively, while $\textbf{g}=(0, g)$ is the gravity vector and $\nu, \kappa >0$ are the viscosity and diffusivity coefficients. A derivation of those equations starting from the Navier-Stokes Fourier system, in the limit of the Oberbeck-Boussinesq approximation, is performed for instance in \cite{Danchin13}. It is customary to introduce some additional hypotheses to obtain the so-called Boussinesq equations. We explain this in the following. In many physical systems of non-homogeneous fluids, the variations of the density profile are negligible compared to its (constant) average. One then assumes that the equilibrium stratification is a \emph{stable} profile $\bar \rho (x, y) = \bar \rho(y)$, with $\de_y\bar \rho (y) <0$. 
Among all the possible stratification's equilibria, one usually takes into account
locally affine profiles, so that $\de_y\bar{\rho}(y)$ is constant, see \cite{BDSR19, Lannes-Saut} and references therein.
One linearizes equations \eqref{eq:nonlin} around the \emph{hydrostatic equilibrium}, namely a steady solution with zero velocity field such that
\begin{align*}
(\rho, u, w, p)=(\bar \rho(y), 0, 0,  \bar p(y)),
\end{align*}
where $\bar p'(y)=-g \bar \rho$. More precisely, we consider the following expansions.
\begin{align*}
\rho(t,\x)&=\bar \rho(y)+ \tilde{\rho}(t, \x); \notag\\
u_1(t,\x)&=\tilde{u}_1(t, \x ), \; u_2(t,\x)=\tilde{u}_2(t, \x); \notag\\
p(t,\x)&=  \bar p(y)+\rho_0\tilde{P}(t,\x),
\end{align*}
with $\bar \rho (y)=\rho_0+r(y)$, where $\rho_0$ is the (constant) averaged density and $r(y)$ is a function of the vertical coordinate, such that $r'(y) < 0$. Plugging the previous expansions in system \eqref{eq:nonlin}, {one applies} the \emph{Boussinesq approximation}, see \cite{Rieutord}, which consists in neglecting density variations in all the terms but the one involving gravity. In other words, in the Boussinesq regime, the restoring force of equilibrium's fluctuations is gravity: it is an application of \emph{Archimedes' principle}.
Since gravity has a direct impact on the equation satisfied by the vertical velocity $u_2$, one focuses on it. After plugging the above expansions into it, one obtains 
\begin{align*}
(\bar \rho (y) + \tilde \rho) ( \partial_t \tilde u_2 + \tilde{\textbf{u}} \cdot \nabla \tilde u_2 ) - g \bar \rho (y) + \rho_0 \partial_y \tilde P = \nu \Delta \tilde u_2 -g \bar \rho (y) - g \tilde \rho,
\end{align*}
which yields
\begin{align*}
\partial_t \tilde u_2 + \frac{\rho_0 \partial_y \tilde P}{(\bar \rho (y) + \tilde \rho)} = \frac{\nu}{\bar \rho (y) + \tilde \rho} \Delta \tilde u_2  - g \frac{\tilde \rho}{(\bar \rho (y)+\tilde \rho)} - \tilde{\textbf{u}} \cdot \nabla \tilde u_2.
\end{align*}
Relying on the Boussinesq hypothesis, one neglects all density variations but the gravity term $g \tilde \rho$. Then we replace $\bar \rho(y) + \tilde \rho$ in the above equation by the averaged constant density $\rho_0$. Now, denoting $\tilde \nu:= \frac{\nu}{\rho_0}$, the equation satisfied by the fluctuation of the vertical velocity reads
\begin{align*}
\partial_t \tilde u_2 +\partial_y \tilde P=\tilde \nu \Delta \tilde u_2  - g \tilde \rho - \tilde{\textbf{u}} \cdot \nabla \tilde u_2.
\end{align*}
Applying the same reasoning to the density equation, one gets
\begin{align*}
(\de_t+ \tilde \uu \cdot \nabla) \tilde \rho + \tilde u_2 \de_y \bar \rho(y) = \kappa \Delta \tilde \rho. 
\end{align*}
Introducing the buoyancy variable $b:=\frac{g}{\rho_0} \rho$ and dropping the \emph{tilde} $\tilde{}$ for lightening the notation, we obtain
 \begin{equation}\label{eq:system-vel-nodamp}
\begin{aligned}
\partial_t b - N^2  u_2 &={\kappa \Delta b}-\textbf{u} \cdot \nabla b, \\
\partial_t u_1 + \partial_{x_1} P &= \nu \Delta u_1 - \textbf{u} \cdot \nabla u_1, \\
\partial_t u_2 + \partial_{x_2} P & = \nu \Delta u_2 - b  - \textbf{u} \cdot \nabla u_2, \\
\partial_{x_1} u_1 + \partial_{x_2} u_2 & = 0,
\end{aligned}
\end{equation}

where 
\begin{align}
N^2=-\dfrac{g \bar \rho'(y)}{\rho_0}
\end{align}
is the Brunt-V\"ais\"al\"a frequency, which is a strictly positive quantity (as $\bar \rho'(y)<0$, for a stable stratification) representing the maximal frequency of oscillations of \emph{internal gravity waves}, propagated by the Boussinesq system, see \cite{Rieutord, DY1999, BDSR19}. These waves are studied in detail in the discussion below, where we work under the \emph{linear stratification assumption}: we assume that the background density profile $\bar \rho(y)$ is an affine function, so that $\de_y \bar \rho (y) = {constant}$. In the ocean and the middle atmosphere, this hypothesis is realistic, see for instance \cite{DY1999, Dauxois-new}, but more general background stratifications are also of great interest, see \cite{Dauxois-new}. When the Brunt-V\"ais\"al\"a frequency $N$ is not constant and is instead a function of $y$, the wave dynamics is governed by a Sturm-Liouville problem, which is described in \cite{Dauxois-new} and studied in \cite{Lannes-Saut}. This more general framework is out of the scope of the present work.

The Boussinesq system \eqref{eq:system} under the inviscid and linear approximation (with $\nu=\kappa=0$) is a system of (internal gravity) waves.
The dispersion relation of these waves can be derived in analogy with acoustic/electromagnetic waves. More precisely, we seek for a solution to system \eqref{eq:system-vel-nodamp} with $\nu=\kappa=0$, in the form of a superposition of plane waves, i.e. $e^{-i \omega t + i \k \cdot \textbf{x}} (u, w, b, p)^T$, with $\textbf{x}=(x_1, x_2)$ the spatial coordinate and $\k=(k_1, k_2) \in \R^2$ the wavenumber. Such a plane wave is a solution of the system under a dispersion relation between the parameters, expressing the time frequency $\omega=\omega (k_1, k_2)$ as a function of the wavenumber $\k=(k_1, k_2)$. For two-dimensional internal gravity waves, it reads
\begin{align}\label{disp}
\omega^2(k_1, k_2)=\frac{k_1^2}{k_1^2+k_2^2}=\sin^2 \theta,
\end{align}
where $\theta$ is the angle between the wave direction (the group velocity) and the horizontal $x_2=0$.

The Boussinesq system \eqref{eq:system} under investigation in this paper is obtained from \eqref{eq:system-vel-nodamp}, after applying a rotation of the $(x_1, x_2)$ coordinates of an angle $\gamma$ (see \eqref{eq:coord1}) and denoting by $(x,y)$ the new spatial coordinates and by $(u, w)$ the rotated velocity field. The smallness of the fluctuations $\tilde \rho, \tilde \uu$ in \eqref{eq:system-vel-nodamp} is represented, in the rotated system \eqref{eq:system}, by the small parameter $\delta>0$, which  measures the small amplitude of waves.

\bcr{
We finally recall that, in the context of incompressible flows, the scalar pressure $p$ can be simply seen as a Lagrange multiplier assuring the divergence-free condition in \eqref{eq:system}. This implies that the scalar pressure $p$ can be always recovered from the vector field $(u, w)^T $, by using the Helmholtz-Hodge decomposition (the Leray projector), see \cite{Chemin2006}. For this reason, in the course of this paper we discard the scalar pressure $p$ from our unknown vector field, which will be denoted by $(u, w, b)^T$.}

\subsection*{The importance of working with beams}
In the \emph{near-critical reflection} of internal gravity waves, an incident wave hits a slope of angle $\gamma$ with respect to the horizontal. This happens for instance in the ocean, where internal gravity waves interact with the bottom topography, see for instance the numerical simulations in \cite{Lamb} and the laboratory experiments from the group of Dauxois \cite{Dauxois-new} and Maas \cite{Maas}. When the fluid is further confined, it leads to the formation of attractors, as found in laboratory experiments in \cite{Maas, Davis} and rigorously proved (with some further assumptions) in \cite{CVSR19, zworski1, zworski2}. 
We consider an incident internal wave hitting the sloping boundary.
Preservation of the angle of the reflection generates a focusing mechanism when the difference between the wavenumber of the incident wave and the slope's inclination is small. This is due to the very peculiar dispersion relation \eqref{disp} (see \cite{long1965} for further details), which assigns the direction of propagation of internal gravity waves rather than fixing the modulus of the wavenumber. 
Even if the size of the incident wave is small, the amplitude of the  wave \textit{reflected} (in the inviscid case where $\nu=\kappa=0$ in \eqref{eq:system}) from the slope is enhanced because of the interaction with the sloping boundary and the energy accumulates along it. This energy focusing plays a role in the quadratic nonlinear terms by radiating a second harmonic and a mean flow which are not confined to the boundary, so that the energy is transferred from the boundary layer to the outer region. This phenomenon is described in \cite{DY1999}. A rigorous mathematical study is provided in \cite{BDSR19}. In particular, \cite{BDSR19} provides a consistent and Lyapunov stable solution in $L^2(\R_+^2)$, which is made of \emph{packets of waves}: those packets are linear superpositions of plane waves, whose wavenumbers are very strongly localized around a given wavenumber $\k_0$ in the frequency space. The counterpart of the strong frequency localization, which makes the packets of waves of \cite{BDSR19} very close (at least in $L^\infty$) to single plane waves (although they have $L^2$ finite energy) is the fact that the packets of waves of \cite{BDSR19} are very spread in all directions of the physical space. In this paper, we adopt a different point of view: we want to work with \emph{beams}, which are wave-packets \emph{spatially localized} along their axis of propagation (a localization corresponding precisely to the upslope and downslope reflection introduced in \cite{DY1999}). Axisymmetric beams are widely investigated in physics, see for instance \cite{Akylas, bou} and the dedicated section in \cite{Dauxois-new}. There are several reasons for being interested in the dynamics of beams. One of them is the fact that their behavior in the context of \emph{instabilities of internal waves} is very different from the case of plane waves. As explained in \cite{Dauxois-new}, the finite width of beams reduces the time of energy absorption from unstable modes and therefore it reduces (and in some cases avoid) instabilities. This is connected to the help of spatial localization in preventing time resonances (and related time growths) that is also a key point  of the \emph{space-time resonance method} and the \emph{null structures} of \cite{Germain, Klainermann, Chr86}. The interest in the dynamics of internal gravity wave beams is one of the motivations of our work. 
In particular, we are able to provide the following explicit characterization: the Leray solution to the near-critical reflection of internal waves in $L^2(\R^2_+)$ is $L^2$ close to a (finite) sum of (spatially localized) beam waves. 
Another motivation to this work has been raised in \cite{BDSR19}, where the approximate solution to the near-critical reflection problem contains a \emph{non-physical term/corrector} that is not a \emph{consistent} solution, so preventing any further improvement of the accuracy of the approximation and/or the extension of its time of validity. Our approach will allow to overcome this problem: we will provide a \emph{fully consistent} and improved approximate (beam wave) solution in terms of \emph{accuracy}. 
This aspect will be further detailed below.

The criticality of the near-critical reflection phenomenon is measured in terms of closeness between the angle of the incident internal wave 
and the angle of the slope. \rob{We therefore define the criticality parameter as follows}
\begin{align}\label{def:critical-relation}
{\zeta:= \sin^2 \theta- \sin^2 \gamma=\omega^2 - \sin^2 \gamma.}
\end{align}
The physical literature on this problem is quite vast, we refer to the first investigation that goes back to 1999 and it is due to Dauxois \& Young \cite{DY1999}, but we point out that this is still an active direction of research in physics, see for instance \cite{kakaota}. In the setting of Dauxois \& Young, see \cite{DY1999}, the considered scaling relation is 
\begin{align}\label{eq:scaling}
|\zeta| \sim \eps^{1/3}. 
\end{align}
Recalling that the dimensionless parameter $0 < \epsilon \ll 1$ represents the size of viscosity/diffusivity ($\nu \sim \kappa \sim \epsilon$) as detailed in \eqref{eq:viscosity-size} (where physical dimensions are encoded in $\nu_0$ and $\kappa_0$, such that $\nu = \nu_0 \epsilon$ and $\kappa = \kappa_0 \epsilon$), it is noteworthy that relation \eqref{eq:scaling} establishes a connection between the criticality parameter $\zeta$ and the size of viscosity/diffusivity. According to \cite[page 278]{DY1999}, ``the main justification for these choices is \emph{a posteriori} - they work in the sense that the dissipative terms are comparable to the others in the final amplitude equation''. To elaborate, the approach in \cite{DY1999} relies on asymptotic expansions at distinct orders followed by \emph{matching conditions} of equations within the inner (near the slope) and outer (far from the boundary) regions. The amplitude of the leading-order term is determined through this matching process by the lower-order expansion. Although our linear analysis in Section \ref{linvis} adopts a different mathematical approach, inspired by \cite{thierry}, and our matching is performed at the boundary of the physical domain (the slope) where suitable boundary conditions are imposed by the system, our leading-order equations for both the inner and outer regions align with those in \cite{DY1999}. Thus, in this work, we adopt the scaling relation \eqref{eq:scaling} for the same reason.

We consider a rotated coordinate system, with rotation angle given by the slope of inclination $\gamma$. The rotated coordinates (see Figure 1 below) will be denoted by $(x,y)$. Setting the notation $\textbf{k}=(k, m)$ for the wavenumber in the rotated coordinate system $(x,y)$, the dispersion relation reads 
\be\label{defomg}
{\omega^2_{k,m}}=\omega^2(k,m)=\frac{(k \cos \gamma-m\sin \gamma)^2}{k^2+m^2}. 
\ee
We also introduce the rotated coordinates $(x^*, y^*)$, where $x^*$ is the direction along the wavenumber $\k$ (which forms an angle $\theta$ with the vertical) and $y^*$ its orthogonal (i.e. the direction given by the group velocity of the beam wave of wavenumber $\k$). We localize the beam wave in the physical space, along the direction $x^*$ of its wavevector (or wavenumber) $\k$, which is orthogonal to its \emph{group velocity}. The spatial localization is expressed in terms of high oscillations driven by a small parameter $\sigma$.  
Notice that at this stage we do not assume any relation between $\eps$ and $\sigma$.
\begin{figure}[h!]
\includegraphics[width=0.7\textwidth]{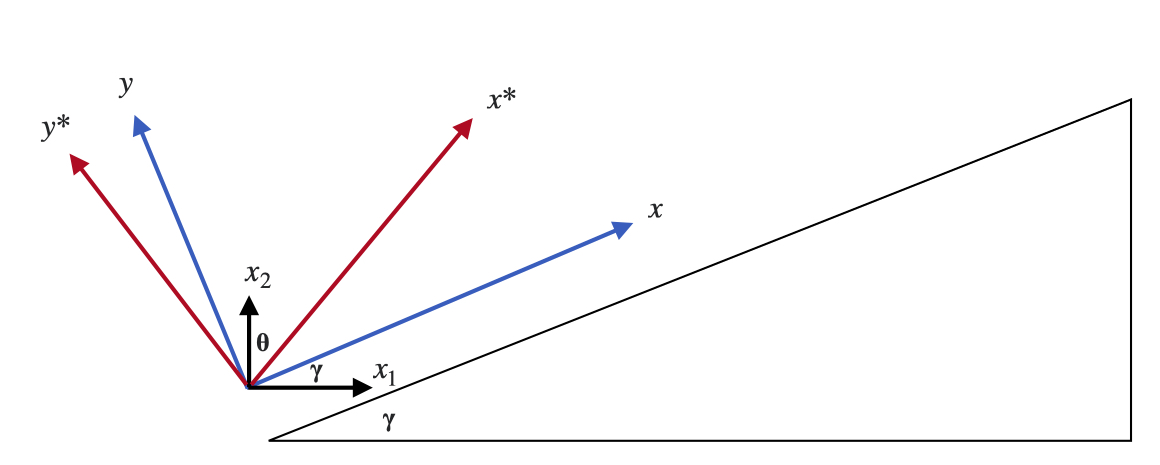}\caption{Coordinate systems}\label{coord}
\end{figure}

\rob{We} consider the following form of a quasi-axisymmetric incident beam wave
\begin{align}\label{eq:def-inc-new}
\W_{\rm{inc}}^0 & =M \int_{\mathbb{R}^2} X_{k,m}  \widehat{\psi}_{\eps, \sigma} (  k_1^*, \theta)\, e^{-i\omega t + i k x + m y}   \, dk \, dm,
\end{align}
where $\widehat{\psi}_{\eps, \sigma} ( k_1^*, \theta)$ is a compactly supported $C_0^\infty$ function that will be introduced more precisely later on, $M$ is a (strictly positive) normalizing constant which will be determined in the following, and $X_{k,m}$ is the eigenvector of the form 
\begin{align}\label{eq:eigenvector}
X_{k,m}=\begin{pmatrix}
1\\
\frac{-k}{m}\\
\frac{i(k \cos \gamma-m\sin \gamma)}{m \omega_{k,m}} \\
\frac 1 k [\omega + \sin \gamma \frac{(k \cos \gamma - m \sin \gamma)}{m \omega_{k,m}}]
\end{pmatrix}.
\end{align}
The above expression of (almost) axisymmetric beam wave has been inspired by several works in theoretical physics on internal gravity wave beams, in particular the group of Akylas, see for instance \cite{Akylas, Akylas1}. As explained in \cite{Dauxois-new}, laboratory experiments on internal wave beams are realized by means of a beam wave generator, that is for instance a cylinder oscillating at a fixed frequency in a uniform stratification.

As we report from \cite{Dauxois-new},  ``Each wave beam is a combination of plane waves over
a continuous range of wavenumbers, resulting in the energy being present only over a finite width
(along the common wave vector direction) in space. The cylinder diameter imposes a length-scale
on the finite width of the wave beams.'' Here, this length-scale is represented by the parameter $\sigma$ in \eqref{eq:def-inc-new}. Notice in particular that the superimposed plane waves inside the integral \eqref{eq:def-inc-new} have the same (fixed) direction of propagation, so that the time-frequency $\omega=\omega_0$ is also fixed, while the modulus of their wavenumber takes values in an interval that is bounded by $\sigma^{-1}$. In fact, we see in [(2.9), \cite{Dauxois-new}] that, in the physics literature, an axisymmetric beam is defined as
\begin{align*}
\psi(t, x, y, \theta)=  e^{-i \omega_0 t}\int_0^\infty \Psi ({\kk}) e^{ i kx+imy} \, d{\kk} + c.c.,
\end{align*} 
for some smooth and compactly supported function $ \Psi $. In the laboratory studies on the reflection of internal waves from a slope (and internal wave attractors) by the group of Maas \& Dauxois, the angle of inclination of the slope is slightly modified by changing the trapezoidal geometry of the domain.

The mathematical approach that we adopt here provides a rigorous description of \rob{this situation}.
Our (incident) beam wave \eqref{eq:def-inc-new} is a superposition, spread enough to provide localization in physical space, of plane waves whose direction (angle of the dispersion relation \eqref{disp}) is \emph{almost} fixed, so providing a \emph{quasi-axisymmetric beam wave}. The angle $\theta$ in \eqref{eq:def-inc-new} (and in Definition \ref{def:inc-beam} later on) models the \emph{near}-criticality of the problem. 

\subsection*{Comparison with the existing literature and presentation of our results} 
The Boussinesq equations attracted the interest of the mathematical community
thanks to their wide range of applications and to some common features with 3D incompressible fluids:
the inviscid 2D Boussinesq equations are equivalent to the incompressible axi-symmetric 3D Euler equations
\cite{Majda}; the global well-posedness of the 2D inviscid Boussinesq system is still largely open (see for instance \cite{Elgindi}), and a very active research direction focuses on the minimal amount of viscous/diffusive dissipation (see \cite{cao, Chae})  and/or damping 
(see \cite{Castro}) that guarantees long-time well-posedness of the system.
The near-critical reflection of internal waves in two dimensions is  well understood within a certain time-scale depending on the physical parameters of the system: after a theoretical
investigation due to Dauxois \& Young \cite{DY1999}, the mathematical theory of this phenomenon was established
by the first author, A.-L- Dalibard and L. Saint-Raymond in \cite{BDSR19} (see also \cite{BO21} for the linear analysis with different sizes of viscosity and diffusivity). 
The approximate solution of \cite{BDSR19} contains a non-physical corrector to replace a \emph{mean flow} that is \emph{degenerate} (see [page 219, \cite{BDSR19}), in the sense that its very slow decay does not assure in general the boundedness of its $L^2$ norm (of its energy) and therefore it cannot be implemented in the construction of the solution. The idea adopted in \cite{BDSR19} to solve this issue is to use a non-physical corrector to assure stability, but this rules out any possibility to improve the accuracy/time-scale of validity as \emph{the non-physical corrector is not a consistent approximation to system \eqref{eq:system}}.
In this paper, we overcome this problem by constructing a spatially localized beam wave approximate solution to the near-critical reflection of internal waves. In particular, our approximate solution is constituted by all  ``physical'' terms, i.e. consistent approximations to \eqref{eq:system}, 
and all the spatially localized wave-beam terms of our solution have finite $L^2$ norm.
Thus, in this paper we are able to construct a \emph{fully consistent} and $L^2$ stable approximate solution to the near-critical reflection of internal waves in the half-plane (Theorem \ref{thm:main1}). We also show that the \emph{accuracy} of our approximate solution can be further improved  (Theorem \ref{thm:main2}), by exploiting the null structure of the convection term for incompressible flows which was already pointed out in \cite{DY1999}. {Finally, a byproduct of our result is the characterization of the (unique) Leray solution to this problem within a certain (nonlinear) time-scale: we prove that the Leray solution, emanating from a suitable set of initial conditions concentrated approximately on rays, maintains the same localization  (Theorem \ref{thm:main3}).}

To conclude, the \emph{beam wave} approach is crucial in this work and it marks a strong difference with respect to \cite{BDSR19}: besides the physical relevance of beam waves (\cite{Akylas, Akylas1}), it allows to \emph{characterize the Leray solution in the physical space}, to obtain \emph{consistency} and to improve the \emph{accuracy} of the approximate solution.

\subsection*{Notation and conventions}
We use the following notation and conventions. 
\begin{itemize}
\item The coordinates $(x_1, x_2)$ denote the usual two-dimensional reference system. We will also use the following rotated coordinate systems
\begin{align}
x&=x_1 \cos \gamma + x_2 \sin \gamma; \notag\\
y&=- x_1 \sin \gamma + x_2 \cos \gamma. \label{eq:coord1} \\\notag\\
x^*&=x_1 \sin \theta + x_2 \cos \theta; \notag\\
y^*&=-x_1 \cos \theta+x_2 \sin \theta, \label{eq:coord2}
\end{align}
see Figure \ref{coord}.
Accordingly, we introduce the new variables $(u, w)$, representing a rotation of angle $\gamma$ of the original unknowns $(u_1, u_2)$, i.e.
\begin{align}
u&=u_1 \cos \gamma + u_2 \sin \gamma; \notag\\
w&=- u_1 \sin \gamma + u_2 \cos \gamma. \label{eq:variables-rotation}
\end{align}
\item Throughout this paper, we will use the convention $a=O(b), \, a \in \R, b \in \R,$ if there exists $C>0$ such that $a \le C b$. 
\item We use the notation $a \approx b, a,b \in \R$ if there exist uniform constants $c>0, C>0$ such that $a\ge c b, \, a\le C b$.
\item For any $f_\eps=f(\eps), g_\eps=g(\eps)=\eps^a\bar g $ with $a \in \R, \bar g \in \R$, we use the notation $f_\eps \sim g_\eps$ if $g_\eps$ is the leading order term of the expansion of $f_\eps$ in terms of $\eps$, i.e. $\lim\limits_{ \eps \rightarrow 0} \eps^{-a} f_\eps = \bar g$ (this corresponds to the usual definition of equivalence).
\item Given a function $f(x)$, we denote by $\widehat{f}(\xi)$ its Fourier transform, namely
\begin{align*}
\widehat{f}(\xi) = \int e^{-i\xi \cdot x} f(x)\, dx.
\end{align*}
\end{itemize}

\section{Main results}\label{sec:results}
A very general form of \emph{wave beam} is given by \eqref{eq:def-inc-new}. The assumptions on $\widehat{\psi}_{\eps, \sigma}(\kk, \theta)$ will be introduced later on. 
Note that since $\k=(k,m)$ is the rotation of angle $\gamma$ in \eqref{eq:coord1} of $(k_1, k_2)$, and $(k_1^*, k_2^*)$ is the rotation of angle $\theta$ (in Figure \ref{coord}) in \eqref{eq:coord2} of $(k_1, k_2)$, then 
\begin{align*}
\exp(ikx+imy)&= \exp( i |\textbf{k}| (x \sin (\theta+\gamma) + y \cos (\theta+\gamma)))=\exp( i{\kk}x^*),
\end{align*}
and 
\begin{align*}
k_1^*&=k_1 \sin \theta + k_2 \cos \theta= |\textbf{k}|.
\end{align*}
Then, in \eqref{eq:def-inc-new}, the function $\widehat{\psi}_{\eps, \sigma} (k_1^*, \theta)=\widehat{\Psi}_{\eps, \sigma} (\kk, \theta)$, where $\theta$ is the direction of the wavenumber $\k$. 
In the \emph{physical space}, our wave beam will be localized in the $x^*$ direction, which is  orthogonal to the wavenumber $\k$ {(a more precise description of this localization will be provided in Lemma \ref{lem:loc}).}\\

We provide below the definition of wave beam (hereafter \bw) and boundary layer wave beam (resp. \blbw). 


\begin{definition}\label{def:inc-beam}[Beams, boundary layer and incident wave ansatz]\ {In the sequel, recall the notation $\k=(k,m)$.} 
\\
\textit{For $\eta>0$, we define}
\begin{enumerate}[label=(\bf\Roman*)]
\item a {beam }  as
any function $v_{\rm{beam}}^t$ in $L^2(\R\times\R^+)$ of the form
\begin{align}\label{defbeam}
v_{\rm{beam}}^t\bcb{(x,y)} & = {\cha(y)}  \int_{\et}^\infty \int_0^{2\pi}  \widehat{\Psi}({\kk}, \theta) 
e^{- i\Omega_{k,m} t +  i k (x-x_0) + i m y}  \, {\kk} \; d{\kk} \, d\theta,
\end{align}
\item a boundary layer beam as
any function $v_{\rm{BL}}^t$ in $L^2(\R\times\R^+)$ of the form
{\begin{align}\label{defBL}
v_{\rm{BL}}^t\bcb{(x,y)} & =  {\cha(y)}\int_{\et}^\infty \int_0^{2\pi} \widehat{\Psi}({\kk}, \theta)
e^{- i\Omega_{k,m} t +  i k (x-x_0) -{\L (k) }y}  \, {\kk} \; d{\kk} \, d\theta,
\end{align}}
where
\begin{align}
&\bullet \quad \widehat{\Psi}(\kk, \theta) \, \text{is a \,} C^\infty \text{compactly supported function in \,} (\kk, \theta); \notag\\
&\bullet \quad  {\text{the time frequency\,} \Omega_{k,m}\text{\, is a function of \,} (k,m)=\k;} \notag\\
&\bullet \quad \text{Re}(\lambda(k)) >0 \quad \text{for all} \quad k, \; |\k| \ge \eta; 
\notag
\end{align}
\item a beam wave (\bw)\ as a three dimensional vector whose components are beams;\\
\item a boundary layer beam wave (\blbw)\ as a three dimensional vector whose components are boundary layer beams.\\
\\
\end{enumerate}
\noindent \textit{The two following objects are crucial in our analysis}:\\

\begin{enumerate}[label=(\bf\Roman*)]
  \setcounter{enumi}{4}
\item the incident \bw as the following \bw solving the linear part of system \eqref{eq:system} 
\begin{align}\label{equivalent}
\W_{\rm{inc}}^0:=\frac{\sigma}{\eps^{1/6}} \int_0^{2\pi} \int_\et^\infty \widehat{\Psi}_{\eps, \sigma}(\kk, \theta) X_{k,m} e^{- i\omega_{k,m} t +  i k (x-x_0) + i m y}  \, {\kk} \; d{\kk} \, d\theta,
\end{align}
where 
\begin{enumerate}
\item for any $0 < \eps, \sigma \ll1$ in the regime of Assumption \ref{ass}, we define 
 \begin{align}\label{def:psi}
\widehat{\Psi}_{\eps, \sigma}(\kk, \theta):&= \G(\sigma \kk) \left(  \G' \left(\frac{\sin \theta-\sin \gamma}{\eps^{1/3}}\right)\G' \left(\frac{\cos \theta-\cos \gamma}{\eps^{1/3}}\right)+\G' \left(\frac{\sin \theta+\sin \gamma}{\eps^{1/3}}\right)\G' \left(\frac{\cos \theta+\cos \gamma}{\eps^{1/3}}\right)\right),\\
& \text{where} \, \G=\G_{\eps, \sigma}, \G'=\G'_{\eps, \sigma} \in C_0^\infty \; \text{with compact support}, \; \text{uniformly in } \eps, \sigma; \notag
\end{align}
\item the eigenvector $X_{k,m}$ is given by \eqref{eq:eigenvector}; 
\item the time frequency $\omega=\omega_{k,m}$ is given by the dispersion relation \eqref{defomg}, i.e. 
\begin{align}
\omega=\omega_\pm=\pm\frac{k \cos \gamma-m \sin \gamma}{\sqrt{k^2+m^2}}, \label{eq:omega} 
\end{align}
and, for the beam to be \emph{incident}, since $\gamma>0$, we choose the sign of the time frequency $\omega_\pm (k,m)$ in \eqref{eq:omega} in such a way that 
$$\nabla_{k,m} \omega_\pm (k,m) \cdot (0,1)^{\rm{T}} < 0;$$
\end{enumerate}
\vspace{3mm}
\item a \emph{linear} \blbw of order $(\alpha, \beta, p, q)$ as a \blbw in (II) solving the linear part of system \eqref{eq:system}; it has the general form
\begin{align}\label{def:BL-general}
\W_{\rm{BL}, \eps^\alpha}^\L:=\frac{\sigma}{\eps^{1/6}} \int_0^{2\pi} \int_\et^\infty \widehat{\Psi}_{\eps, \sigma}^{\nn, \ta}(\kk, \theta) X_{k,\L} e^{- i\omega_{k,m} t +  i k (x-x_0)- {\L (k)}y} \, {\kk} \; d{\kk} \, d\theta, \end{align}
where 
\begin{enumerate}
\item $\alpha>0,$
\item for any $0 < \eps, \sigma \ll1$ in the regime of Assumption \ref{ass}, and for any $p, q \in \R$, 
 \begin{align}\label{def:psi-pq}
\widehat{\Psi}_{\eps,\sigma}^{\nn, \ta}({\kk}, \theta):&= a_{\nn, \ta}(\eps, \kk)\widehat{\Psi}_{\eps,\sigma}({\kk}, \theta) \quad  \text{where} \;  a_{\nn, \ta}(\eps, \kk)=O(\eps^\nn \kk^\ta) \; \text{and} \; \widehat{\Psi}_{\eps, \sigma}\; \text{is defined in \eqref{def:psi}};
\end{align}
\item there exists a function $\ell(\theta)  \in C^\infty (\text{supp} \widehat{\Psi}^{p,q}_{\eps, \sigma}, \mathbb{C});\ \Re{(\ell)}>0$ 
 {such that}   $\lim\limits_{\eps \rightarrow 0} \eps^{\alpha} \L (k)= \ell(\theta) \kk^\beta, \; \beta \in \R$;
\item the eigenvector $X_{k, \L}$ is given below in \eqref{eq:eigen-BLs};
\item the time frequency $\omega=\omega_{k,m}$ is in \eqref{eq:omega}, with the same convention.
\end{enumerate}
\end{enumerate}
\vspace{5mm}
\textit{The following objects will appear in the course of our proofs}:\\
\begin{enumerate}[label=(\bf\Roman*)]
  \setcounter{enumi}{6}
\item  {a {degenerate boundary layer beam wave} of order $(\alpha, \beta, \nn, \ta)$, any (family of) function(s) $\W_{\rm{BL}, \eps^\alpha}^\L$ given by \eqref{def:BL-general}, with $\alpha\le 0$;}\\
\item {a  {mean flow beam} of order $(\nn, \ta)$, any (family of) function(s)  \bw $\W_{\rm{MF}}$ as in (III), where the function $\widehat{\Psi}(\kk, \theta)$ localizes near $\Omega_{k,m} \sim 0$ as $\eps \rightarrow 0$};\\
\item {a {second harmonic beam wave} of order $(\nn, \ta)$, any (family of) function(s)  $\W_{\rm{II}}$ as in (III), where the function  $\widehat{\Psi}(\kk, \theta)$ localizes near $\Omega_{k,m} \sim \pm 2\omega_{k,m}$ as $\eps \rightarrow0$, and $\omega_{k,m}$ being given by \eqref{eq:omega};}
\item {a {mean flow} \blbw of order $(\alpha, \beta, \nn, \ta)$, any (family of) function(s)  $\W_{\rm{BL; MF}}$ as in (IV), where the function $\widehat{\Psi}(\kk, \theta)$ localizes near  $\Omega_{k,m} \sim 0$ as $\eps \rightarrow0$;}\\
\item {a {second harmonic} \blbw of order $(\alpha, \beta, \nn, \ta)$, any (family of) function(s) $\W_{\rm{BL; II}}$ as in (IV), where the function $\widehat{\Psi}(\kk, \theta)$ localizes near  $\Omega_{k,m} \sim \pm 2\omega_{k,m}$ as $\eps\rightarrow 0$, and $\omega_{k,m}$ being given again by \eqref{eq:omega}.}
\end{enumerate}

\end{definition}

Some remarks are in order.
\begin{remark}
We point out that, when evaluated at $y=0$, beams and boundary layers are expressed by the same formula in Definition \ref{def:inc-beam}.
\end{remark}
\begin{remark}\label{rmk:equality}
Notice that, for any fixed $\omega \in \R$, there exist four frequency vectors $\k \in \R^2$ fulfilling the relation $\omega^2=\sin^2 \theta=(\sin^2\theta)( \k)$. However, the amplitude $\widehat{\Psi}_{\eps, \sigma}$ in \eqref{def:psi} selects the two (out of four) \bw or \blbw with collinear frequency vectors. We also remark that the sign of $\omega=\omega_\pm$ in Definition \ref{def:inc-beam} point (e) is chosen in such a way that 
$$\nabla_{k,m} \omega_\pm \cdot (0,1)^T=\mp \frac{k(m \cos \gamma + k \sin \gamma)}{(k^2+m^2)^{3/2}} <0.$$
As the support of $\widehat{\Psi}_{\eps, \sigma}$ forces $\theta=\gamma+O(\eps^\frac 13)$ or $\theta=\gamma+\pi+O(\eps^\frac 13)$, we can check that in both cases the function
$$\omega_+=\frac{k\cos \gamma-m\sin \gamma}{\sqrt{k^2+m^2}}$$ satisfies the \emph{incidence condition}   $\nabla_{(k,m)} \omega_+ \cdot (0,1)^T<0$. In other words, it automatically follows from the incidence condition that \emph{$k$ and $\omega$ always have the same sign}. In fact,  $\omega=\sin \theta = \sin \gamma + O(\eps^\frac 13)>0$ (in polar coordinates) if $k>0$ and $\omega = \sin \theta = \sin (\gamma+\pi)+O(\eps^\frac 13)<0$ if $k<0$.
This remark will be important in the linear boundary layer analysis of the next section, where in some cases the sign of $k/\omega$ determines the sign of $\rm{Re}(\L(k))$ of \blbw in Definition \ref{def:inc-beam}.
\end{remark}

Let us state immediately the following lemma that estimates the norms of $v_{\rm{beam}}$ and $v_{\rm{BL}}$ and which is proven in Appendix \ref{prooflemma} below.
\begin{lemma}\label{estbbl}
Let $v_{\rm{beam}}^t$ be any  beam (\bw)  of order $(\nn,\ta)$ 
and $v_{\rm{BL}}^t$ be any boundary layer beam (\blbw) of order $(\alpha,\beta,\nn,\ta)$. 
Then, uniformly in $t\in\R_+$,
\begin{align}
\|v_{\rm{beam}}^t\|_{L^2(\R\times\R_+)}&=O(\eps^\nn\sigma^{-\ta})\nonumber;\\
\|v_{\rm{BL}}^t\|_{L^2(\R\times\R_+)}&=O(\eps^{\frac16+\nn+\frac\alpha2}\sigma^{-\ta-\frac12});\nonumber\\
\|v_{\rm{beam}}^t\|_{L^\infty(\R\times\R_+)}&=O(\eps^{\nn+\frac 16} \sigma^{-\ta-1})\nonumber;\\
\|v_{\rm{BL}}^t\|_{L^\infty(\R\times\R_+)}&=O(\eps^{\nn+\frac 16} \sigma^{-\ta-1}). \label{eq:estimatenorm}
\end{align}

Moreover, the following hold.
\begin{enumerate}[label=(\roman*)]
\item\label{xyb} $\partial_x{v}_{\rm{beam}}^t,\partial_y{v}_{\rm{beam}}^t$ are \bw of order $(\nn,\ta+1)$.
\item\label{xybl} $\partial_x{v}_{\rm{BL}}^t$ (resp. $\partial_y{v}_{\rm{BL}}^t$) is a \blbw of order $(\alpha,\beta,\nn,\ta+1)$ (resp. $(\alpha,\beta,\nn-\alpha,\ta+\beta)$).
\item Given $v_{\rm{BL}}^t, v_{\rm{BL}}'^t$ \blbw of order $(\alpha, \beta, \nn, \ta)$ and $(\alpha', \beta', \nn', \ta')$ respectively, and $v_{\rm{beam}}^t, v_{\rm{beam}}'^t$ \bw of order $(\nn, \ta)$ and $(\nn', \ta')$, one has the following estimates:
\begin{align}\label{eq:normprod}
\|v_{\rm{beam}}^t \times v_{\rm{beam}}'^t\|_{L^2(\R \times \R_+)} &= O(\eps^{\nn+\nn'+\frac 16}\sigma^{-1 - \ta- \ta'}) \notag\\
\|v_{\rm{BL}}^t \times v_{\rm{BL}}'^t\|_{L^2(\R \times \R_+)} & = O(\eps^{\nn+\nn'+\frac{\max\{\alpha, \alpha'\}}{2} + \frac 13} \sigma^{-\frac 32 - \ta - \ta'}), \\
\|v_{\rm{BL}}^t \times v_{\rm{beam}}'^t\|_{L^2(\R \times \R_+)} &= O(\eps^{\nn+\nn'+\frac \alpha 2 + \frac 13}\sigma^{-\frac 32 - \ta- \ta'}).\notag
\end{align}
\end{enumerate}
\end{lemma}

Lemma \ref{estbbl} (together with other features that will appear in the course of our proofs) leads to the following general hypothesis on the parameters present in our analysis.
\begin{assumption}[Admissible range of the involved parameters]
 \label{ass}
 First, we assume that $\gamma \in (0, \frac \pi 2)$.
 The scaling assumptions in terms of the dimensionless parameter $\eps$ are listed below.
 

\begin{itemize}
\item
{$\nu =  \eps \nu_0, \quad  \kappa = \eps \kappa_0,$} \hfill\emph{{[scaling of  viscosity/diffusion coefficients]}}
\item {$\delta =O( \sigma^{\frac 23} \eps^{\frac 12}) $ }\hfill{\emph{
[strength of  nonlinearity w.r.t. beam spatial concentration and critical parameter
]}}
\item $\sigma>\eps^{\mu} \, \text{for any}\;  0<\mu < \frac 18. $ \hfill{\emph{ [\bcr{order of viscosity/diffusion  w.r.t.  spreading in frequencies of the beam
}]}}
\end{itemize} 
First, the case $\gamma=\frac \pi 2$ corresponds to the very degenerate case of the vertical propagation of internal gravity waves, which is excluded from our analysis (and for $\gamma=0$ there is no slope).

The first scaling condition is simply the smallness of the viscosity $\nu$ and diffusivity $\kappa$ coefficients. The third one is dictated by the range of validity of the asymptotic expansions of the roots of the boundary layers in the next section, which is chosen to be in accordance with the studies of \cite{DY1999}. If they are violated, different asymptotics for the decay of the boundary layers (see the next section) should be expected. This is out of the scope of the present investigation. The condition on the weakly nonlinear parameter $\delta$ is the minimal assumption (i.e. the maximal order of $\delta$) which allows to prove that the approximate solution to system \eqref{eq:system} constructed in this work is \emph{stable} in $L^2(\R^2_+)$ (i.e. close to the Leray solutions) for a (large) logarithmic time-scale given by \eqref{eq:time}. The smaller is $\delta$, the larger is the stability time-scale. 
\end{assumption}

\begin{remark}
It is interesting to compare Definition \ref{def:inc-beam} ((V) and (VI)) with the setting of \cite{BDSR19}. In \cite{BDSR19} (see (2-14)-(2-16), pp. 223-224 in \cite{BDSR19}), the wavenumbers of the superposed waves $\k=(k, m)$ is an $\eps^{\frac 13}$ correction of the \emph{critical wavenumber} $\k_0=(k_0, m_0)$, so that its modulus (not only the angle) is almost fixed, i.e. ${\kk} \sim |\k_0|$. In this work, instead, the modulus of the wavenumber is upper bounded by $\sigma^{-1}\le \eps^{-\mu}$ (where $\sigma \ll 1$, so that $\sigma \gg 1$ as $\mu>0$), while its inclination $\theta$ is an $\eps^{\frac 13}$ correction of the \emph{critical angle} $\gamma$. Thus, in this work we deal with a more general framework where the spatial frequency of oscillations of the superimposed waves can vary largely. The choice of (almost) fixing the angle/direction of the wavenumber is natural in the context of internal gravity waves, whose \emph{anisotropic} dispersion relation is completely determined by the inclination of the wavenumber, see \cite{BDSR19, DY1999, CVSR19}.
\end{remark}

\bcr{\begin{remark}[On the existence of Leray solutions]\label{rmk:existence}
It follows from the standard theory on the two-dimensional Navier-Stokes equations in general domains (see \cite{Chemin2006}), which can be readily extended to the two-dimensional Boussinesq system \eqref{eq:system} in the half-plane, that there exists a unique global-in-time weak solution $\W(t)$ to system \eqref{eq:system} in $C(\mathbb{R}^+; \mathbb{V}_\sigma')  \cap L^\infty(\mathbb{R}^+; L^2(\mathbb{R}^2_+)) \cap L^2_{\rm{loc}}(\mathbb{R}^+; \mathbb{V}_\sigma)$,  where $\mathbb{V}_\sigma:=\{ (u, w, b) \in H^1(\mathbb{R}^2_+); \; \de_x u+\de_y w=0\}$ and $ \mathbb{V}'_\sigma \rm{\  is\ the\ dual \ of \;} \mathbb{V}_\sigma$,  with initial data $\W_0:=\W(0) \in L^2(\R^2_+)$. The proof of that, which is an adaptation of the argument in  \cite{Chemin2006}, can be found in the Appendix of \cite{BDSR19}.
\end{remark}}

The main aim of this work is to prove the following three theorems.
\begin{theorem}[Consistency \& stability]\label{thm:main1}
\bcb{Let  $\W_{\rm{inc}}^{0}$ be \rob{any} incident \bw (beam wave) \rob{defined by} \eqref{equivalent}-\eqref{eq:omega} for some $\G, \G'$.}

Assume the validity, for any $0 < \eps \ll1$, of the \bcb{Assumptions \ref{ass}} on the scaling parameters of the Boussinesq system \eqref{eq:system}.

Then, there exists an approximate solution $\W^{\rm{app}}$ to system \eqref{eq:system}
\bcb{of the form}
$$\W^{\rm{app}}=(u^{\rm{app}} , w^{\rm{app}}, b^{\rm{app}})^T:=\W_{\rm{inc}}^{0}+\W_{\rm{BL}}+\W_{\rm{II}} + \W_{\rm{MF}},$$
where  $\W_{\rm{BL}}$ is a \blbw (boundary layer beam wave), $\W_{\rm{II}}$ is a second harmonic \bw (double time frequency of the incident beam wave), $\W_{\rm{MF}}$ is a mean flow \bw (almost vanishing time frequency).
\vskip 0.3cm
More precisely, the following results hold true.\\
$(i)$  $\W^{\rm{app}}$ is a \emph{consistent} approximation to system \eqref{eq:system}, in the sense that it satisfies the following system\\
\begin{align*}
\de_t u^{\rm{app}} - b^{\rm{app}} \sin \gamma + \de_x p^{\rm{app}} -  \nu  \Delta u^{\rm{app}}  &= -\delta (u^{\rm{app}} \de_x + w^{\rm{app}}  \de_y) u^{\rm{app}}+R_u^{\rm{app}},\notag\\
\de_t w^{\rm{app}}  - b^{\rm{app}}  \cos \gamma + \de_y p^{\rm{app}}  -   \nu \Delta w^{\rm{app}}  & =  -\delta (u^{\rm{app}} \de_x + w^{\rm{app}}  \de_y) w^{\rm{app}}+R_w^{\rm{app}} ,\notag\\
\de_t b^{\rm{app}}  + u^{\rm{app}} \sin \gamma  + w^{\rm{app}}  \cos \gamma -  \kappa  \Delta b^{\rm{app}}  & =  -\delta (u^{\rm{app}} \de_x + w^{\rm{app}}  \de_y)b^{\rm{app}} + R_b^{\rm{app}} ,\notag\\
\de_x u^{\rm{app}}  + \de_y w^{\rm{app}} &=0,
\end{align*}
with a remainder 
$$R_{\rm{app}}=(R_u^{\rm{app}} , R_w^{\rm{app}} , R_b^{\rm{app}})^T = {O}(\max\{\delta {\eps^{\frac 16}}{{\sigma}^{-2}}, \delta^2 \eps^{-\frac 13} \sigma^{-\frac 52}\}) \quad \text{in} \quad L^2(\mathbb{R}^2_+).$$
In particular, when $\delta=O(\eps^\frac 12 \sigma^{\frac 23})$ (Assumptions \ref{ass}),
$$\|R_{\rm{app}}\|_{L^2(\R^2_+)} = O(\delta {\eps^{\frac 16}}{{\sigma}^{-2}})=O(\eps^\frac 23 \sigma^{-\frac{10}{3}}).$$
$(ii)$ $\W^{\rm{app}}$ is a \emph{stable} approximation to system \eqref{eq:system} in the following sense. 
\bcb{Consider the unique global-in-time Leray solution $\W(t)=\W(t,x,y)$ to system \eqref{eq:system} in $C(\mathbb{R}^+; \mathbb{V}_\sigma')  \cap L^\infty(\mathbb{R}^+; L^2(\mathbb{R}^2_+)) \cap L^2_{\rm{loc}}(\mathbb{R}^+; \mathbb{V}_\sigma)$,  where $\mathbb{V}_\sigma:=\{ (u, w, b) \in H^1(\mathbb{R}^2_+); \; \de_x u+\de_y w=0\}$ and $ \mathbb{V}'_\sigma \rm{\  is\ the\ dual \ of \;} \mathbb{V}_\sigma$,  with initial data $\W({t=0})=\W^{\rm{app}}({t=0})$.}
Then, for any $t\in \R^+$, 
\begin{align}\label{eq:stabestimate}
\|\W^{\rm{app}}(t)-\W(t)\|_{L^2(\mathbb{R}^2)}^2  \le C\delta
\eps^\frac 56 \sigma^{-\frac{10}{3}}\exp\left( \delta \eps^{-\frac 12} \sigma^{-\frac 23}t\right).
\end{align}
\end{theorem}

\begin{proof}[Proof of Theorem \ref{thm:main1}]
\bcb{$(i)$ The first part of the theorem is proven in Corollary \ref{cor:nonlin} in Section \ref{nonlin} below.}\\
\bcb{$(ii)$}
The construction of the approximate solution $\W^{\rm{app}}$ will be the main body of the paper. The existence of a unique global-in-time Leray solution $\W$, which satisfies the energy inequality 
 
 \begin{align} \label{energy-inequality}
\|\W(t)\|^2_{L^2(\R^2_+)} + 2 \eps \nu_0 \int_0^t \Bigg(\|\nabla u(s)\|^2_{L^2(\R^2_+)}+ \|\nabla w(s)\|_{L^2(\R^2_+)}^2\Bigg) \, ds\notag \\
+2\eps \kappa_0 \int_0^T \|\nabla b(s)\|^2_{L^2(\R^2_+)} \, ds
\le \|\W_0\|^2_{L^2(\R^2_+)},
\end{align} 
is classical and it follows by applying the Spectral Theorem to the inverse of the Stokes operator acting on the half plane (see Remark \ref{rmk:existence} and  \cite{BDSR19, Chemin2006} for details, we omit them). It provides a regularized solution $\W_k(t)$, for which we can derive an energy inequality (the inequality is rather for the difference $\W^{\rm{app}}-\W_k$)
\begin{align*}
\frac{d}{dt} \|\W^{\rm{app}}-\W\|_{L^2(\mathbb{R}^2)}^2  & + \langle \mathcal{L}_\eps \W^{\rm{app}}-\W, \W^{\rm{app}}-\W)(t)\rangle_{L^2(\mathbb{R}^2)} \\
&\quad + \eps {\nu_0} (\|\nabla(u^{\rm{app}}-u)\|^2_{L^2(\mathbb{R}^2)} + \|\nabla(w^{\rm{app}}-w)\|^2_{L^2(\mathbb{R}^2)}) + \eps {\kappa_0} \|\nabla(b^{\rm{app}}-b)\|_{L^2(\mathbb{R}^2)}^2 \\
& \le  \frac{2\delta}{\sqrt \eps \sigma^{2/3}} \|\W^{\rm{app}}-\W\|_{L^2(\mathbb{R}^2)}^2  + \frac{\sqrt \eps \sigma^{2/3}}{\delta} \|R^{\rm{app}}\|_{L^2(\mathbb{R}^2)}^2,
\end{align*}
which gives the desired control thanks to the skew-symmetry of $\mathcal{L}_\eps$, using $\delta \|\nabla \W^{\rm{app}}\|_{L^\infty(\mathbb{R}_+^2)} =O( {\delta}{ \eps^{-\frac 12} \sigma^{-\frac 2 3}})$ (from Proposition \ref{prop:BL} and Corollary \ref{rmk:BLsizes}), and using the estimate of  $\|R^{\rm{app}}\|_{L^2(\mathbb{R}^2)} $ from Corollary \ref{cor:nonlin}.
 \end{proof}

 \begin{remark}[Consistency \& Stability]\label{rmk:timescale}
Notice that the requirements for \emph{consistency} alone and \emph{consistency and stability} are different. In fact, consistency is ensured as long as $\| R_{\text{app}}\|_{L^2}=o(1)$, i.e. $\delta \ll \sigma^{2} \eps^{-\frac 16}$, while a more restrictive condition, i.e. $\delta=O(\eps^{\frac 12} \sigma^{\frac 23})$, is needed to have stability.
The stability estimate \eqref{eq:stabestimate} implies that,
if $\delta \sim \eps^{\frac 12} \sigma^{\frac 23}$, then 
$$\|\W^{\rm{app}}(t)-\W_{\rm{Leray}}(t)\|_{L^2}^2=O(\delta \eps^{\frac 56 - \frac{10}{3}\mu}) \exp(t),$$
so that $\|\W^{\rm{app}}(t)-\W_{\rm{Leray}}(t)\|_{L^2}=o(1)$ within the \emph{logarithmic} time-scale 
\begin{align}\label{eq:time}
T_{\rm{log}}^\delta = o(\rm{log} (\eps^{-\frac 56} \sigma^{\frac{10}{3}} \delta^{-1})) = o(\rm{log} (\eps^{-\frac 43} \sigma^{\frac 83})).
\end{align}
To have a longer stability time-scale, $\delta$ has to be smaller, namely the \emph{stability time-scale} depends on the weakly nonlinear parameter $\delta$.
 \end{remark}
 
The approximate solution $\W^{app}$ will be constructed in the following by a fixed point argument: we start from the linear solution, then we correct the quadratic interactions due to the convection term. 
We could try to improve the approximation by iteration, correcting the trilinear, quartic interactions and so on. However, the obstruction is represented by the possible appearance of \emph{secular growths} (see for instance \cite{lannes}). A classical example of secular growth is the weakly damped harmonic oscillator 
\begin{align*}
\ddot{x}+2\eps \dot x =0, \quad x(0)=0, \quad \dot x(0)=1. 
\end{align*}
The approximation at order $\eps$ reads $x^{\rm{app}}(t)=\sin(t)-\eps t \sin (t)$, so that the perturbation expansion becomes invalid starting from the time-scale $t \sim \eps^{-1}$. 

The approximate solution $\W^{\rm{app}}$ provided by Theorem \ref{thm:main1} contains \emph{correctors to the quadratic terms}. 
We shall see in the following that in our case \emph{trilinear interactions} can always be corrected as secular growths do not appear thanks to the structure of the convection term for divergence-free vector fields. Therefore, the approximate solution $\W^{\rm{app}}$ can be improved, by providing a more accurate \emph{consistent} solution $\widetilde \W^{\rm{app}}$ which contains the two next order terms of the expansion (so taking into account trilinear and quartic interactions, 
but secular growth could be generated by quintic interactions). 
This is the content of the following result, whose proof is given in Section \ref{sec:proof2}.
 \begin{theorem}\label{thm:main2} 
 Let $\W^{\rm{app}}$ be the approximate solution constructed in Theorem \ref{thm:main1}.
 For any $t>0$, there exists a corrector $\widetilde \W_{\rm{corr}}(t) \in L^2(\mathbb{R}^2_+)$, such that
 $$\widetilde \W^{\rm{app}} := \W^{\rm{app}} + \widetilde \W_{\rm{corr}}$$
 is a consistent approximate solution to system \eqref{eq:system} in the sense of Theorem \ref{thm:main1}, with a remainder $\widetilde R_{\rm{app}}(t)$ such that

$$\|\widetilde R_{\rm{app}}(t) \|_{L^2} = o( {\delta^3 \eps^{-\frac 56}}{\sigma}^{-\frac{19}{6}} t).$$
When $\delta=O(\eps^\frac 12 \sigma^\frac 23)$ (Assumption \ref{ass}), then
 $$\|\widetilde R_{\rm{app}}(t) \|_{L^2} = o(\eps^\frac 23 \sigma^{-\frac 76} t),$$
 and the consistency time-scale $T_c \gg \eps^{-\frac 23} \sigma^{\frac 76}$. 
\end{theorem}
 
 In the following result, we prove how beam waves are localized in the physical space. This marks a strong difference with respect to the solution constructed in \cite{BDSR19}.
\begin{lemma}[beam waves are localized in the physical space]\label{lem:loc}
The support of the incident beam wave (\bw) in \eqref{equivalent} is concentrated, as $\eps$ goes to 0, inside the following union of cones of the (upper) half plane
\begin{align}\label{def:Cset}
\mathcal{C}&=\mathcal C_+\cup \mathcal C_-,
\end{align}
where, for $\pp>0$ small enough,
\begin{align*}
 \mathcal C_+&= \{ (x, y) \in \R_+^2: \; -\sigma^{1-\pp} \le (x-x_0) \sin (2\gamma +\eps^{\frac 13-\pp}) + y \cos (2\gamma +\eps^{\frac 13-\pp}) \le  \sigma^{1-\pp}\}; \\
  \mathcal C_-&= \{ (x, y) \in \R_+^2: \;-\sigma^{1-\pp} \le (x-x_0) \sin (2\gamma+\pi +\eps^{\frac 13-\pp}) + y \cos (2\gamma+\pi +\eps^{\frac 13-\pp}) \le  \sigma^{1-\pp}\},
\end{align*}
in the sense that
\begin{align*}
\|\W_{\rm{inc}}^0\|_{L^2(\R^2_+ \backslash\mathcal C)} = O(\eps^\infty) \|\W_{\rm{inc}}^0\|_{L^2(\R^2_+)}. 
\end{align*} 
The support of the boundary layer beam waves (\blbw)  $\W^0_{\rm{BL}, \eps^{1/3}}$ and $\W^0_{\rm{BL}, \eps^{1/2}}$ provided by Proposition \ref{prop:BL} are localized in a horizontal strip of size $O(\eps^{\frac 13-\mu}), O(\eps^{\frac 12-\mu})$ respectively, so that 
\begin{align*}
\|\W^0_{\rm{BL}, \eps^{1/3}}\|_{L^2(\R^2_+); \, y> \eps^a; \; a \ge \frac 13 - \mu} & = O(\eps^{\frac 23 \mu}) \|\W^0_{\rm{BL}, \eps^{1/3}}\|_{L^2(\R^2_+)}, \\
\|\W^0_{\rm{BL}, \eps^{1/2}}\|_{L^2(\R^2_+); \, y> \eps^a; \; a \ge \frac 12 - \mu} & = O(\eps^{\mu}) \|\W^0_{\rm{BL}, \eps^{1/2}}\|_{L^2(\R^2_+)}.
\end{align*}
Finally, 
\begin{align*}
\|\W_{\rm{BL}, \eps^{1/3}}+\W_{\rm{BL}, \eps^{1/2}}\|_{L^2(\R^2_+ \backslash(\mathcal C \cap \{y > \eps^a\}) )} = O(\eps^\infty) \|\W_{\rm{BL}, \eps^{1/3}}+\W_{\rm{BL}, \eps^{1/2}}\|_{L^2(\R^2_+)}.
\end{align*}
\end{lemma}
\begin{proof}
From the definition of incident \bw in \eqref{equivalent} and integration in $\kk$, we deduce that there exists a function $g(\theta)$ such that 
\begin{align*}
\W_{\rm{inc}}^0(x,y)&=\frac{\sigma}{\eps^{1/6}} \int_0^{2\pi} e^{-i\omega(\theta)t} g(\theta) \widehat{\chi}\left(\frac{(x-x_0)\sin(\theta+\gamma)+y\cos(\theta+\gamma)}{\sigma}\right) \, d\theta.
\end{align*}
As ${\chi}$ is a compactly supported function, it follows that $\widehat{\chi}$ is a Schwartz function and therefore the first assertion is proved. \\
For the \blbw results, we consider $\W_{\rm{BL}, \eps^{1/2}}$ (the same argument applies to $\W_{\rm{BL}, \eps^{1/3}}$). From Proposition \ref{prop:BL}, we know that
\begin{align*}
\W_{\rm{BL}, \eps^{1/2}}(x,y)&=\frac{\sigma}{\eps^{1/6}} \int_0^\infty \int_0^{2\pi} X_{k, \L_3} \widehat{\Psi}_{\eps, \sigma} (\kk, \theta) e^{-i\omega_{k,m}t + (x-x_0) \kk \sin(\theta+\gamma) - \L_3 y} \, \kk \, d\kk \, d\theta,
\end{align*}
where $\L_3 =\eps^{-\frac 12} \tilde \ell_3(\theta) (1+O(\eps^\frac 16))$, so that the concentration of $\W_{\rm{BL}, \eps^{1/2}}$ in a vertical band of size $\eps^{\frac 12-\mu}$ easily follows.
The fact that $\W_{\rm{BL}, \eps^{1/2}}+\W_{\rm{BL}, \eps^{1/3}}$ is concentrated in $\mathcal C$ is a direct consequence of Proposition \ref{prop:BL}, from which we know that the boundary contribution $\W_{\rm{BL}, \eps^{1/3}}|_{y=0}+\W_{\rm{BL}, \eps^{1/2}}|_{y=0}=-\W_{\rm{inc}}|_{y=0}$, where $\W_{\rm{inc}}$ is localized in $\mathcal C$.
The proof is concluded.
\end{proof}

As it will be clear in the course of the proofs, the approximate solution that we build shares the same localization than the one in Lemma \ref{lem:loc}. Therefore, applying the aforementioned lemma, we get our last main result.

\begin{theorem}[Localization of the Leray solution]\label{thm:main3}
For any $\eps>0$ and $t=o(\log (\eps^{-\frac 43} \sigma^\frac 83))$, the Leray solution $\W_{\rm{Leray}}(t)$ to system \eqref{eq:system} with initial condition $\W_{\rm{Leray}}^{{in}}=\W^{\rm{app}}({t=0})$ is $L^2$-localized in $\mathcal C$ defined by \eqref{def:Cset} modulo $\eps^\infty$, i.e.
\begin{align*}
\|\W_{\rm{Leray}}(t)\|_{L^2(\R^2_+\backslash\mathcal C)} = O(\eps^\infty) \|\W_{\rm{Leray}}^{{in}}\|_{L^2(\R_+^2)}. 
\end{align*}
\end{theorem}
 
It is interesting to notice that the above cones $\mathcal C_+$ and $\mathcal{C}_-$ correspond to the places where the \emph{upslope reflection} and the \emph{downslope reflection}, introduced in \cite{DY1999}, take place.\\

 \begin{remark}[Time-scale of validity of the higher-order approximate solution $\widetilde W^{\rm{app}}$]
 We would like to provide some remarks about the improved approximate solution $\widetilde W^{\rm{app}}$, provided by Theorem \ref{thm:main2}. Its validity holds within a certain time-scale.
The time-scale of consistency alone is entirely determined by the possible appearance of \emph{secular growths}. 
More precisely, notice that by definition of \emph{boundary layers}, the leading part of the approximate solution (far from the boundary) is given by  $\W^{\rm{app}}= \W_{\rm{inc}} + \W_{\rm{II}}+ \W_{\rm{MF}}$, where $\W_{\rm{inc}}$ has time oscillations which are localized around the frequencies $\pm \omega_0$,  and for instance $\W_{\rm{II}}$ oscillates around $\pm 2\omega_0$. Thus in the nonlinear contribution, one has for instance the term $\delta Q(\W_{\rm{inc}} + \W^{\rm{II}}, \W_{\rm{inc}} + \W_{\rm{II}})$. Since the linear part (the linear oscillatory part, without dissipation) $\mathcal{L}_\eps$ of the operator of system \eqref{eq:system} in the scaling of the boundary layers contains oscillations $\pm \omega_0$, then $e^{\mathcal{L}_\eps t} Q(\W_{\rm{inc}}, \W_{\rm{II}})$ may give rise to linear time growths, which are usually called \emph{secural growths}, see \cite{DY1999}. In particular, it might produce an error $\widetilde R_{\rm{app}} =O( {\delta^2 \eps^{-\frac 13} }{\sigma}^{-\frac 52 } t)$ in $L^2$, so that consistency would hold within a time-scale of the order $T_c=o(\delta^{-2}\eps^\frac 13  \sigma^{\frac 52})$ (see the end of the proof of Corollary \ref{cor:nonlin}). 
 However, we will show in the following that no resonance (and then no secular growth) is generated by the terms $e^{\mathcal{L}_\eps t} Q(\W_{\rm{inc}}, \W_{\rm{II}}), \, e^{\mathcal{L}_\eps t} Q( \W_{\rm{II}}, \W_{\rm{inc}})$. This is due to some cancellations related to the structure of the convection term for divergence-free vector fields (which satisfies the Jacobi identity, see \cite{DY1999, Akylas}). Then, the estimated time of consistency of the improved approximate solution $\widetilde \W^{\rm{app}}$ is much larger than the above $T_c$ (if $\delta=O(\eps^\frac 12 \sigma^\frac 23)$), as stated in Theorem \ref{thm:main2}.
\end{remark}

\begin{remark}[Stability of the higher-order approximate solution $\widetilde W^{\rm{app}}$]
Theorem \ref{thm:main2} states that the higher-order approximate solution $\widetilde W^{\rm{app}}$ is \emph{consistent}. In fact, consistency is what we prove in Section \ref{ltc}. Notice however that it can be proved that \emph{Lyapunov stability} (and $L^2$ closeness to the Leray solution) remains valid for the improved approximation $\widetilde W^{\rm{app}}$, by simply applying the machinery of {Step 2} of Section \ref{sec:31} (lifting the boundary conditions) to each further corrector of the higher order approximation. We will not fully develop this last point in this paper for the sake of readibility, but we are rather confident that this argument allows to prove that  $\widetilde W^{\rm{app}}$ is still stable within the time-scale $T_{\rm{log}}^\delta$ in \eqref{eq:time} with very small (improved) remainder term $\widetilde R^{\rm{app}}$, whose $L^2$ norm is given by Theorem \ref{thm:main2}.
\end{remark}

\subsection*{Plan of the paper}
The paper is organized as follow. In Section \ref{linvis}, we provide a boundary layer analysis of the linear part of system \eqref{eq:system} and we construct an (almost) exact approximate solution to the linear problem. Section \ref{nonlin} is devoted to the  weakly nonlinear system \eqref{eq:system}, for which we construct an approximate solution that is constituted by the linear solution of Section \ref{linvis} and further terms originating from nonlinear interactions. Section \ref{nonlin} contains also the proof of Theorem \ref{thm:main1} (i) in Corollary \ref{cor:nonlin}.
Finally, in Section \ref{ltc}, we show how the approximate solution can be improved and extended at least to the next two orders of the asymptotic expansion with a very small error in $L^2$ prove Theorem \ref{thm:main2} in Section \ref{sec:proof2}.
\section{Linear analysis}\label{linvis}
In this section, we solve the \emph{linear} part of system \eqref{eq:system} (with $\delta=0$) by ansatz given by Definition \ref{def:inc-beam}. In this definition, the function $\widehat{\Psi}(\kk, \theta)$ provides different localizations in the Fourier variables or, in other words, different regimes of the parameter 
\begin{align}\label{def_zeta}
\zeta:=\omega^2-\sin^2\gamma.
\end{align}
The regime $\zeta \approx \eps^{\frac 13}$ (i.e. $\omega \sim \pm \sin \gamma$) is called \emph{near-critical}, according to \cite{DY1999, BDSR19} and will be studied in Section \ref{subsec:31}.
Other regimes where $0<c \le |\zeta|$, $c$ independent of $\eps$, will be treated in Section \ref{subsec:33}.

\subsection{Linear analysis in the near-critical regime $\zeta \approx \eps^{\frac 13}, \, \omega \sim \pm \sin \gamma $. }\label{subsec:31}
We provide the following result, where we show that one can construct \emph{linear} boundary layer beam waves (\blbw) to balance the contribution of the incident wave beam $\W_{\rm{inc}}^0$ (in Definition \ref{def:inc-beam}, (X)) on the boundary $y=0$.

\begin{proposition}[Critical reflection for a beam wave (\bw) ]\label{prop:BL}
Let $\nu_0>0, \kappa_0>0$ be such that $\nu=\nu_0 \eps, \, \kappa=\kappa_0 \eps$ for any $\eps>0$. Let Assumption \ref{ass} be fulfilled. 
There exists a function $\W^0=(u^0, w^0, b^0)^T$ and a remainder
\begin{align}
\de_t u^0 - b^0 \sin \gamma + \de_x p^0 -  \nu_0 \eps  \Delta u^0 &=r^0_u,\notag\\
\de_t w^0 - b^0 \cos \gamma + \de_y p^0 -   \nu_0 \eps \Delta w^0 & =r^0_w,\notag\\
\de_t b^0 + u^0\sin \gamma  + w^0 \cos \gamma -  \kappa_0 \eps  \Delta b^0 & = r^0_b,\notag\\
\de_x u^0 + \de_y w^0&=0,\notag\\
u^0|_{y=0}=w^0|_{y=0}=\de_y b^0|_{y=0}&=0,
\label{eq:linsystem}
\end{align}
with
$$\|(r^0_u, r^0_w, r^0_b)\|_{L^2(\R^2_+)}=O(\eps \sigma^{-2}).$$
The function $\W^0$ reads
\begin{align*}
\W^0=\W_{\rm{inc}}^0 + \W_{\rm{BL}}^0,  
\end{align*}
where $\W_{\rm{inc}}^0$ is the incident \bw and $\W_{\rm{BL}}^0$ is a \blbw, as both defined in Definition \ref{def:inc-beam}.
More precisely,
$$\W^0_{\rm{BL}}:=\W^0_{\rm{BL}, \eps^{1/3}}+\W^0_{\rm{BL}, \eps^{1/2}},$$
where 
\begin{itemize}
\item  $\W^0_{\rm{BL}, \eps^{1/3}}=(u_{\rm{BL}, \eps^{1/3}}, w_{\rm{BL}, \eps^{1/3}}, b_{\rm{BL}, \eps^{1/3}})^T$ is a \blbw of order $(\frac 13, \frac 13, -\frac 13, -\frac 23)$, as in Definition \ref{def:inc-beam}, \text{(VI)};
\item $\W^0_{\rm{BL}, \eps^{1/2}}=(u_{\rm{BL}, \eps^{1/2}}, w_{\rm{BL}, \eps^{1/2}}, b_{\rm{BL}, \eps^{1/2}})$ is a \blbw of order $(\frac 12, 0, - \frac 16, -\frac 13)$. 
\end{itemize}
\end{proposition}

{\begin{corollary}(Sizes of the components of $\W^0$)\label{rmk:BLsizes}
Let $\W^0$ be the approximate solution to the linear part of system \eqref{eq:system} with exact boundary conditions \eqref{eq:cond-BL}. The following estimates hold provided that $\mu < \frac 12$ :
\begin{align}\label{eq:sizesBL}
\|\W_{\rm{inc}}\|_{L^2(\R_+^2)} &= {O}(1); \quad  \| \W_{\rm{BL}, \eps^{ 1/3} }^0\|_{L^2(\R_+^2)} = {O}( \sigma^\frac 16); \quad \| \W_{\rm{BL}, \eps^{ 1/2} }^0\|_{L^2(\R_+^2)} = {O}(\eps^{\frac 14} \sigma^{-\frac 16}); \notag\\
\|\W_{\rm{inc}}\|_{L^\infty(\R_+^2)} &= {O}(\eps^{\frac 16} \sigma^{-1}); \quad  \| \W_{\rm{BL}, \eps^{ 1/3} }^0\|_{L^\infty(\R_+^2)} = {O}(\sigma^{-\frac 13}\eps^{-\frac 16}); \quad \| \W_{\rm{BL}, \eps^{ 1/2} }^0\|_{L^\infty(\R_+^2)} = {O}(\sigma^{-\frac 23});\notag\\
\|\nabla \W^0\|_{L^\infty(\R_+^2)} & \le \|\nabla \W_{\rm{inc}}^0\|_{L^\infty(\R_+^2)} + \|\nabla \W_{\rm{BL} }^0\|_{L^\infty(\R_+^2)}=O(\eps^{-\frac 12} \sigma^{-\frac 23}).
\end{align}
\end{corollary} 
\begin{proof}
Recalling that $X_{k,m}$ in \eqref{eq:eigenvector} and $X_{k, \L_j}$  in \eqref{eq:eigen-BLs} (for $\L_j, j=1,2,3$ provided by Lemma \ref{lem:asym1}) are uniformly bounded and since we know that $\W_{\rm{inc}}$ is a \bw of order $(0,0)$, while $\W^0_{\rm{BL}, \eps^{1/3}}, \W^0_{\rm{BL}, \eps^{1/2}}$ are \blbw of order $(\frac 13, \frac 13, -\frac 13, -\frac 23)$ and $(\frac 12, 0, -\frac 16, -\frac 13)$ respectively, then the results are a straightforward application of Lemma \ref{estbbl}.
\end{proof}}

\begin{remark}
The incident \bw $\W_{\rm{inc}}$ is chosen in such a way that it has automatically $L^2$ norm of $O(1)$. The intuition behind the finite $L^2$ norm of \blbw is the following: because of the beam localization in the physical space, a \blbw is localized in a band of size $\sigma$ of the $(x,y)$ positive half plane. The \blbw $\W_{\rm{BL}, \eps^{1/3}}$ is also localized in a band of size $\eps^{1/3}$ near $y=0$, thanks to the decay in $y$ and the lower bound $|k| \ge \eta$. Then the support  of the \blbw $\W_{\rm{BL}, \eps^{1/3}}$ is of size $(\sigma \times \eps^{1/3})$. Thus
\begin{align*}
\|\W_{\rm{BL}, \eps^{1/3}}\|_{L^2} \le \|\W_{\rm{BL}, \eps^{1/3}}\|_{L^\infty} (\sigma \times \eps^{\frac 13})^\frac 12 = O (\sigma^{\frac 16}).
\end{align*} 
\end{remark}

\begin{remark}
Notice that the solution $\W^0$ provided by Proposition \ref{prop:BL} solves the linear system \eqref{eq:linsystem} with boundary conditions \eqref{eq:cond-BL} without any error on the boundary $y=0$. However, $\W^0$ is an \emph{almost} exact linear solution as an error of size $\eps\sigma^{-2}$ in $L^2$ is due to the purely oscillatory nature of the incident \bw $\W_{\rm{inc}}^0$, which is only an \emph{approximate} solution to the linear viscous and diffusive system \eqref{eq:linsystem} ($r^0$ in Proposition \ref{prop:BL}).
\end{remark}


\subsection{Proof of Proposition \ref{prop:BL}}
The first step to construct $\W^0$ is to determine the values of $(\alpha, \beta, \nn, \ta)$ characterizing  \blbw in \eqref{def:BL-general} (as in Definition \ref{def:inc-beam}) which solve the linear system \eqref{eq:linsystem}.
Let us consider the ansatz in \eqref{def:BL-general}, where we set
\begin{align}
\R^4 \ni X_{\lambda, k}&=\begin{pmatrix}
U_\L\\
W_\L\\
B_\L\\
P_\L
\end{pmatrix}, \quad \text{with \;} U_\L=U_\L(k), \, W_\L=W_\L(k), \,B_\L=B_\L(k), \, P_\L=P_\L(k).
\end{align}
Inserting this ansatz inside system \eqref{eq:linsystem}, we obtain the corresponding (algebraic) linear system

\begin{align}\label{eq:algebra}
A_\eps(\omega, \kappa_0, \nu_0, k, \L) \begin{pmatrix}
U_\L\\
W_\L\\
B_\L\\
P_\L
\end{pmatrix}=0, \quad \text{where \;} A_\eps(\omega, \kappa_0, \nu_0, k, \L) \in \, \text{space of matrices \;} \mathcal{M}^{4\times 4}.
\end{align}
Then, we look for vectors $X_{k, \lambda} \in \ker (A_\eps(\omega, \kappa_0, \nu_0, k, \L))$, with the restriction that $\lambda=\lambda(k)$ is such that $\rm{Re}(\lambda)>0$, as in Definition \ref{def:inc-beam}. To satisfy $\ker A_\eps(\omega, \kappa_0, \nu_0, k, \L) \neq \{0\}$, we impose $\det A_\eps(\omega, \kappa_0, \nu_0, k, \L)=0$. This amounts at finding the roots in $\L$ of the following 
characteristic polynomial associated with $A_\eps(\omega, \kappa_0, \nu_0, k, \L)$:
\begin{align}\label{eq:pol}
\mathcal{P}(\L)&:=-\eps^{2} \kappa_0 \nu_0  \lambda^6 + (-i  \eps \omega (\kappa_0+\nu_0)+3\nu_0\kappa_0 \eps^{2}k^2)\L^4  + (\zeta  +2i\omega (\kappa_0+\nu_0)\eps k^2-3\nu_0\kappa_0 \eps^{2}k^4) \L^2\notag\\
&\quad  - 2i \L k \sin\gamma \cos \gamma + k^2 (\cos^2\gamma-\omega^2-i\eps\omega (\nu_0+\kappa_0) k^2+\nu_0\kappa_0 \eps^{2}k^4).
\end{align}
Furthermore, the vector $X_{k, \L_j} \in \ker (A_\eps(\omega, \kappa_0, \nu_0, k, \lambda_j)$ reads
\begin{align}\label{eq:eigen-BLs}
X_{k, \L_j}&= \begin{pmatrix}
1 \\
 \frac{ik}{\L_j} \\
\frac{\L_j \sin \gamma + i k \cos \gamma}{\L_j (i\omega - \eps \kappa_0 (k^2-\L_j^2))} \\
\frac{1}{ik} [i\omega + \eps \nu_0 (\L^2-k^2) + \sin\gamma \frac{\sin \gamma+ik\L^{-1} \cos \gamma}{i\omega+\eps \kappa_0 (\L^2-k^2)}]
\end{pmatrix}.
\end{align}

We shall rely on the following intermediate result.

\begin{lemma}[Asymptotics of the roots in the critical case $\zeta \approx \eps^\frac 13$ for $\omega \sim \pm \sin \gamma$]\label{lem:asym}
Let $\nu_0>0, \kappa_0>0$ be such that $\nu=\eps \nu_0, \kappa=\eps \kappa_0$ for any $0 < \eps \ll 1$. Let $c_0 \eps^\frac 13 < |\omega \pm \sin \gamma| \le C_0 \eps^\frac 13$ and $\zeta = \zeta_0 \eps^\frac 13$ for some universal constants $c_0>0, C_0>0$ and $\zeta_0 \in \R$. 
There exists $\eta_0=\eta_0(\kappa_0, \nu_0, \gamma)>0$ such that, for any $\eta \ge \eta_0$, the characteristic polynomial $\mathcal{P}(\L)$ in \eqref{eq:pol} associated with a \emph{linear} \blbw ansatz as in \eqref{def:BL-general}, i.e. 
the determinant of the matrix $A_\eps(\omega, \kappa_0, \nu_0, k, \L)$ of the linear algebraic system \eqref{eq:algebra}, admits exactly three roots $\L_j=\L_j(k), j=1,2,3,$ with $\rm{Re}(\L_j)>0$. Their asymptotic expansions in terms of the singular parameter $0<\eps \ll 1$ is as follows:
\begin{align*}
\L_1& = \eps^{-\frac 13}{\kk^\frac 13} \tilde \ell_1(\theta)+ O(\kk) , \quad \L_2= \eps^{-\frac 13}{\kk^\frac 13}\tilde \ell_2(\theta)+ O(\kk), \quad
 \L_3= \eps^{-\frac 12} \tilde \ell_3 (\theta)+O(\eps^{-\frac 16}), 
\end{align*}
for some functions $\tilde \ell_i(\theta)  \in C^\infty (\rm{supp} \widehat{\Psi}_{\eps, \sigma} , \mathbb{C}), \; i \in \{1,2,3\}.$\\
More precisely, if $\eta \le  \kk \le \sigma^{-1}$ with $\sigma \ge \eps^{\mu}$ and $\mu>0$ small enough (Assumption \ref{ass}), then 
\begin{align}\label{eq:eigenv}
\L_1& = \eps^{-\frac 13}{\kk^\frac 13} \tilde \ell_1(\theta) (1+O(\eps^{\frac13 - \frac 23 \mu})), \quad \L_2= \eps^{-\frac 13} \kk^\frac 13 \tilde \ell_2(\theta)(1+O(\eps^{\frac 13- \frac 23 \mu})), \quad
 \L_3= \eps^{-\frac 12} \tilde \ell_3 (\theta)(1+O(\eps^{\frac 16})).
\end{align}
Moreover, there exists an exact solution to the linear part of system \eqref{eq:system}, with boundary conditions 
\be\label{bc2}
\begin{pmatrix}
u^0_{|y=0}\\w^0_{|y=0}\\\de_y b^0_{|y=0}
\end{pmatrix}=\W_{\rm{inc}}^0|_{y=0}
\ee
where $\W_{\rm{inc}}^0$ is given in Definition \ref{def:inc-beam}, (X). The  solution is 
\begin{align}\label{eq:linear-sol}
\W_{\rm{BL}}^0=\W_{\rm{BL}, \eps^{1/3}}^{\L_1}+\W_{\rm{BL}, \eps^{1/3}}^{\L_2}+\W_{\rm{BL}, \eps^{1/2}}^{\L_3},
\end{align}
where $\W_{\rm{BL}, \eps^{\alpha}}^{\L_j}, \, j=1,2,3$ is a linear \blbw as in \eqref{def:BL-general}, where $X_{k, \L_j}$  is given by \eqref{eq:eigen-BLs}.
More precisely,
$\W_{\rm{BL}, \eps^{1/3}}^{\L_1}, \W_{\rm{BL}, \eps^{1/3}}^{\L_2}$ are \blbw of order $(\frac 13, \frac 13, -\frac 13, -\frac 23)$, while $\W_{\rm{BL}, \eps^{1/2}}^{\L_3}$ is a \blbw of order $(\frac 12, 0, -\frac 16, - \frac 13)$.
\end{lemma}
\begin{proof}
We want to apply Lemma \ref{lem:roots} to prove that there exist roots $\L$ of $\mathcal{P}(\L)$ with singular (leading order) asymptotic $\L \sim \eps^{-\alpha} \kk^\beta \ell $ where $\ell=\ell(\theta), \alpha>0, \beta \in \R$, i.e. \blbw as in $(\rm{II})$ of Definition \ref{def:inc-beam}. To this end, we first have to identify the possibile values of $\alpha>0$. 
In other words, plugging the ansatz $\L \sim \eps^{-\alpha} \kk^\beta \ell$ inside the expression of $\mathcal{P}(\L)$, one has to find the values of $\alpha>0$ for which the leading part of  $\mathcal{P}(\eps^{-\alpha} \kk^\beta \ell)$ is represented at least by two different monomials, so providing a reduced (algebraic) equation with non-trivial solutions. 
We obtain the following.
\begin{itemize}
\item Plugging the ansatz $\L \sim \eps^{-\frac 13} \kk^{1/3} \ell$ inside $\hP(\L)$, and using that $\kk \le \sigma^{-1} \le \eps^{-\mu}$ (Assumption \ref{ass}), the leading order part of the polynomial reads 
\begin{align}\label{eq:eps2BL}
\omega (\kappa_0+\nu_0) \ell^3  + 2  \sin \gamma \cos \gamma=0.
\end{align}
Since $k$ and $\omega$ have always the same sign (see Definition \ref{def:inc-beam} and Remark \ref{rmk:equality}), then the above equation admits two roots
$$\ell_{1}=\ell_{1}(\theta)=\left(\frac{2\sin \gamma \cos \gamma}{\omega (\kappa_0+\nu_0)}\right)^{1/3} \exp(i\pi/3), \quad \ell_{2}=\ell_{2}(\theta)=\left(\frac{2\sin \gamma \cos \gamma}{\omega (\kappa_0+\nu_0)}\right)^{1/3} \exp(i5\pi/3)$$
 with strictly positive real part and such that $\ell_j(\theta) \in C^{\infty}(\rm{supp} \widehat{\Psi}_{\eps, \sigma}), \, j=1,2$. Now we want to prove that there exist two roots $\L_j,  \, j \in \{1,2\}$ of the original polynomial $\hP(\L)$ which are $r$-close to $\eps^{-1/3}\kk^{1/3}  \ell_{j}$ with $\ell_{j}$ solutions to \eqref{eq:eps2BL} for any $0 < r < \eps^{-1/3}\kk^{1/3}$. 
We verify that the assumptions of Lemma \ref{lem:roots} are fulfilled: for $\ell_{j}$ roots of \eqref{eq:eps2BL}, one has for $\eps>0$ sufficiently small that \\
\begin{align*}
&\bullet \quad \frac{|\zeta_0|}{2} \frac{\kk^{1/3}}{\eps^{1/3}}\le |\hP(\eps^{-1/3} {\kk}^{1/3} \ell_{j}) | \le 2 |\zeta_0| \frac{\kk^{1/3}}{\eps^{1/3}}; \\
& \bullet \quad |\ell_{j}|  \kk |\sin \gamma \cos \gamma|\le |\hP'(\eps^{-1/3} {\kk}^{1/3} \ell_{j})| \le 4 |\ell_{j}|  \kk |\sin \gamma \cos \gamma|;\\
&\bullet \quad |\hP''(z)| \le 2 |\zeta_0| \eps^{1/3} \; \text{for} \; |z| \le 2r \; \text{where} \;  r \ge \left|\frac{\hP(\eps^{-1/3} {\kk}^{1/3} \ell_{j})}{\hP'(\eps^{-1/3} {\kk}^{1/3} \ell_{j})}\right| \gtrsim \eps^{-1/3} \kk^{-2/3}.
\end{align*}
Therefore, we have that
\begin{align*}
& |\hP(\eps^{-1/3} {\kk}^{1/3} \ell_{j})| \le 2 |\zeta_0| \frac{\kk^{1/3}}{\eps^{1/3}} \le \frac{|\ell_{j}|^2 \kk^2 |\sin \gamma \cos \gamma|^2}{8|\zeta_0| \eps^{1/3}} \le    \frac{|\hP'(\eps^{-1/3} {\kk}^{1/3} \ell_{j})|^2}{4 \sup_{z \le 2 r} |\hP''(z)|}\\
&\quad \Leftrightarrow \quad 16  |\zeta_0|^2 \le |\ell_{j}|^2 \kk^{5/3} |\sin \gamma \cos \gamma|^2 \quad \Leftrightarrow \quad  \kk \ge \left( \sup_{j \in \{1,2\}} \frac{16|\zeta_0|^2}{|\ell_{j}|^2 |\sin\gamma \cos \gamma|^2}  \right)^{3/5} \\
&\quad \Leftrightarrow \quad \kk \ge 	\left( \frac{16 |\zeta_0|^2 (\omega (\kappa_0+\nu_0))^{2/3}}{2^{2/3}|\sin \gamma \cos \gamma|^{5/3}}\right)^{3/5} :=\eta_0,
\end{align*}

where $\eta_0>0$ is in Definition \ref{def:inc-beam}. This way we obtain that the roots $\L_j, \, j\in\{1,2\}$ of the polynomial $\hP(\L)$ admit the following expansion
$$\L_j \sim \eps^{-1/3} \kk^{1/3} \ell_{j}+ \eps^{-1/3} \kk^{-2/3} \ell_{j}', \quad \text{where} \;  \ell_{j} \; \text{solves} \; \eqref{eq:eps2BL}\; \text{and for some} \; \ell_{j}'=\ell_{j}'(\theta).$$

Applying Lemma \ref{lem:roots} once again, we obtain that
$$\L_j \sim \eps^{-1/3} \kk^{1/3} \ell_{j}+ \eps^{-1/3} \kk^{-2/3} \ell_{j}'+ \kk \ell_{j}'',  \; \text{for some} \; \ell_{j}''=\ell_{j}''(\theta),$$

where we recall that $\eta \le \kk \le \eps^{-\mu}$.

\item Plugging the ansatz $ \L_3\sim \eps^{-1/2} \ell $ inside $\hP(\L)$, we obtain that the leading order terms satisfy 
\begin{align}\label{eq:eps3BL}
\kappa_0 \nu_0  \ell^2 + i\omega (\kappa_0+\nu_0) =0.
\end{align}
We denote by 
$$\ell_{3}=\sqrt{\frac{\omega (\kappa_0+\nu_0)}{\kappa_0 \nu_0}} \exp (i7\pi/4)$$ 
the root with strictly positive real part. Once again, we want to apply Lemma \ref{lem:roots} to prove that there exists a root $\L_3 \sim \eps^{-1/2}$ of $\hP(\L)$. In this case we have
\begin{align*}
&\bullet \quad \hP( \eps^{- 1/2} \ell_{3})\sim \eps^{-2/3}, \quad \hP'( \eps^{- 1/2} \ell_{3})\sim \eps^{-1/2} \; \Rightarrow \; r=O(\eps^{-1/6})\\
&\bullet \quad \sup \hP''(z) = O(1), \quad   |z| \le C \eps^{-1/6}. 
\end{align*}
We check then that the assumptions of Lemma \ref{lem:roots} are fulfilled with $r \sim \eps^{-1/6}$, and we obtain that
$$\L_3 \sim \eps^{-1/2} \ell_{3} + \eps^{-1/6} \ell_{3}' \; \text{for some} \; \ell_{3}'=\ell_{3}(\theta) \; \text{uniformly bounded}.$$
\end{itemize}

\noindent It remains to construct an exact solution to the linear part of system \eqref{eq:system} with boundary conditions \eqref{bc2}. With $\L_j$ and $X_{k, \L_j}$ in \eqref{eq:eigen-BLs} at hand for $j=1,2,3$, let us construct a $\W_{\rm{BL}, \eps^{\alpha_j}}^{\L_j}$ as in \eqref{def:BL-general}.
By the above analysis, it is automatically known that $\alpha_1=\alpha_2=\frac 13$, $\alpha_3=\frac 12$ and $\beta_1= \beta_2= \frac 13, \, \beta_3=0$. We need to find the indexes $(p_j, q_j)$, i.e. to determine the function $a_{p_j, q_j}(\eps, \kk)$ for each $j=1,2,3$. 
Given the eigenvector $X_{k, \L_j}$ as in \eqref{eq:eigen-BLs}, let us denote $(U_{\L_j}, W_{\L_j}, B_{\L_j})$ the first three components of the four dimensional vector $X_{k, \L_j}$ respectively. In order to determine $a_{p_j, q_j}=a_{p_j, q_j}(\eps, \kk)$ for each $j=1,2,3$, we have to solve the following linear algebraic system
\begin{align}
a_{p_1, q_1} U_{\L_1} + a_{p_2, q_2} U_{\L_2}+a_{p_3, q_3} U_{\L_3}&=\mathfrak{u}, \notag \\
a_{p_1, q_1} W_{\L_1} + a_{p_2, q_2} W_{\L_2}+a_{p_3, q_3} W_{\L_3}&=\mathfrak{w},\notag\\
-a_{p_1, q_1}  \L_1 B_{\L_1}  - a_{p_2, q_2} \L_2  B_{\L_2}+a_{p_3, q_3} \L_3 B_{\L_3}&=\mathfrak{b},\notag
\end{align}
where we used the notation
$$\W_{\rm{inc}}^0|_{y=0}
=\int_{\R^2} 
\begin{pmatrix}
\mathfrak u\\ \mathfrak w\\ \mathfrak b
\end{pmatrix}
\widehat{\Psi}_{\eps, \sigma} (\kk, \theta) e^{i(kx-\omega t)} \, dk \, dm.$$

Using the precise structure of the eigenvector \eqref{eq:eigen-BLs}, one deduces that this amounts at inverting the matrix
\begin{align*}
M&=\begin{pmatrix}
1 & 1 & 1\\
 \frac{ik}{\L_1}  & \frac{ik}{\L_2} & \frac{ik}{\L_3}\\
-\frac{\L_1 \sin \gamma + i k \cos \gamma}{ i\omega - \eps \kappa_0 (k^2-\L_1^2)} & - \frac{\L_2 \sin \gamma + i k \cos \gamma}{ i\omega - \eps \kappa_0 (k^2-\L_2^2)}  & -\frac{\L_3 \sin \gamma + i k \cos \gamma}{ i\omega - \eps \kappa_0 (k^2-\L_3^2)}. \\
\end{pmatrix}
\end{align*}
From \eqref{eq:eigenv}, we know that $\L_j=\L_j(k) \sim \eps^{\eps^\alpha_j} \ell_j (\theta) \kk^{\beta_j}$ with $\ell_j(\theta)  \kk^{\beta_j} \neq \ell_i(\theta)  \kk^{\beta_i}$ for all $\eps>0$ if $i\neq j$. Then,
 the eigenvectors $X_{k, \L_j},\, j=1,2,3$ are linearly independent,
 $$\det M=\sum_{i,j=1, \, i \neq j}^3 O\left(k \frac{\L_i}{\L_j}\right),$$
 and the above matrix $A$ is invertible. Moreover, the leading order entries of the inverse matrix read
 \begin{align*}
 M^{-1}&\sim (\det M)^{-1}  \begin{pmatrix}
 ik (\frac{\L_2}{\L_3}-\frac{\L_3}{\L_2}) & \L_3-\L_2 & ik (\frac{1}{\L_2}-\frac{1}{\L_3})\\
 ik(\frac{\L_1}{\L_3}-\frac{\L_3}{\L_1}) & \L_1-\L_3 & ik (\frac{1}{\L_1}-\frac{1}{\L_3})\\
 ik(\frac{\L_1}{\L_2}-\frac{\L_2}{\L_1}) & \L_2-\L_1 & ik  (\frac{1}{\L_2}-\frac{1}{\L_1}).
 \end{pmatrix}.
 \end{align*}
 Since $(\det M)^{-1} \sim \left(\sum_{i=1,2} \frac{k \L_{3}}{\L_i}\right)^{-1} \sim \eps^\frac{1}{6} k^{-\frac{2}{3}}$, then one obtains the leading order terms of $a_{p_j, q_j}\, j=1,2,3$, i.e.
 \begin{align*}
 a_{p_1, q_1} & \sim \eps^\frac{1}{6} |k|^{-\frac{2}{3}} \times (\L_3-\L_2) \sim \eps^{-\frac{1}{3}} |k|^{-\frac{2}{3}},\notag \\
  a_{p_2, q_2} & \sim \eps^\frac{1}{6} |k|^{-\frac{2}{3}} \times (\L_1-\L_3) \sim \eps^{-\frac{1}{3}} |k|^{-\frac{2}{3}},\notag\\
   a_{p_3, q_3} & \sim \eps^\frac{1}{6} |k|^{-\frac{2}{3}} \times (\L_2-\L_1) \sim \eps^{-\frac{1}{6}} |k|^{-\frac{1}{3}}. 
 \end{align*}
 
Recalling that $(\kk, \frac{\pi}{2}-(\theta+\gamma))$ are the polar coordinates of $\k=(k,m)$, i.e. $k=\kk \sin (\theta+\gamma)$, then thanks to the angular localization due to $\widehat{\Psi}_{\eps, \sigma}^{p_j, q_j}$ in Definition \ref{def:inc-beam}, we have
 \begin{align}
|a_{p_1, q_1}| & \sim  \eps^{-\frac{1}{3}} \kk^{-\frac{2}{3}}, \quad 
|a_{p_2, q_2}|   \sim \eps^{-\frac{1}{3}} \kk^{-\frac{2}{3}}, \quad
|a_{p_3, q_3}|  \sim \eps^{-\frac{1}{6}} \kk^{-\frac{1}{3}}. \label{eq:amplitudeBL}
 \end{align}

This concludes the proof the lemma.
\end{proof}

\begin{proof}[End of the proof of Proposition \ref{prop:BL}]
Let $\Wbl^0:=\W_{\rm{sol}, \omega_0}^{\rm{crit}}=\W_{\rm{BL}, \eps^{1/3}}^{\L_1}+\W_{\rm{BL}, \eps^{1/3}}^{\L_2}+\W_{\rm{BL}, \eps^{1/2}}^{\L_3}$, as constructed in Lemma \ref{lem:asym1}. Then the solution to \eqref{eq:linsystem} with boundary conditions \eqref{eq:cond-BL} is given by $\W^0:=\W_{\rm{inc}}^0+\Wbl^0$, where the incident \bw is in Definition \ref{def:inc-beam}. The only error of the approximation is due to viscosity and diffusivity acting on the incident \bw and, from Lemma \ref{estbbl}, we have
 $$ \|(r^0_u, r^0_w, r^0_b) \|_{L^2} = O(\eps \|\Delta \W_{\rm{inc}}\|_{L^2}) =O(\eps \sigma^{-2}).$$

The proof of the proposition is concluded.
\end{proof}


\subsection{Linear analysis in different regimes}\label{subsec:33}
This section is dedicated to the construction of \bw and \blbw solving the linear system \eqref{eq:linsystem} for different asymptotics of $\omega$ and $\zeta$.
The linear analysis in the regimes $\omega \approx 0$ and $\omega \approx 1$ in the non-critical case $\zeta \approx 1$, which we develop below, will be needed.

\begin{lemma}[Asymptotics of the roots in the non-critical case $\zeta \approx 1$ for $|\omega| \ll 1$]\label{lem:asym1}
Let $\nu_0>0, \kappa_0>0$ be such that $\nu=\eps \nu_0, \kappa=\eps \kappa_0$ for any $0 < \eps \ll 1$. Let $ \omega \approx \eps^\frac 13$. Consider a \blbw ansatz  as in Definition \ref{def:inc-beam} for some constant value $\eta=\eta(\kappa_0, \nu_0, \gamma)>0$ solving the linear part of system \eqref{eq:system}.
The characteristic polynomial $\mathcal{P}(\L)$ in \eqref{eq:pol}, i.e. the determinant of the matrix $A_\eps(\omega, \kappa_0, \nu_0, k, \L)$ associated with the linear algebraic system \eqref{eq:algebra}, admits exactly three roots $\L_j=\L_j(k), j=1,2,3,$ with $\rm{Re}(\L_j)>0$.  Their asymptotic expansions in terms of the singular parameter $0<\eps \ll 1$ read as follows:
\begin{align*}
\L_1 &= \frac{ik}{\zeta}(\sin \gamma \cos \gamma + \omega \sqrt{1-\omega^2}) + \eps \ell_1''(\theta) k^3   +O(\eps^2 k^5); \\
\L_j&= \eps^{-\frac 12} \ell_j(\theta) + O(\eps^{-\frac16})\;  \quad \text{for} \; j=2,3,
\end{align*}
for some functions $\ell_j(\theta) \in C^\infty (\rm{supp}(\widehat{\Psi}_{\eps, \sigma}), \mathbb{C}),\; j=2,3,$ $\ell_1''(\theta) \in C^\infty (\rm{supp}(\widehat{\Psi}_{\eps, \sigma}), \mathbb{R})$, s.t. $ \rm{Re}(\ell_1''(\theta))>0$,
and where $\zeta$ is the criticality parameter in \eqref{def:critical-relation}.
\end{lemma}
\begin{proof}
We apply again Lemma \ref{lem:roots} following the method of the proof of Lemma \ref{lem:asym}.\\
$\bullet$ Plugging the ansatz $\tilde \L \sim \eps^{-\frac 13} \ell$ inside $\mathcal{P}(\L)$, and using that $\kk \le \sigma^{-1}\le \eps^{-\mu}$, the leading order part of the polynomial reads
\begin{align*}
\nu_0\kappa_0 \ell^4 +\sin^2\gamma=0.
\end{align*}
The solutions with positive real part are given by
\begin{align*}
\ell_{2}=(\sin^2\gamma/(\nu_0 \kappa_0))^{1/4} \exp(i\pi/4), \quad \ell_{3}=(\sin^2\gamma/(\nu_0 \kappa_0))^{1/4} \exp(i7\pi/4). 
\end{align*}
To apply Lemma \ref{lem:roots}, it is enough to notice that for $i=1,2,$
\begin{align*}
|\hP(\eps^{-1/2} \ell_{i}) | \le 2 \eps^{-2/3} |\omega_0| (\kappa_0+\nu_0) |\ell_{i}|^4, 
\end{align*}
while 
\begin{align*}
\frac{|\hP(\eps^{-1/2} \ell_{i})|}{|\hP'(\eps^{-1/2} \ell_{i})|} \le C \eps^{-1/6} \quad \text{for some universal constant} \, C>0. 
\end{align*}
Moreover,
\begin{align*}
\frac{|\hP'(\eps^{-1/2} \ell_{i})|^2}{8 \sup_{|z | \le C \eps^{-1/6}} |\hP''(z)|} \ge C' \eps^{-1}\quad \text{for some universal constant} \, C'>0.
\end{align*}
We check therefore that the assumptions of Lemma \ref{lem:roots} are fulfilled for every $\eps \le \eps_0$ small enough. \\
$\bullet$ Next, we look for an eigenvalue of the form $\tilde \L\sim k \ell$. We find that the leading order equation is
\begin{align*}
\zeta \ell^2 - 2 i \ell\sin \gamma \cos \gamma + \cos^2\gamma-\omega^2=0,
\end{align*}
whose solutions are given by
\begin{align*}
\tilde \L_{1}^\pm= \frac{ik}{\zeta^2}(\sin \gamma \cos \gamma + \omega \sqrt{1-\omega^2}) =-{ik}\ell_1\pm i \eps^{1/3} k \ell_{1}'(\theta) \quad \text{with}\quad \ell_1=-i \cot \gamma;  \; \ell_1'(\theta)=-\frac{i\omega_0(\theta)}{\sin^2\gamma}+o(1),
\end{align*}
(in this context $\omega=\eps^\frac 13 \omega_0$).
We need the next order expansion of $\tilde \L_{1}^\pm$, 
\begin{align*}
\tilde \L_1'^{\pm}-\tilde \L_1^{\pm} & = - \frac{\hP(k \ell_{1})}{\hP'(k \ell_{1})} \sim  \frac{i \eps \omega (\kappa_0+\nu_0) k^4 (\ell_{1})^4 - 2i\omega (\nu_0+\kappa_0) \eps k^4 (\ell_{1})^2 + i \eps \omega (\kappa_0+\nu_0) k^4}{-2ik \sin \gamma \cos \gamma - 2 k \sin^2\gamma  \ell_{1}+ 2\eps^{2/3} \omega_0^2 k \ell_{1} \pm 2i \eps^{1/3}k \sin^2\gamma \ell_1' }=\frac{N}{D}.
\end{align*}
As $\ell_{1}=- \frac{i}{\tan \gamma}$, then $N>0$ for $\eps$ small enough, while the first two addends of $D$ vanish. Choosing $\tilde \L_1^-=-\frac{ik}{\tan \gamma} - i \eps^{1/3} \ell_1'$, one has
\begin{align*}
\tilde \L_1'^-&\sim \frac{i \eps \omega (\kappa_0+\nu_0) k^3 ((\tan \gamma)^{-4} +  (\tan \gamma)^{-2}+1)}{2i \eps^{1/3}\sin^2\gamma \ell_1'+ 2i\eps^{2/3} \omega_0^2 \tan \gamma^{-1} } =  \eps k^3 \ell_1'',
\end{align*}
for some function  $\ell_1''= \ell_1''(\theta)$ such that $\text{Re}( \ell_1'')>0$ (we recall that $\omega=\eps^{1/3} \omega_0$.)
The next order expansion is then 
\begin{align*}
\tilde \L_1+\tilde \L_1'^-= \frac{i k}{\tan\gamma} - {i\eps^{1/3}k\ell_1'} + \eps k^3  \ell_1'', \qquad \text{Re}(\ell_1'')>0.
\end{align*}
Now, we want to prove that there exists a root $\L_1$ of the original polynomial $\hP(\L)$ which is $r$-close to $\tilde \L_1+\tilde \L_1'^-$ applying Lemma \ref{lem:roots}. We have the following
\begin{align*}
|\hP(\tilde \L_1+\tilde \L_1'^-)| &\le \frac{12\nu_0 \kappa_0}{\tan^2\gamma} \eps^2 k^6; \\
\frac{|\hP(\tilde \L_1+\tilde \L_1'^-)|}{|\hP'(\tilde \L_1+\tilde \L_1'^-)|} &\le C(\nu_0, \kappa_0, \omega_0, \gamma) \eps^2 |k|^5; \\
\frac{|\hP'(\tilde \L_1+\tilde \L_1'^-)|^2}{|\hP''(z)|_{|z| \le C \eps^2 |k|^5}} &\ge \frac{k^4}{2\tan^2\gamma}.
\end{align*}
We obtain therefore that the assumptions of Lemma \ref{lem:roots} are always fulfilled for $\eps\le \eps_0$ small enough provided that $\eps^2 k^2 \ll 1$, i.e. $\eps^{2-2\mu} \ll 1$, namely $\mu < 1$. The proof is concluded.
\end{proof}

\begin{lemma}[Asymptotics of the roots in the non-critical case $\zeta \approx 1$ for bounded $\omega$]\label{lem:asym3}
Let $\nu_0>0, \kappa_0>0$ be such that $\nu=\eps \nu_0, \kappa=\eps \kappa_0$ for any $0 < \eps \ll 1$. Let $ \omega \approx 1$ and $\zeta \approx 1$. Consider a \blbw ansatz  as in Definition \ref{def:inc-beam} for some constant value $\eta=\eta(\kappa_0, \nu_0, \gamma)>0$ solving the linear part of system \eqref{eq:system}.
The characteristic polynomial $\mathcal{P}(\L)$ in \eqref{eq:pol}, i.e. the determinant of the matrix $A_\eps(\omega, \kappa_0, \nu_0, k, \L)$ associated with the linear algebraic system \eqref{eq:algebra}, admits exactly three roots $\L_j=\L_j(k), j=1,2,3,$ with $\rm{Re}(\L_j)>0$. Their asymptotic expansions in terms of the singular parameter $0<\eps \ll 1$ read as follows:
\begin{align*}
\L_1 = \frac{ik}{\zeta} (\sin \gamma \cos \gamma +  \omega \sqrt{1-\omega^2})+O(\eps |k|^3) ; \quad \L_j= \eps^{-\frac 12} \ell_j(\theta)+O(|k|), \; j=2,3,
\end{align*}
for some functions $\ell_j(\theta) \in C^\infty (\rm{supp}(\widehat{\Psi}_{\eps, \sigma}) , \mathbb{C}), \; j=2,3$.
\end{lemma}
\begin{proof}
We apply again the method of Lemma \ref{lem:asym}. \\
$\bullet$ We look for a root o13f the type $\tilde \L \sim k \ell$. Since in this case $\omega \sim 1$, then the leading order terms read
\begin{align*}
\zeta \ell^2 - 2i \ell \sin \gamma \cos \gamma +\cos^2\gamma-\omega^2=0,
\end{align*}
whose solutions are given by
$$\ell_{\pm} =  \frac{i}{\zeta} (\sin \gamma \cos \gamma \pm  \omega \sqrt{1-\omega^2}).$$
We want to apply Lemma \ref{lem:roots}. Since $|\hP(k \ell_\pm)|=O(\eps k^4)$ and  $|\hP (k \ell_{\pm})/\hP'(k\ell_\pm)| = (\eps |k|^3)$, we have that
$$\frac{|\hP'(k \ell_\pm)|^2}{\sup_{|z| \le C \eps |k|^2} |\hP''(z)|} \ge C' k^2 \ge C'' \eps k^4 \ge |\hP(k \ell_\pm)|,$$
for all $\eps \le \eps_0$ small enough provided that $\eps \kk^2 \le \eps^{1-2\mu} \ll 1$, namely $\mu < \frac 12$. Therefore, the assumptions of Lemma \ref{lem:roots} are fulfilled and we deduce that there exists a root $\L_1$ of the original polynomial $\hP(\L)$ such that
$$\L_1 =   \frac{ik}{\zeta} (\sin \gamma \cos \gamma +  \omega \sqrt{1-\omega^2})+O(\eps |k|^3), \quad \text{Re} (\L_1)>0, \quad \text{Re}(\L_1) = O(\eps |k|^3). $$
$\bullet$ Plugging the ansatz $\tilde \L \sim \eps^{-1/2} \ell$, applying the same procedure as in the first part of the proof of the previous lemma (where roots of asymptotics $\sim \eps^{-1/2}$ are also investigated), we obtain similalar expansions satisfying the assumptions of Lemma \ref{lem:roots}. We omit the details and refer to the first part of the proof of the previous lemma. The proof is therefore concluded.
\end{proof}

\section{The weakly nonlinear system}\label{nonlin} 
Since system \eqref{eq:system} is \emph{weakly nonlinear} as the nonlinear term is weakened by the small parameter $\delta>0$, it is reasonable to look at the solution $\W^0$ to the linear system \eqref{eq:linsystem} in Proposition \ref{prop:BL} as an approximate solution to the weakly nonlinear system.

In the following, let us denote by $\mathbb{P}$ the standard Leray projector in $L^2(\R^2_+, \R^3)$, and we introduce 
\begin{align}\label{def:L}
\mathcal{L}_\eps \begin{pmatrix}
u \\
w\\
b
\end{pmatrix} = \mathbb{P} \begin{pmatrix}
-\nu_0 \eps \Delta & 0 & -  \sin \gamma\\
0 & -\nu_0 \eps  \Delta & - \cos \gamma\\
\sin \gamma & \cos \gamma & -\kappa_0 \eps  \Delta 
\end{pmatrix}\begin{pmatrix}
u \\
w\\
b
\end{pmatrix}.
\end{align}

For the unknowns  $\W=(u, w, b)^T, \W'=(u', w', b')^T$, we also introduce the following bilinear form 
\begin{align}\label{eq:quadratic}
Q(\W, \W')=\mathbb{P} (u \de_x + w\de_y) \W'. 
\end{align}

With this notation, the nonlinear system \eqref{eq:system} rewrites as
\begin{align}\label{eq:compactsystem}
\de_t \W + \mathcal{L}_\eps \W = -\delta Q(\W, \W).
\end{align}

Plugging the linear solution $\W^0$ inside the nonlinear system in compact form, it generates an error due to the quadratic nonlinearity 
\begin{align}\label{eq:compact}
\mathcal{E}^0=-\delta Q(\W^0, \W^0).
\end{align}
This section is devoted to the construction an approximate solution to \eqref{eq:compactsystem} correcting the error $\mathcal{E}^0$ generated by the nonlinear term. Note that despite the presence of $0<\delta \ll 1$ in front of the nonlinear term, $\mathcal{E}^0$ contains non-negligible error terms because of the boundary layers of the linear solution  $\W^0$ provided by Proposition \ref{prop:BL}.

\begin{proposition}\label{prop:nonlin}
Let $\W^0$ provided by Proposition \ref{prop:BL} be the \bcb{solution to the linear system \eqref{eq:linsystem} with boundary conditions \eqref{bc2}}. There exists a function $\W^1=(u^1, w^1, b^1)^T$ and a remainder $r^1$ such that
\begin{align*}\de_t \W^1 + \mathcal{L}_\eps \W^1 = -\delta Q(\W^0, \W^0)+r^1, \\
u^1|_{y=0} = w^1|_{y=0}=\de_yb^1|_{y=0}=0,
\end{align*}
with
$$\|r^1\|_{L^2(\mathbb{R}^2)} =O(\delta \eps^{\frac 16} \sigma^{-2}) =O(\delta \eps^{\frac 16-2\mu}).$$
The function $\W^1$ reads
$$
\W^1=\W_{\rm{II}}^1 + \W_{\rm{MF}}^1 + \W^1_{\rm{BL}, \eps^{1/3}} + \W^1_{\rm{BL}, \eps^{1/2}}
 $$
 where
 \begin{itemize}
 \item $\W_{\rm{II}}^1$ ($\rm{II}$ for second harmonic) is  associated to the time frequency $\omega_{\rm{II}}=\pm 2 \sin \gamma + O(\eps^\frac 13);$
 \item $\W_{\rm{MF}}^1$ (MF for mean flow) is associated to the frequency $\omega_{\rm{MF}}=O(\eps^\frac 13)$
 \item $\W^1_{\rm{BL}, \eps^{1/2}}$ (resp. $ \W^1_{\rm{BL}, \eps^{1/3}}$) is the product of a \blbw\ of order $(\frac 12, 0, -\frac 23, - \frac 13)$ and a \blbw of order $(\frac 12, 0, 0, \frac 23)$ (resp. $(\frac 13 ,\frac 13,0,0)$ and $(\frac 13, \frac 13, - \frac 23, -\frac 13)$). 
 \end{itemize}
The sizes of the boundary layer components read
$$\|\W_{\rm{BL}, \eps^{1/3} }^1\|_{L^2} = O(\delta \eps^{-\frac 16} \sigma^{-\frac 76}), \quad \|\W^1_{\rm{BL}, \eps^{1/2}}\|_{L^2(\mathbb{R}^2)}=O(\delta \eps^{-\frac{1}{12}} \sigma^{-\frac{11}{6}}),$$
while, for the mean flow and the second harmonic, we have
$$ \|\W_{\rm{MF}}^1\|_{L^2} \sim \|\W_{\rm{II}}^1\|_{L^2} = O(\delta \eps^{-\frac 16} \sigma^{-\frac{10}{3}}). $$
\end{proposition}

\begin{corollary} [Proof of (i) of Theorem  \ref{thm:main1}]\label{cor:nonlin}
The function $\W^{\rm{app}}=(u^{\rm{app}}, w^{\rm{app}}, b^{\rm{app}})^T:=\W^0+\W^1$, where $\W^0, \W^1$ are provided by Proposition \ref{prop:BL} and Proposition \ref{prop:nonlin} respectively, is
an approximate solution to the system \eqref{eq:system}, in the following sense:
\begin{align*}
\de_t \W^{\rm{app}}+ \mathcal{L}_\eps \W^{\rm{app}} &= \delta Q(\W^{\rm{app}}, \W^{\rm{app}})+ R^{\rm{app}}, \\
u^{\rm{app}}|_{y=0} & = w^{\rm{app}}|_{y=0}=\de_y b^{\rm{app}}|_{y=0}=0,
\end{align*}
where $\|R^{\rm{app}}\|_{L^2}=O(\max\{\delta \eps^{\frac 16} \sigma^{-2}, \delta^2 \eps^{-\frac 13} \sigma^{-\frac 52}\})$. If $\delta=O(\eps^\frac 12 \sigma^\frac 23)$ (Assumption \ref{ass}), then $\|R^{\rm{app}}\|_{L^2}=O(\delta \eps^{\frac 16} \sigma^{-2}).$
\end{corollary}
\begin{proof}
The approximate solution $\W^{\rm{app}}$ is given by
$$\W^{\rm{app}}=\W^0+\W^{1}.$$
Plugging now $\W^{\rm{app}}$ inside system \eqref{eq:compactsystem}-\eqref{def:L}, we have immediately from Proposition \ref{prop:BL} and Proposition \ref{prop:nonlin} that 
\begin{align*}
\de_t \W^{\rm{app}}+\mathcal{L}_\eps \W^{\rm{app}} &= \delta Q(\W^0, \W^0) +r^0+ r^1\\
&= \delta Q( \W^{\rm{app}},  \W^{\rm{app}}) - \delta (Q(\W^0, \W^1)+Q(\W^1, \W^0)+Q(\W^1, \W^1))+r^0+ r^1\\
&= \delta Q( \W^{\rm{app}},  \W^{\rm{app}}) + R^{\rm{app}},
\end{align*}
where 

\begin{align}
R_{\rm{app}}:= r^0+r^1 - \delta (Q(\W^0, \W^1)+Q(\W^1, \W^0)+Q(\W^1, \W^1)),
\end{align}
with $r^0, r^1$ in Proposition \ref{prop:BL} and Proposition \ref{prop:nonlin} respectively. It can be checked that the term with the largest $L^2$ norm among $\delta Q(\W^0, \W^1), \delta Q(\W^1, \W^0), \delta Q(\W^1, \W^1)$ is proportional to 
$$\delta^2 w_{\rm{BL}, \eps^{1/3}} w_{\rm{BL}, \eps^{1/3}} \de_{yy} u_{\rm{BL}, \eps^{1/3}}, $$
where  $\W^0_{\rm{BL}, \eps^{1/3}}=(u_{\rm{BL}, \eps^{1/3}}, w_{\rm{BL}, \eps^{1/3}}, b_{\rm{BL}, \eps^{1/3}})^T$ is provided by Proposition \ref{prop:BL}. We have that
\begin{align*}
 \delta^2 \|w_{\rm{BL}, \eps^{1/3}} w_{\rm{BL}, \eps^{1/3}} \de_{yy} u_{\rm{BL}, \eps^{1/3}}\|_{L^2} \lesssim \delta^2  \|w_{\rm{BL}, \eps^{1/3}}\|_{L^\infty} \|w_{\rm{BL}, \eps^{1/3}}\de_{yy} u_{\rm{BL}, \eps^{1/3}}\|_{L^2}.
\end{align*}
Noticing that $w_{\rm{BL}, \eps^{1/3}}$ is a \blbw of order $(\frac 13, \frac 13, 0, 0)$ and $\de_{yy} u_{\rm{BL}, \eps^{1/3}}$ is a \blbw of order $(\frac 13, \frac 13, -1, 0)$, using Lemma \ref{estbbl} we have 
$$ \delta^2  \|w_{\rm{BL}, \eps^{1/3}}\|_{L^\infty} \|w_{\rm{BL}, \eps^{1/3}}\de_{yy} u_{\rm{BL}, \eps^{1/3}}\|_{L^2} =O(\delta^2 \times  \eps^\frac 16 \sigma^{-1}\times  \eps^{-\frac 23+\frac 16} \sigma^{-\frac 32}) = O(\delta^2 \eps^{-\frac 13} \sigma^{-\frac 52}).$$
Notice that as $\delta=O(\eps^\frac 12 \sigma^\frac 23)$ (Assumptions \ref{ass}), then $O(\delta^2 \eps^{-\frac 13} \sigma^{-\frac 52})=O(\delta \eps^\frac 16 \sigma^{-\frac 32})$, and this implies that the leading order term is $r^1$, so that 
$$\|R_{\rm{app}}\|_{L^2(\R^2_+)} = O(\|r^1\|_{L^2(\R^2_+)})=O(\delta \eps^\frac 16 \sigma^{-2}).$$
The proof is concluded.
\end{proof}

The rest of this section is dedicated to the proof of Proposition \ref{prop:nonlin}. First, the sizes of the error terms are provided by the following result.
\begin{lemma}
The quadratic terms to be corrected in \eqref{eq:compact} are described by the following Table \ref{table1}.
\begin{table}[h!]
\caption{List of quadratic interactions}\label{table1}
\centering

\begin{tabular}
{|l|l|l|}
\hline  
\scriptsize  {\bf type of interaction} &\scriptsize {\bf size in $L^2$ } &\scriptsize {\bf typical decay rate  }\\  \hline

 \scriptsize $\rm{(a1)}=Q(\W_{\rm{BL}, \eps^{1/3}}^0, \W_{\rm{BL}, \eps^{1/3}}^0)$ &\scriptsize $O(\sigma^{-7/6} \eps^{-1/6})$ &\scriptsize {$\eps^{-1/3}$ }\\  \hline

  \scriptsize$\rm{(a2)}=Q(\W_{\rm{inc}}, \W_{\rm{BL}, \eps^{1/3}}^0)$ &\scriptsize $O(\sigma^{-7/6} \eps^{-1/6} )$ &\scriptsize {$\eps^{-1/3}$ }\\  \hline

 \scriptsize $\rm{(b1)}=Q(\W_{\rm{BL}, \eps^{1/3}}^0, \W_{BL, \eps^{1/2}}^0)$ &\scriptsize $O( \sigma^{-7/6}\eps^{-1/12} )$ &\scriptsize {$\eps^{-1/2}$ }\\  \hline

 \scriptsize$\rm{(b2)}=Q(\W_{\rm{inc}}, \W_{\rm{BL}, \eps^{1/2}}^0)$ &\scriptsize $O(\sigma^{-7/6} \eps^{-1/12} ) $&\scriptsize {$\eps^{-1/2}$ }\\  \hline

 \scriptsize $\rm{(c1)}=Q(\W_{\rm{inc}}, \W_{\rm{inc}})$ &\scriptsize $O(\sigma^{-2} \eps^{1/6} )$ &\scriptsize {no decay }\\  \hline

 \scriptsize $\rm{(c2)}=Q(\W_{\rm{BL}, \eps^{1/3}}^0, \W_{\rm{inc}}) $ &\scriptsize  {$O( \sigma^{-11/6} \eps^{1/6})$} &\scriptsize {$\eps^{-1/3}$ }\\ \hline

 \scriptsize $\rm{(d1)}=Q(\W_{\rm{BL}, \eps^{1/2}}^0, \W_{BL, \eps^{1/3}}^0)$ &\scriptsize  {$O(\sigma^{-11/6} \eps^{1/4})$} &\scriptsize {$\eps^{-1/2}$ }\\  \hline

 \scriptsize $\rm{(d2)}=Q(\W_{\rm{BL}, \eps^{1/2}}^0, \W_{BL, \eps^{1/2}}^0)$ &\scriptsize $O(\sigma^{-11/6}\eps^{1/4})$ &\scriptsize {$\eps^{-1/2}$ }\\  \hline
 
 \scriptsize $\rm{(d3)}=Q(\W_{\rm{BL}, \eps^{1/2}}^0, \W_{\rm{inc}} )$ &\scriptsize $O(\sigma^{-13/6}\eps^{5/12}) $&\scriptsize {$\eps^{-1/2}$ }\\  \hline

\end{tabular}

\label{Tab:interactions}
\end{table}
\end{lemma}
\begin{proof}
We systematically apply Lemma \ref{estbbl}.
Since $Q(\cdot, \cdot)$ is a (non-symmetric) bilinear operator and $\W^0=\W^0_{\rm{inc}}+\W_{\rm{BL}, \eps^{1/3}}^0+\W_{\rm{BL}, \eps^{1/2}}^0$ is composed of three terms, $Q(\W^0, \W^0)$ is a sum of nine terms. Their $L^2$ norm can be easily computed by applying formula \eqref{eq:normprod} of Lemma \ref{estbbl}. Recalling that the unknown vector is $\W=(u, v, b)^T$, from the expression of the quadratic term $Q(\W, \W')= \mathbb{P}(u \de_x + v \de_y) \W'$ we notice that, if $\W'$ is a b.l.w.b., one obtains that for a universal constant $C>0$
$$\|Q(\W_{\rm{BL}}, \W_{\rm{BL}} )\|_{L^2} \le C \|w_{\rm{BL}} \times \de_y u'_{\rm{BL}} \|_{L^2},$$
thanks to the structure of the Leray projector.
On the other hand, when $\W, \W'$ are \bw, then one can use 
$\|Q(\W, \W )\|_{L^2} \le C \|u \times \de_x u' \|_{L^2}$ or $\|Q(\W, \W )\|_{L^2} \le C \|w \times \de_y u' \|_{L^2}.$
In both cases, we apply Lemma \ref{estbbl}.
\begin{itemize}
\item (a1) $Q(\W_{\rm{BL}, \eps^{1/3}}^0, \W_{\rm{BL}, \eps^{1/3}}^0):$ from \eqref{eq:eigenv}-\eqref{eq:amplitudeBL}, $ w_{\rm{BL}, \eps^{1/3}}$ is a b.l.w.b. of order $(\frac 13, \frac 13, 0, 0)$, and $\de_y u_{\text{BL}, \eps^{1/3}}$ is a b.l.w.b. of order $(\frac 13, \frac 13, -\frac 23, -\frac 13)$. From estimate \eqref{eq:normprod} of Lemma \ref{estbbl}, we immediately obtain $$\| Q(\W_{\rm{BL}, \eps^{1/3}}^0, \W_{\rm{BL}, \eps^{1/3}}^0)\|_{L^2} \le C \|w_{\rm{BL}, \eps^{1/3}} \times \de_y u_{\rm{BL}, \eps^{1/3}} \|_{L^2} = O(\sigma^{-\frac 76} \eps^{-\frac 16}).$$
\item (a2) $Q(\W_{\rm{inc}}, \W_{\rm{BL}, \eps^{1/3}}^0)$: we see that $w_{\rm{inc}}$ is a \bw of order $(0, 0)$ and $\de_y u_{\rm{BL}, \eps^{1/3}}$ is a \blbw of order $(\frac 13, \frac 13, -\frac 23, -\frac 13)$. Then we have that
$$\| Q(\W_{\rm{inc}}^0, \W_{\rm{BL}, \eps^{1/3}}^0)\|_{L^2} \le C \|w_{\rm{inc}} \times \de_y u_{\rm{BL}, \eps^{1/3}} \|_{L^2} = O(\sigma^{-\frac 76} \eps^{-\frac16}).$$
\item (b1) $Q(\W_{\rm{BL}, \eps^{1/3}}^0, \W_{\rm{BL}, \eps^{1/2}}^0)$: first, from \eqref{eq:eigenv}-\eqref{eq:amplitudeBL}, $w_{\rm{BL}, \eps^{1/3}}$ is a \blbw of order $(\frac 13, \frac 13, 0, 0)$, next $\de_y u_{\rm{BL}, \eps^{1/2}}$ is a \blbw of order $(\frac 12, 0, -\frac 23, -\frac 13)$, so that
$$\|Q(\W_{\rm{BL}, \eps^{1/3}}^0, \W_{\rm{BL}, \eps^{1/2}}^0)\|_{L^2} \le C \|w_{\rm{BL}, \eps^{1/3}} \times \de_y u_{\rm{BL}, \eps^{1/2}} \|_{L^2} = O(\eps^{-\frac{1}{12}} \sigma^{-\frac 76}).$$
\item (b2) $Q(\W_{\rm{inc}}^0, \W_{\rm{BL}, \eps^{1/2}}^0)$: $w_{\rm{inc}}$ is a \bw of order $(0, 0)$ and $\de_y u_{\rm{BL}, \eps^{1/2}}$ is a \blbw of order $(\frac 12, 0, -\frac 23, -\frac 13)$, so that
$$\|Q(\W_{\rm{inc}}^0, \W_{\rm{BL}, \eps^{1/2}}^0)\|_{L^2} \le C \|w_{\rm{inc}} \times \de_y u_{\rm{BL}, \eps^{1/3}} \|_{L^2} = O(\eps^{-\frac{1}{12}} \sigma^{-\frac 76}).$$
\item (c1) $Q(\W_{\rm{inc}}^0, \W_{\rm{inc}}^0)$: $u_{\rm{inc}}$ is a \bw of order $(0,0)$, next $\de_x u_{\rm{inc}}$ is a \bw of order $(0, 1)$, so that
$$\|Q(\W_{\rm{inc}}^0, \W_{\rm{inc}}^0)\|_{L^2} \le C \|u_{\rm{inc}}\times  \de_x u_{\rm{inc}} \|_{L^2} = O(\eps^{\frac 16} \sigma^{-2}).$$
\item (c2) $Q(\W_{\rm{BL}, \eps^{1/3}}^0, \W_{\rm{inc}}^0)$: $u_{\rm{BL}, \eps^{1/3}}$ is a \blbw of order $(\frac 13, \frac 13, -\frac 13, -\frac 23)$ and $\de_x u_{\rm{inc}}$ is a \bw of order $(0, 1)$ 
$$\|Q(\W_{\rm{BL}, \eps^{1/3}}^0, \W_{\rm{inc}}^0)\|_{L^2} \le C \|u_{\rm{BL}, \eps^{1/3}} \times \de_x u_{\rm{inc}} \|_{L^2} = O(\eps^{\frac 16} \sigma^{-\frac{11}{6}}).$$
\item (d1) $Q(\W_{\rm{BL}, \eps^{1/2}}^0, \W_{\rm{BL}, \eps^{1/3}}^0)$: first, from \eqref{eq:eigenv}-\eqref{eq:amplitudeBL}, $w_{\rm{BL}, \eps^{1/2}}$ is a \blbw of order $(\frac 12, 0, \frac 13, \frac 23)$, next $\de_y u_{\rm{BL}, \eps^{1/3}}$ is a \blbw of order $(\frac 13, \frac 13, -\frac 23, -\frac 13)$, so that
$$\|Q(\W_{\rm{BL}, \eps^{1/2}}^0, \W_{\rm{BL}, \eps^{1/3}}^0)\|_{L^2} \le C \|w_{\rm{BL}, \eps^{1/2}} \times \de_y u_{\rm{BL}, \eps^{1/3}} \|_{L^2} = O(\eps^{\frac 14} \sigma^{-\frac{11}{6}}).$$
\item (d2) $Q(\W_{\rm{BL}, \eps^{1/2}}^0, \W_{\rm{BL}, \eps^{1/2}}^0)$: $w_{\rm{BL}, \eps^{1/2}}$ is a b.l.w.b. of order $(\frac 12, 0, \frac 13,  \frac 23)$, next $\de_y u_{\rm{BL}, \eps^{1/2}}$ is a \blbw of order $(\frac 12, 0, -\frac 23, - \frac 13)$, so that
$$\|Q(\W_{\rm{BL}, \eps^{1/2}}^0, \W_{\rm{BL}, \eps^{1/2}}^0)\|_{L^2} \le C \|w_{\rm{BL}, \eps^{1/2}} \times \de_y u_{\rm{BL}, \eps^{1/2}} \|_{L^2} = O(\eps^{\frac 14} \sigma^{-\frac{11}{6}}).$$
\item (d3) $Q(\W_{\rm{BL}, \eps^{1/2}}^0, \W_{\rm{inc}}^0)$: $u_{\rm{BL}, \eps^{1/2}}$ is a \blbw of order $(\frac 12, 0, -\frac 16, - \frac 13)$, next $\de_x u_{\rm{inc}}$ is a \blbw of order $(0, 1)$, so that
$$\|Q(\W_{\rm{BL}, \eps^{1/2}}^0, \W_{\rm{inc}}^0)\|_{L^2} \le C \|u_{\rm{BL}, \eps^{1/2}} \times \de_x u_{\rm{inc}} \|_{L^2} = O(\eps^{\frac{5}{12}} \sigma^{-\frac{13}{6}}).$$
\end{itemize}
\end{proof}

\begin{remark}[On the size of the quadratic term and the role of spatial localization]
The {product} of two \bw (or two \blbw or a \bw and a \blbw) does not preserve the localization in frequency space: the angle associated with the vector $(k+k', m+m')$ can assume any value even though the angles of $(k,m)$ and $(k',m')$ are well localized. 
However, the localization in the physical space can provide a better (i.e. smaller) $L^2/L^\infty$ norm with respect to the case of packets of waves (this happens for instance for $\W_{\rm{BL}, \eps^{1/3}}$ in Corollary \ref{rmk:BLsizes}). \\
In contrast to \cite{BDSR19}, thanks to the spatial localization, every \emph{physical corrector} of our approximate solution will have a bounded $L^2$ norm and there will be no artificial corrector, so that the approximate solution that we build is \emph{fully consistent}.

\end{remark}
We shall prove Proposition \ref{prop:nonlin} interaction type by interaction type.
\subsection{Interactions of type $\rm{(a)}$}\label{sec:31}
There are two terms of type $\rm{(a)}$ interaction, as given by Table \ref{table1}: \\$\text{(a1)}=Q(\W_{\rm{BL}, \eps^{1/3}}^0, \W_{\rm{BL}, \eps^{1/3}}^0)$ and $\text{(a2)}=Q(\W_{\rm{inc}}, \W_{\rm{BL}, \eps^{1/3}}^0)$.
\begin{proposition}\label{prop:inta}
There exists a corrector $\W_{\rm{(a)}}^1=(u_{\rm{(a)}}^1, w_{\rm{(a)}}^1, b_{\rm{(a)}}^1)^T$ which solves
\begin{align}
\de_t \W_{\rm{(a)}}^1 + \mathcal{L}_\eps \W_{\rm{(a)}}^1 = - \delta (\rm{(a1)+(a2)}) + r_{\rm{(a)}}^1, \label{eq:firsteqa}\\
u_{\rm{(a)}}^1|_{y=0}=w_{\rm{(a)}}^1|_{y=0}=\de_y b_{\rm{(a)}}^1|_{y=0}=0, \label{eq:secondeqa}
\end{align}
with $\|r_{\rm{(a)}}^1\|_{L^2} =O(\delta \eps^{\frac 16}\sigma^{-\frac{11}{6}})=O(\delta \eps^{\frac 16 -\frac{11}{6}\mu})$.
The corrector is composed of
$$\W_{\rm(a)}^1=\W_{\rm{(a)}; \rm{II}}^1 + \W_{\rm{(a)}; \rm{MF}}^1 + \W^1_{\rm{(a)}\; \rm{BL}, \eps^{1/3}} + \W^1_{\rm{(a)}; \rm{BL}, \eps^{1/2}},$$
where for the boundary layer part, we have
$$\|\W_{\rm{(a)}; \rm{BL}, \eps^{1/3} }^1\|_{L^2} = O(\delta \eps^{-\frac 16} \sigma^{-\frac 76}), \quad \|\W^1_{\rm{(a)}; \rm{BL}, \eps^{1/2}}\|_{L^2(\mathbb{R}^2)}=O(\delta \eps^{-\frac{1}{12}} \sigma^{-\frac{11}{6}}),$$
while for the mean flow \bw $\rm{(MF)}$ and the second harmonic \bw $\rm{(II)}$ it holds
$$ \|\W_{\rm{(a)}; \rm{MF}}^1\|_{L^2} \sim \|\W_{\rm{(a)}; \rm{II}}^1\|_{L^2} =O(\delta \eps^{-\frac 16} \sigma^{-\frac{10}{3}}).$$
Moreover, if $\sin \gamma > \frac 12$ then the second harmonic $\W_{\rm{(a)}; \rm{II}}^1$ is evanescent, i.e. it has a decay of order $O(1)$ in $y$. 
\end{proposition}

The rest of Section \ref{sec:31} is devoted to the proof of Proposition \ref{prop:inta}, which is divided in two steps. We develop the two steps below. 

\paragraph{\textbf{Step 1 : solving equation \eqref{eq:firsteqa}.}}
We consider $\rm{(a1)}=Q(\W_{\rm{BL}, \eps^{ 1/3}}^{0}, \W_{\rm{BL}, \eps^{1/3}}^{0}).$ Being very similar, the treatment of $\rm{(a2)}$ will be omitted. 

We want to find a corrector $\V^{\rm{1, nonlin}}_{\rm{(a1)}}$ solving the (a1) part of equation \eqref{eq:firsteqa} (in this step we ignore the boundary conditions, which will be considered in the next step), i.e. such that
$$\de_t \V^{\rm{1, nonlin}}_{\rm{(a1)}} + \mathcal{L}_\eps \V^{\rm{1, nonlin}}_{\rm{(a1)}} = - \delta \rm{(a1)}+ r_{\rm{(a)}}^1.$$
The remainder $r_{\rm{(a)}}^1$ is generated by the approximation of the Leray projector: we will see below that, in the scaling introduced by the boundary layers, the 3x3 system \eqref{eq:compactsystem}-\eqref{def:L} can be reduced to a 2x2 system by introducing the error $r_{\rm{(a)}}^1$. 

\begin{itemize}
\item[\rm{(i)}]: \underline{correcting the error for the components $U,B$.}\\
We write the expression of the convection term 
$$\text{(a1)}=Q(\W_{\text{BL}, \eps^{ 1/3}}^{0}, \W_{\text{BL}, \eps^{1/3}}^{0})=\sum_{i,j=1}^2 \mathbb{P}(u_{\text{BL}, \eps^{1/3}}^{0, \L_i} \de_x + w_{\text{BL}, \eps^{1/3}}^{0, \L_i} \de_y) \W_{\text{BL}, \eps^{1/3}}^{0, \L_j},$$
where $\L_i, \L_j$ with $i,j \in \{1,2\}$ are given by Lemma \ref{lem:asym}.
For any $i,j \in \{1,2\}$, we use the notation
\begin{align*}
\W_{\text{BL}, \eps^{1/3}}^{0, \L_i}&= (u_{\text{BL}, \eps^{1/3}}^{0, \L_i} , w_{\text{BL}, \eps^{1/3}}^{0, \L_i} , b_{\text{BL}, \eps^{1/3}}^{0, \L_i})^T=(U, W, B)^T,\\
\W_{\text{BL}, \eps^{1/3}}^{0, \L_j}&= (u_{\text{BL}, \eps^{1/3}}^{0, \L_j} , w_{\text{BL}, \eps^{1/3}}^{0, \L_j} , b_{\text{BL}, \eps^{1/3}}^{0, \L_j})^T=(U', W', B')^T.
\end{align*}
This yields 
\begin{align}\label{eq:quadratic}
(u_{\rm{BL}, \eps^{1/3}}^{0, \L_i} \de_x + w_{\rm{BL}, \eps^{1/3}}^{0, \L_i} \de_y) \W_{\rm{BL}, \eps^{1/3}}^{0, \L_j}& =  {\sigma}{\eps^{-\frac 13}} \int a_i  a_i' \widehat{\Psi}_{\eps,\sigma}  \widehat{\Psi}_{\eps,\sigma}'  e^{-i(\omega + \omega ') t + i(k+k') (x-x_0)} \notag\\
&  \times e^{-(\L_i+\L_i') y} \left(i k' U - \L_i' W\right) \begin{pmatrix}
U'\\
W'\\
B'
\end{pmatrix}\, dk \, dk' \, dm \, dm',
\end{align}
where $a_i, a_i, \widehat{\Psi}_{\eps, \sigma},  \widehat{\Psi}'_{\eps, \sigma}$ are deduced from the proof of Lemma \ref{lem:asym}.

Notice that the Leray projector, which is a homogeneous pseudodifferential operator of degree 0, does not change the structure of the convection term, but since $\rm{Re}(\L_i) \sim \eps^{-\frac 13} {\kk}^\frac 13$ for $i=1,2$ and $\rm{Re}(\L_3) \sim \eps^{-1/2}$ (from Lemma \ref{lem:asym}), it introduces a (spatial) scaling in terms of powers of the small parameter $\eps$ (see also  [\cite{BDSR19}, Section 3]).

We explain this fact in the following for the boundary layers of decay $\L_i \sim \eps^{-\frac 13} {\kk}^\frac 13$, $i=1,2$ (in Lemma \ref{lem:asym}). 
It is known from classical references such as \cite{Chemin2006} that the two-dimensional Leray projector is the 0-th order pseudodiffe\-rential operator of symbol of $\text{Id}-\nabla \Delta^{-1} \nabla \cdot$. In order to include the variable $b$, in the present case the Leray projector will be the operator associated with the symbol of 
$$\mathbb{P}=\begin{pmatrix} 
\text{Id}_{\R^2}-\nabla \Delta^{-1} \nabla \cdot &\bold{0}\\
\textbf{0}^T & 1
\end{pmatrix}, \quad \bold{0}=\begin{pmatrix}
0\\
0
\end{pmatrix}.
$$
%
Then, for $\lambda \sim \eps^{-\frac 13}{\kk^{\frac 13}}, k\sim \pm \kk \sin \gamma $ (Lemma \ref{lem:asym}) and  $\eta \le {\kk} \le \eps^{-\mu}$ (Assumption \ref{ass}), we have
\begin{align*}
(\text{Id}-\mathbb{P}) \left(e^{ik x- \L y} \begin{pmatrix}
U'\\
W'\\
B' 
\end{pmatrix}\right) = 
{e^{ik x- \L y} 
\left(\begin{pmatrix}
0\\
W'\\
0
\end{pmatrix}+{O}(\eps^{\frac 13}\sigma^{-\frac 23\mu})\right),}
\end{align*}
yielding
\begin{align}\label{eq:lerayproj}
\mathbb{P} \left(e^{i k x- \L y} 
\begin{pmatrix}
U'\\
W'\\
B' 
\end{pmatrix}\right)=  
{e^{ik x- \L y} 
\left(
\begin{pmatrix}
U'\\
0\\
B'
\end{pmatrix}+{O}(\eps^{\frac 13}\sigma^{-\frac 23\mu})\left|  \begin{pmatrix}
U'\\
0\\
B'
\end{pmatrix}\right|
\right).}
\end{align} 

Therefore, the approximation of the Leray projector introduces an error of the order $\delta \times \eps^{\frac 13- \frac 23 \mu}$ ($\delta$ is due to the quadratic term), which will be included in $r_{\text{(a)}}^1$ in Proposition \ref{prop:inta}. 
This implies that the main order system corresponding to these boundary layers only involves two of the three variables, i.e. $u$ and $b$, while the vertical velocity $w$ will be constructed by exploiting the divergence-free condition.
Since the component $W$ is expected to be $\eps^{\frac 13}$ smaller than $U, B$ in this regime, to handle system \eqref{eq:system} in a  perturbative way, we first reduce the system \eqref{eq:system} to the equations for $u, b$ ($(U, B)$ denote their Fourier transform). In other words, we consider the first and third equation where $w=0$.

For $\mathcal{V}=(U, B)^T$, the equations in compact formulation read
\begin{align}\label{eq:compactL}
\de_t \mathcal{V} + L  \mathcal{V} &= \delta \exp (ilx-i\alpha t - \L y) \mathcal{V}', 
\end{align}
where 
\begin{align}\label{eq:matrix-L}
L=\sin \gamma \begin{pmatrix}
0 & -1 \\
1 & 0
\end{pmatrix}.
\end{align}
We simply have that $\exp(Lt)=\exp(it \sin \gamma ) \Pi_+  + \exp(-it \sin \gamma ) \Pi_-$, where 

$$\Pi_+=\frac{1}{2}\begin{pmatrix}
1 & -i \\
i & 1
\end{pmatrix}, \qquad  \Pi_-=\overline{\Pi}_+ .$$

The solution to \eqref{eq:compactL} is given by
\begin{align*}
\mathcal{V}=\mathcal{V}(t)&=\delta \exp (ilx-i\alpha t - \L y) \sum_{\pm} (-i\alpha \pm i\sin \gamma)^{-1} \Pi_\pm \mathcal{V}'.
\end{align*}
Since the localization due to $\widehat{\Psi}_{\eps, \sigma}$ in \eqref{def:psi} implies that $\theta$ is supported in a ball of radius $\eps^{1/3}$ centered in $\gamma$ or $\gamma+\pi$, if follows that $\sin (\theta+\gamma) + \sin (\theta'+\gamma)$ is localized either in a ball of size $\eps^{1/3}$ around $\pm 2 \sin \gamma$ or in a ball of size $\eps^{1/3}$ around $0$. These two scenarios correspond to a second harmonic and a mean flow respectively.
It follows in particular that $|-\alpha \pm \sin \gamma| >c_0>0$ as $\alpha$ is either close to $\pm 2 \sin \gamma$ or close to $0$. Thus, recalling from Lemma \ref{lem:asym} the expressions of $\L_i$, $i=1,2$, we can define by superposition:
\begin{align}\label{eq:W1UB}
\text{the ``nonlinear''} & \text{mean flow corrector for } \, \omega+\omega'=O(\eps^{1/3}) \, \text{as}\notag  \\
\mathcal{V}^{1, \rm{nonlin}}_{\rm{(a1); MF }}&=  \frac{\delta \sigma^2}{\eps^{1/3}}  \int \widehat{\Psi}_{\eps,\sigma}  \widehat{\Psi}_{\eps,\sigma}' \sum_{i,i'=1,2}  a_i  a_{i}' e^{-i(\omega+\omega') t + i(k+k') (x-x_0) -(\L_i+\L_i') y}\times \frac{(-k' U + i \L_i'W)}{\omega+\omega' \pm \sin \gamma} \Pi_\pm \begin{pmatrix}U'\\
B' \end{pmatrix}\, dk \, dk'\, dm \, dm';\\
\text{the ``nonlinear''} & \text{second harmonic corrector for } \, \omega+\omega'=\pm \sin 2\gamma + O(\eps^{1/3}) \, \text{as} \notag \\
\mathcal{V}^{1, \rm{nonlin}}_{\rm{(a1); II }}&=  \frac{\delta \sigma^2}{\eps^{1/3}}  \int \widehat{\Psi}_{\eps,\sigma}  \widehat{\Psi}_{\eps,\sigma}' \sum_{i,i'=1,2}  a_i  a_{i}' e^{-i(\omega+\omega') t + i(k+k') (x-x_0) -(\L_i+\L_i') y}\times \frac{(-k' U + i \L_i'W)}{\omega+\omega' \pm \sin \gamma} \Pi_\pm \begin{pmatrix}U'\\
B' \end{pmatrix}\, dk \, dk'\, dm \, dm'.\label{eq:W1UB1}
\end{align}
Hereafter we will use the notation
$$\mathcal{V}^{1, \rm{nonlin}}_{\text{(a1)}}=(u_{\text{(a1)}}, b_{\text{(a1)}})^T \quad (\text{resp.} \; \mathcal{V}^{1, \rm{nonlin}}_{\text{(a2)}}=(u_{\text{(a2)}}, b_{\text{(a2)}})^T) $$ either for $\mathcal{V}^{1, \rm{nonlin}}_{\rm{(a1); MF}}$ (resp. $\mathcal{V}^{1, \rm{nonlin}}_{\rm{(a2); MF}}$) or for $\mathcal{V}^{1, \rm{nonlin}}_{\rm{(a1); II}}$ (resp.$\mathcal{V}^{1, \rm{nonlin}}_{\rm{(a2); II}}$). 

\item[\rm{(ii)}]: \underline{correcting the error of the quadratic term for $W$ by restoring the divergence-free condition.}\\ The vertical velocity $w_{\text{(a1)}}$ does not appear in the main order system, as its size is smaller than $\mathcal{V}^{1, \rm{nonlin}}_{\rm{(a\ell)}}=(u_{\rm{ (a\ell)}}^{1, \text{nonlin}}, b_{\rm{(a\ell)}}^{1, \text{nonlin}})$, $\ell=1,2$, but it can be simply recovered for $\text{(a1)}$ (resp. $\text{(a2)}$) by integrating the divergence free condition:
$$\de_x u_{\text{(a1)}}+ \de_y w_{\text{(a1)}}=0.$$
This gives, for $i=1,2$,
\begin{align}\label{eq:W1}
w_{\rm{(a1)}}&=  \frac{\delta \sigma^2}{ \eps^{1/3}}\int \sum_{i,i'=1,2}  \sum_{\pm} a_i  a_{i}'  \widehat{\Psi}_{\eps,\sigma}  \widehat{\Psi}_{\eps,\sigma}'e^{-i(\omega+\omega') t + i(k+k') (x-x_0) -(\L_i+\L_i')y}\notag\\
& \qquad \qquad \times \frac{i(k+k')}{\L_i+\L_i'} \times \frac{(-k'U+i \L_i' W)}{\omega+\omega'\pm \sin \gamma} (U'\pm iB') \,  dk \, dk'\, dm \, dm',
\end{align}
and a similar expression holds for $w_{\rm{(a2)}}^{1, \rm{nonlin}}$.
Let use denote
\begin{align}\label{eq:W1nonlin}
\mathcal{W}_{\rm{(a\ell)}}^{1, \rm{nonlin}}=(u_{\rm{(a\ell)}}, w_{\rm{(a\ell)}}, b_{\rm{(a\ell)}})^T \quad \ell=1,2.
\end{align}
By construction
\begin{align*}
\de_t \left(\sum_{\ell=1,2} \mathcal{W}_{\rm{(a\ell) }}^{1, \rm{nonlin}}\right)+\mathcal{L}_\eps \left (\sum_{\ell=1,2} \mathcal{W}_{\rm{(a\ell)}}^{1, \rm{nonlin}}\right)= - \delta (\rm{(a1)}+\rm{(a2)}) + r_{(a)}^1, 
\end{align*}
where $ r_{(a)}^1$ is an error term, generated by the approximation of the Leray projector in \eqref{eq:lerayproj} (of order $\eps^{\frac 13-\frac 23 \mu}$) and the viscosity term ($(\eps \de_{yy} +\eps \de_{xx})e^{ikx - \frac{\ell(\theta) \kk^\frac 13}{\eps^{1/3}}y} =  O (\eps^{\frac 13})$ for $\mu$ small enough, see Assumption \ref{ass} ). 
Therefore, we have
\begin{align*}
\|r_{\text{(a)}}^1\|_{L^2} \le C\delta \eps^{\frac 13 - \frac 23 \mu } \|\W_{\rm{(a\ell)}}^{1, \rm{nonlin}}\|_{L^2},
\end{align*}
where, from (a1) in Table \ref{table1}, we recall that for $\ell=1,2$ we have 
\begin{align*}
\|\W_{\rm{(a\ell) }}^{1, \rm{nonlin}}\|_{L^2} =O(\delta \eps^{-\frac 16} \sigma^{-\frac 76}),
\end{align*}
so that
\begin{align}\label{eq:ra1}
\|r_{\text{(a)}}^1\|_{L^2} =O( {\delta \eps^{\frac 16-\frac{11}{6} \mu}}). 
\end{align}
Finally note that $\W_{\rm{(a1) }}^{1, \rm{nonlin}}+\W_{\rm{(a2) }}^{1, \rm{nonlin}}$ corresponds to $\W_{\rm{BL}, \eps^{1/3} }^1$ in the statement of Proposition \ref{prop:BL}.

\begin{remark}
Notice that the interactions of type $\text{(c)}$ in Table 1 have approximately the size of $r_{\text{(a)}}^1$ in the $L^2$ norm. Therefore we will not correct them.
\end{remark}
\end{itemize}
\paragraph{\textbf{Step 2: solving \eqref{eq:secondeqa} by lifting the boundary conditions.}} The oscillating terms in \eqref{eq:W1UB}-\eqref{eq:W1UB1}
provide different contributions. In fact, the superimposed plane waves in the integral  \eqref{eq:W1UB} represent a \emph{nonlinear mean flow} $ \W^{1, \rm{nonlin}}_{\rm{(a1)}; \rm{MF} }$ as $\omega+\omega' = O(\eps^\frac 13)$, while \eqref{eq:W1UB1} is a \emph{nonlinear second harmonic} $\W^{1, \rm{nonlin}}_{\rm{(a1)}; \rm{II} }$ as $\omega + \omega'=\pm 2 \sin \gamma +O(\eps^\frac 13)$.
\begin{remark}
The above contributions are a {second harmonic} and a {mean flow} in terms of \emph{time oscillation}, thanks to the strong localization of the angle $\theta$. On the other hand, the spatial frequency of our beam wave is quite spread, indeed ${\kk} \le  \sigma^{-1}\le \eps^{-\mu}, \, \mu>0$, see \eqref{def:inc-beam}. This is different with respect to the framework of [page 233, \cite{BDSR19}], where the quadratic term is a second harmonic (resp. a mean flow) in \emph{time and space}, in the sense that $\omega = \pm 2\sin \gamma+O(\eps^\frac 13)$ (resp. $\omega=O(\eps^\frac 13)$ \textit{and} $\k = \pm 2 \k_0+O(\eps^\frac 13)$ (resp. $\k=O(\eps^\frac 13)$) for a given $\k_0$. In fact, while in \cite{BDSR19} the approximate solution is constructed as a sum of packets of plane waves with $\k$ lying in a neighborhood of some given $\k_0$ (microlocalized), in the present work we construct a beam wave approximate solution with spreading frequency and strong localization in the physical space.
\end{remark}
Let us consider a \emph{linear mean flow} $\W^{1, \rm{lin}}_{\rm{(a\ell); MF}}$ and a \emph{linear second harmonic} $\W^{1, \text{lin}}_{\rm{ (a\ell); II}}$ solving  the linear part of system \eqref{eq:system} for $\ell=1,2$ (in the non-critical case, with $\omega \sim 0$ and $\omega \sim \pm 2\sin \gamma$ respectively), which are provided by Lemma \ref{lem:asym1} (for the mean flow) and Lemma \ref{lem:asym3} (for the second harmonic). To balance the boundary contribution due to $\W^{1, \rm{nonlin}}_{\rm{(a1)}; \rm{MF} }$ (for $\omega + \omega'=O(\eps^\frac 13)$) and $\W^{1, \rm{nonlin}}_{\rm{(a1)}; \rm{II} }$ (for $\omega + \omega'=\pm 2 \sin \gamma + O(\eps^\frac 13)$), we  will solve 

\begin{align}\label{eq:balanceMF}
\sum_{\ell=1}^2 (\W^{1, \rm{nonlin}}_{\rm{(a\ell)}; \rm{MF} }|_{y=0} + \W^{1, \rm{nonlin}}_{\rm{(a\ell)}; \rm{II} }|_{y=0}+\W^{1, \rm{lin}}_{\rm{(a\ell)}; \rm{MF} }|_{y=0} + \W^{1, \rm{lin}}_{\rm{(a\ell)}; \rm{II} }|_{y=0})=0.
\end{align}

We deal with the MF (mean flow) and the II (second harmonic) term separately. Let us start with MF.


\begin{itemize}
\item[\rm{(i)}] \textbf{The mean flow corrector.}
The \emph{nonlinear} mean flow contribution is the three dimensional vector in \eqref{eq:W1nonlin}, with $\omega +\omega'=O(\eps^\frac 13)$. Evaluating it in $y=0$ for $\ell=1$, it reads as follows
\begin{align}\label{eq:Wnonlin}
 \W^{1, \rm{nonlin}}_{\rm{(a1)}; \rm{MF} }|_{y=0} &  =   \frac{ \delta\sigma^2}{\eps^{1/3}} \int_{\textbf{1}_{\mathcal{M}^0}} \sum_{i=1,2} a_i a_i' \widehat{\Psi}_{\eps, \sigma} \widehat{\Psi}'_{\eps, \sigma}
 \frac{-k'U + i\L_i' W}{\omega+\omega' \pm \sin \gamma} \times 
 \begin{pmatrix}
\Pi_\pm U'\\
\frac{i(k+k')}{\L_i+\L_i'} (U' \pm i B') \\
\Pi_\pm B'
 \end{pmatrix} e^{-i (\omega+\omega') t + i (k+k') (x-x_0)} \, dk \, dk'\, dm \, dm',
\end{align}
where 
\begin{align}\label{def:M0}
\mathcal{M}^0:=\{(k, k') \in \R^2 \, : \,  \omega+\omega' =O(\eps^\frac 13), \; \eps^\frac 13 \eta \le |k+k'| \le  \eps^{-\mu}, \; \mu>0\}.
\end{align}

A similar expression holds for $\W_{\text{(a2); (MF)}}^{1, \text{nonlin}}$.

As $\omega+\omega'=O(\eps^{1/3})$, we seek for a solution to the \emph{linear} part of system \eqref{eq:system} such that $\omega\sim \eps^{\frac 13} $. To this end, we systematically apply Lemma \ref{lem:asym1}, from which we know that for $\omega \sim \eps^\frac 13$, the characteristic polynomial \eqref{eq:pol} admits exactly three roots with strictly positive real part (for fixed $\eps$), which read
\begin{align}\label{eq:L1mean}
\L_1 &= - \frac{i(k+k')}{\sin^2\gamma} (\sin \gamma \cos \gamma + \omega+\omega' + o( \eps^\frac 13)) + \eps \ell_1''(\theta) (k+k')^3   +O(\eps^2(k+k')^5); \notag \\
\L_j&= \eps^{-\frac 12} \ell_j(\theta) + O(\eps^{-\frac16}), \, j=2,3.
\end{align}
We will use the apexes \emph{nonlin} and \emph{lin} in order to distinguish the linear and nonlinear contribution at each step. 
Using the general expression of the eigenvector in \eqref{eq:eigen-BLs}, we write the first two components and multiply the last one by $-\L_j$ to express the no-flux condition on the buoyancy term on $y=0$ in \eqref{eq:cond-BL}. Then, the vector  associated with the \emph{linear mean flow} (the above $\L_1$) reads 
\begin{align}
\begin{pmatrix}
U_{\rm{(a1); {MF}}}^{\L_1}\\
W_{\rm{(a1); {MF}}}^{\L_1}\\
-\L_j B_{\rm{(a1); {MF}}}^{\L_1}\\
\end{pmatrix} = \begin{pmatrix}
1\\
-\tan \gamma + O(\eps^{\frac 13-\mu})\\
 \frac{(k+k')}{\sin^2 \gamma} + O(\eps^{\frac{1}{3}-\mu})
\end{pmatrix}. 
\end{align}
Moreover, in the regime $\omega \sim \eps^\frac 13 \omega_0$ for some $\omega_0 \neq 0$ uniformly, the eigenvectors related to $\L_j=\eps^{-\frac 12} \ell_j(\theta)+O(\eps^{-\frac 16}), \, j=2,3,$ read
\begin{align}
\begin{pmatrix}
U_{\rm{(a1); {MF}}}^{\L_j}\\
W_{\rm{(a1); {MF}}}^{\L_j}\\
-\L_j B_{\rm{(a1); {MF}}}^{\L_j}\\
\end{pmatrix} = \begin{pmatrix}
1\\
\frac{i\eps^\frac 12 (k+k')}{\ell_j} + O(\eps^{\frac 13-\mu})\\
\eps^{-\frac 56} \frac{i \sin \gamma \ell_j }{\omega_0 }+O(\eps^{-\frac 13-\mu})
\end{pmatrix}.
\end{align}
The goal now is to find $(c_1, c_2, c_3) \in \mathbb{C}^3, \, c_j= c_j(k, k')$ such that
\begin{align}\label{eq:BCmf1}
\sum_{\ell=1}^2 (\W^{1, \rm{nonlin}}_{\rm{(a\ell )}; \rm{MF} }|_{y=0}+\W^{1, \rm{lin}}_{\rm{(a \ell)}; \rm{MF} }|_{y=0})=0,
\end{align}
where 
\begin{align*}
\W^{1, \rm{lin}}_{\rm{(a\ell)}; \rm{MF} }
=\sum_{j =1}^3 \W^{1, \rm{lin}, \L_j}_{\rm{(a\ell)};\text{MF} },
\end{align*}
and, for $j=1,2,3$ and $\ell=1,2$,
\begin{align}
\W^{1, \rm{lin}, \L_j}_{\rm{(a\ell)}; \rm{MF} }:&= \frac{\sigma^2}{ \eps^{1/3}} \int_{\mathbf{1}_{\mathcal{M}^0}}  c_j(k,k') \widehat{\Psi}_{\eps, \sigma}\widehat{\Psi}'_{\eps, \sigma}
\begin{pmatrix}
U_{\rm{(a\ell); {MF}}}^{\L_j}\\
 W_{\rm{(a\ell); {MF}}}^{\L_j}\\
 B_{\rm{(a\ell); {MF}}}^{\L_j}
\end{pmatrix}
 e^{-i (\omega+\omega') t + i (k+k')(x-x_0) - \L_j(k+k') y} \, dk \, dk' \, dm \, dm'. \label{eq:mflinear} 
\end{align}

Using \eqref{eq:W1nonlin} and the above expression of the components in \eqref{eq:W1UB} and \eqref{eq:W1}, equation \eqref{eq:BCmf1} reads
 \begin{align}\label{eq:algebraic-system}
\mathcal{C}' 
\begin{pmatrix}
c_1\\
c_2\\
c_3\\ 
\end{pmatrix} =\sum_{i=1,2} a_i a_i' \widehat{\Psi}_{\eps, \sigma}\widehat{\Psi}'_{\eps, \sigma}
 \frac{-k'U + i \L_i' W}{\omega+\omega' \pm \sin \gamma} \times 
 \begin{pmatrix}
\Pi_\pm U'\\
\frac{i(k+k')}{\L_i+\L_i'} (U' \pm i B') \\
\Pi_\pm B'
 \end{pmatrix}=: \text{(r.h.s.)},
 \end{align}
where 
\begin{align}\label{eq:algebraic}
\mathcal{C}'=\begin{pmatrix}
U_{\rm{(a1); {MF}}}^{\L_1} & U_{\rm{(a1); {MF}}}^{\L_2} & U_{\rm{(a1); {MF}}}^{\L_3}\\
W_{\rm{(a1); {MF}}}^{\L_1} & W_{\rm{(a1); {MF}}}^{\L_2} & W_{\rm{(a1); {MF}}}^{\L_3}\\
-\L_1 B_{\rm{(a1); {MF}}}^{\L_1} & -\L_2 B_{\rm{(a1); {MF}}}^{\L_2} & -\L_3 B_{\rm{(a1); {MF}}}^{\L_3}\\
\end{pmatrix}\sim  \begin{pmatrix}
1 & 1& 1\\
-\tan\gamma  & \eps^{\frac 12} (k+k')\ell_2^{-1} & \eps^{\frac 12} (k+k')\ell_3^{-1} \\
 \frac{(k+k')}{\sin^2 \gamma}& \frac{i \sin \gamma \ell_2}{\omega_0\eps^{5/6}} & \frac{i \sin \gamma\ell_3}{\omega_0\eps^{5/6}} 
\end{pmatrix}.
\end{align}
In the regime of Assumptions \ref{ass}, where $\eps^\frac 13 \eta \le |k+k'| \le C\sigma^{-1}$, one easily checks that
\begin{align*}
({\mathcal{C}}')^{-1} \sim \frac{i \tan \gamma (\ell_3-\ell_2) }{\omega_0}
&\begin{pmatrix}
\eps^\frac 12 (k+k') (\frac{\ell_2}{\ell_3}-\frac{\ell_3}{\ell_2}) &  \frac{i\sin \gamma (\ell_2-\ell_3)}{\omega_0} & \eps^\frac 43 (k+k')(\ell_3^{-1}-\ell_2^{-1}) \\
 \frac{i \tan \gamma \sin \gamma \ell_3}{\omega_0} & \frac{i \sin \gamma \ell_3 }{\omega_0} & -\eps^\frac 56 \tan \gamma \\
 -\frac{i \tan \gamma \sin \gamma \ell_2}{\omega_0} & -\frac{i \sin \gamma \ell_2 }{\omega_0} & \eps^\frac 56 \tan \gamma \\
\end{pmatrix}\\
&=\begin{pmatrix}
O(\eps^{\frac 12} (k+k')) & O(1)  & O( \eps^{\frac 43} (k+k'))\\
O(1)   & O(1)  & O(\eps^{\frac 56}) \\
O(1)   & O(1)  & O(\eps^{\frac 56})
\end{pmatrix}.
\end{align*}
Now we can solve \eqref{eq:algebraic-system} by applying $(\mathcal{C}')^{-1}$ to (r.h.s). Let us first note from \eqref{eq:W1UB} and \eqref{eq:W1} (where we recall that $i=1,2$ according with the numerology of Lemma \ref{lem:asym}) that
\begin{align*}
\text{(r.h.s.)} = \eps^{-\frac 23} \kk^{-\frac 13} |\k'|^\frac 23 \begin{pmatrix}
O(1)\\
O(\eps^\frac 13 (k+k'))\\
O(1)
\end{pmatrix},
\end{align*}
so that
\begin{align}\label{eq:amplitudeMF}
c_1 = O(\eps^{-\frac 13} \kk^{-\frac 13} |\k'|^{\frac 23}) \times  (k+k'), \quad c_{2,3}   =O(\eps^{-\frac 23} \kk^{-\frac 13}|\k'|^\frac 23).
\end{align}
{
We will now estimate the norm of $\W^{1, \rm{lin}, \L_j}_{\rm{(a1)}; \rm{MF} }$, $j=1,2,3$.
Let us first focus on $\W^{1, \rm{lin}, \L_1}_{\rm{(a1)}; \rm{MF} }$. 

Recalling that $\begin{pmatrix}
U_{\rm{(a\ell); {MF}}}^{\L_1}\\
 W_{\rm{(a\ell); {MF}}}^{\L_1}\\
 B_{\rm{(a\ell); {MF}}}^{\L_1}
\end{pmatrix}$ is $O(1)$ in $L^\infty$, using \eqref{eq:mflinear}, \eqref{eq:L1mean}, the expression of $c_1$ in \eqref{eq:amplitudeMF}, 
we have
\begin{align}
\W^{1, \rm{lin}, \L_1}_{\rm{(a\ell)}; \rm{MF} }&= \frac{\delta\sigma^2}{ \eps^{1/3}} \int_{\mathbf{1}_{\mathcal{M}^0}}  c_1\widehat{\Psi} \widehat{\Psi}' 
\begin{pmatrix}
U_{\rm{(a\ell); {MF}}}^{\L_1}\\
 W_{\rm{(a\ell); {MF}}}^{\L_1}\\
 B_{\rm{(a\ell); {MF}}}^{\L_1}
\end{pmatrix}
 e^{-i (\omega+\omega') t + i (k+k')(x-x_0)  + i(k+k')(\cot\gamma + \frac{\omega + \omega'}{\sin^2\gamma} + o(\eps^{1/3})) y}\notag\\
 & \quad  \times e^{-\text{Re}(\L_1) y} \, dk \, dk' \, dm \, dm'\notag\\
 &= \frac{\delta\sigma^2}{ \eps^{2/3}} \int_{\mathbf{1}_{\mathcal{M}^0}}  O(\kk^{-\frac 13}) O(|\k'|^{\frac 23})(k+k') \widehat{\Psi}_{\eps, \sigma}\widehat{\Psi}'_{\eps, \sigma}
 e^{-i (\omega+\omega') t + i (k+k')(x-x_0)  + i(k+k')(\cot\gamma + \frac{\omega + \omega'}{\sin^2\gamma} + o(\eps^{1/3})) y}\notag\\
 & \quad  \times e^{-\text{Re}(\L_1) y} \, dk \, dk' \, dm \, dm'. \label{eq:Wlinlambda1}
\end{align}
At a first stage, we ignore the decay. We want to compute the $L^2$ norm of 
\begin{align*}
\W_{\rm{no decay}} &=  \frac{\delta \sigma^2}{ \eps^{2/3}} \int_{\mathbf{1}_{\mathcal{M}^0}}  O(\kk^{-\frac 13}) O(|\k'|^{\frac 23})(k+k') \widehat{\Psi}_{\eps, \sigma}\widehat{\Psi}'_{\eps, \sigma}
 e^{-i (\omega+\omega') t + i (k+k')(x-x_0)  + i(k+k')( \cot\gamma + \frac{\omega + \omega'}{\sin^2\gamma} + o(\eps^{1/3})) y}\, dk \, dk' \, dm \, dm' .
\end{align*}

Let us consider the following change of variables  $\widetilde \k=(k, \widetilde m)$, $\widetilde \k'=(k', \widetilde m')$, where 
\begin{align}\label{def:tildem}
\widetilde{m}= k \left(\cot\gamma+ \frac{\omega + \omega'}{\sin^2\gamma}\right); \quad \widetilde{m}'= k' \left(\cot\gamma + \frac{\omega + \omega'}{\sin^2\gamma}\right).
\end{align}
The (transpose of the) Jacobians of the change of variables are
\begin{align*}
\det \begin{pmatrix}
1 & \cot \gamma +\frac{\omega+\omega' + k \de_k \omega}{\sin^2\gamma}\\
0 & \frac{k(\de_m \omega)}{\sin^2\gamma} 
\end{pmatrix}= - \frac{k^2 (k\sin \gamma + m \cos \gamma)}{|\k|^3(\sin\gamma)^2}; \quad \det \begin{pmatrix}
1 & \cot \gamma+\frac{\omega+\omega' + k' \de_{k'} \omega'}{\sin^2\gamma}\\\
0 & \frac{k'(\de_{m'} \omega')}{\sin^2\gamma}
\end{pmatrix}= - \frac{(k')^2 (k'\sin \gamma + m' \cos \gamma)}{|\k'|^3(\sin\gamma)^2},
\end{align*}
where the expressions of $\de_m \omega$ and $\de_{m'} \omega'$ can be computed from 
 \eqref{defomg}.
Note that as $(\kk, \frac \pi 2-(\theta+\gamma))$ are the polar coordinates of $(k,m)$ (resp. $(k',m'))$, 
it follows that $k \sin \gamma + m \cos \gamma=\kk \cos(\theta)$ (resp.  $k' \sin \gamma + m' \cos \gamma =|\k'| \cos(\theta')$). Now, thanks to the angular localization provided by $\widehat{\Psi}_{\eps, \sigma} \widehat{\Psi}'_{\eps, \sigma}$ inside the integral, one has that
\begin{align*}
dkdk'dm  dm' = O(1) \, dkdk'd\widetilde{m} d\widetilde{m}',
\end{align*}
so that we can rewrite
\begin{align}\label{eq:nodecay}
\W_{\rm{no decay}} &=  \frac{\delta \sigma^2}{ \eps^{2/3}} \int_{\mathbf{1}_{\mathcal{M}^0}}  O(\widetilde \k^{\frac 23}) O(|\widetilde \k'|^{\frac 53}) \widehat{\Psi} \widehat{\Psi}' 
 e^{-i (\omega+\omega') t + i (k+k')(x-x_0)  + i(\widetilde m + \widetilde m'+ o(\eps^{\frac 13})) y}\,  dk \, dk' \, d\widetilde m \, d\widetilde m'.
\end{align}

Now note that from \eqref{def:tildem}, denoting by $\widetilde \theta$ the angle of $\widetilde \k$ and by $\widetilde \theta'$ the angle of  $\widetilde \k'$, as $\omega+\omega'=O(\eps^\frac 13)$ we have that
\begin{align*}
\tan \widetilde \theta = \frac{\widetilde m}{k} = \cot \gamma + O(\eps^\frac 13); \quad \tan \widetilde \theta' = \frac{\widetilde m'}{k'} = \cot \gamma + O(\eps^\frac 13).
\end{align*}
We use this information to express the angular localization due to $\widehat{\Psi}_{\eps, \sigma} \widehat{\Psi}'_{\eps, \sigma}$ inside the integral in terms of the new angular variables $\widetilde \theta, \widetilde \theta'$. Recall the expression of $\widehat{\Psi}_{\eps, \sigma}$ in Definition \ref{def:inc-beam}, i.e. 
\begin{align*}
\widehat{\Psi}_{\eps, \sigma}&=\chi(\sigma \kk)( \chi' (\eps^{-\frac 13} (\sin \theta-\sin \gamma)) \chi (\eps^{-\frac 13} (\cos \theta-\cos \gamma)) + \chi' (\eps^{-\frac 13} (\sin \theta+\sin \gamma)) \chi (\eps^{-\frac 13} (\cos \theta+\cos \gamma))).
\end{align*}
We know from \eqref{eq:omega} that
$$\omega= \sin \theta, \quad \omega' =  \sin \theta',$$
so that the change of variables \eqref{def:tildem} reads
\begin{align}\label{eq:locangle}
\frac{\widetilde{m}}{k}= \tan \widetilde \theta =\cot\gamma + \sin \theta + \sin \theta'; \quad \frac{\widetilde{m}'}{k'}= \tan \widetilde \theta' = \cot \gamma + \sin \theta + \sin \theta'.
\end{align}
Now, the presence of $\widehat{\Psi}_{\eps, \sigma}, \widehat{\Psi}'_{\eps, \sigma}$ inside the integral implies that 
$$\theta, \theta' \in  \{\gamma + O(\eps^\frac 13), \gamma+\pi +O(\eps^\frac 13)\}.$$
As we are dealing with the \emph{linear mean flow} \bw, by definition $\omega+\omega'=O(\eps^\frac 13)$. Thus let us consider for simplicity the case 
$$\theta= \gamma + O(\eps^\frac 13); \quad \theta' =  \gamma+\pi + O(\eps^\frac 13).$$
It then follows from \eqref{eq:locangle} that
\begin{align}\label{eq:ang-loc}
\sin \theta - \sin \gamma = \tan \widetilde \theta - \cot \gamma+O(\eps^\frac 13); \quad \sin \theta' + \sin \gamma = \tan \widetilde \theta' - \cot \gamma + O(\eps^\frac 13),
\end{align}
so that
\begin{align*}
\widetilde \theta = \arctan (\sin \theta - \sin \gamma+\cot \gamma)+O(\eps^\frac 13); \quad \widetilde \theta' = \arctan (\sin \theta' + \sin \gamma+\cot\gamma)+O(\eps^\frac 13).
\end{align*}
Since $\sin \theta - \sin \gamma=O(\eps^\frac 13)$ and $\sin \theta' + \sin \gamma=O(\eps^\frac 13)$, we deduce that
\begin{align*}
\widetilde \theta=\arctan \cot \gamma +O(\eps^\frac 13); \quad \widetilde \theta'=\arctan \cot \gamma +O(\eps^\frac 13),
\end{align*}
which in the new \emph{tilde} variables provides an angular localization with the same $\eps^{1/3}$ scaling (and a different angle), i.e.
\begin{align*}
\widetilde \theta, \widetilde \theta' \in \{ \frac \pi 2 - \gamma, \frac{3\pi}{2}- \gamma\}. 
\end{align*}
In order to express the above angular localization of $\widetilde \theta, \widetilde \theta'$ inside the integral, one has to write $\widehat{\Psi}_{\eps, \sigma}, \widehat{\Psi}'_{\eps, \sigma}$ as functions of $|\widetilde \k|, \widetilde \theta$ and $|\widetilde \k'|, \widetilde \theta'$ respectively. It is immediate to notice that since \eqref{def:tildem} gives 
\begin{align*}
\widetilde m&= k \cot \gamma + O(\eps^{\frac 13-\mu})=|\k| \cot \gamma \sin (2\gamma) + O(\eps^{\frac 13-\mu}),\\
\widetilde m'&= k' \cot \gamma + O(\eps^{\frac 13-\mu})=-|\k'| \cot \gamma \sin(2\gamma) + O(\eps^{\frac 13-\mu}),
\end{align*}
it follows that there exist two universal constants $C, C'$ such that  $\chi (\sigma \kk) =\chi (C \sigma |\widetilde \k|)$ and $\chi (\sigma |\k'|) =\chi (C \sigma |\widetilde \k'|)$. This provides lower ($C\eta, C'\eta$) and upper ($C\sigma^{-1}, C'\sigma^{-1}$) extreme values of the interval where $|\widetilde \k|, |\widetilde \k'|$ take values. 

\noindent Taking into account the normalization factor $\eps^{-\frac 16}{\sigma}$ of a beam wave (see Definition \ref{def:inc-beam}) (and the weakly nonlinear small parameter $\delta$), we deduce from the above computations that \eqref{eq:nodecay} is approximately  a product of a \bw of order $(-\frac 16, \frac 23)$ by a \bw of order $(-\frac 16, \frac 53)$. Thus we can reproduce the proof of Lemma \ref{estbbl} to estimate its $L^2$ norm. Using \eqref{eq:normprod} yields
\begin{align}\label{eq:normWnodec}
\|\W_{\rm{no decay}}\|_{L^2} & =O(\delta \eps^{-\frac 16} \sigma^{-\frac{10}{3}}). 
\end{align}}

Now we would like to use this information on the $L^2$ norm of $\W_{\rm{no decay}}$ to deduce the size of $\W^{1, \rm{lin}, \L_1}_{\rm{(a\ell)}; \rm{MF} }$ in $L^2$. We use the following change of variables
\begin{align*}
\zeta_1=k+k', \quad \zeta_2=\widetilde m + \widetilde m'.
\end{align*}
This yields that
\begin{align*}
\W_{\rm{no decay}} &=  \frac{\delta \sigma^2}{ \eps^{2/3}} \int_{\mathbf{1}_{\mathcal{M}^0}}  O(|\widetilde \k|^{\frac 23}+ |\zeta|^{\frac 23} ) \zeta_1  \widehat{\Psi}_{\eps, \sigma}\widehat{\Psi}'_{\eps, \sigma}
 e^{-i (\omega+\omega') t + i \zeta_1 (x-x_0)  + i(\zeta_2 + o(\eps^{\frac 13-\mu})) y}\,  dk  \, d\widetilde m \, d\zeta_1 \, d\zeta_2.
\end{align*}

Let us go back to 
$
\W^{1, \rm{lin}, \L_1}_{\rm{(a\ell)}; \rm{MF} }
$ in \eqref{eq:Wlinlambda1},
with the slow decay in $y$ given by $\text{Re}(\L_1)$ in \eqref{eq:L1mean}.
We have that
\begin{align*}
\W^{1, \rm{lin}, \L_1}_{\rm{(a\ell)}; \rm{MF} }& =  \frac{\delta \sigma^2}{ \eps^{2/3}} \int_{\mathbf{1}_{\mathcal{M}^0}}O(|\widetilde \k|^{\frac 23}+ |\zeta|^{\frac 23} ) \zeta_1   (1- e^{-\text{Re}(\L_1) y})   \widehat{\Psi}_{\eps, \sigma}\widehat{\Psi}'_{\eps, \sigma}
 e^{-i (\omega+\omega') t + i \zeta_1 (x-x_0)  + i(\zeta_2 + o(\eps^{\frac 13-\mu})) y}\,  dk  \, d\widetilde m \, d\zeta_1 \, d\zeta_2.
\end{align*}

One has to estimate the latter.
Note that
\begin{align}\label{eq:Ftransfx}
\mathcal{F}_{(x-x_0)\rightarrow \zeta_1}(\W_{\rm{nodecay}}-\W^{1, \rm{lin}, \L_1}_{\rm{(a\ell)}; \rm{MF} }) & =  \frac{\delta\sigma^2}{\eps^{2/3}} \int  O(|\widetilde \k|^{\frac 23}+ |\zeta|^{\frac 23} ) \zeta_1    \widehat{\Psi}_{\eps, \sigma}\widehat{\Psi}'_{\eps, \sigma} e^{-i (\omega+\omega') t + i \zeta_2 y} (1-e^{-\text{Re}(\L_1) y})\ \,  dk  \, d\widetilde m \, d\zeta_2.
\end{align}

Now, we can compute for any $\lambda>0$
\begin{align*}
 \|\mathcal{F}_{(x-x_0)\rightarrow \zeta_1}(\W_{\rm{nodecay}}-\W^{1, \rm{lin}, \L_1}_{\rm{(a\ell)}; \rm{MF} })\|_{L^2_{y \le \eps^{-\lambda}}}& \le \frac{\delta \sigma^{2}}{\eps^{2/3}} \int O((|\widetilde \k|^{\frac 23}+ |\zeta|^{\frac 23}) |\zeta| )  \widehat{\Psi}_{\eps, \sigma}\widehat{\Psi}'_{\eps, \sigma}\|1-e^{-\text{Re}(\L_1) y}\|_{L^{2}_{y \le \eps^{-\lambda}}}\, dk  \, d\widetilde m \, d\zeta_2.
\end{align*}

Let us look at $ \|1-e^{-\text{Re}(\L_1) y}\|_{L^{2}_{y \le \eps^{-\lambda}}}$. From \eqref{eq:L1mean} and recalling that we are integrating over $\mathcal{M}^0$ in \eqref{def:M0}, one has that
\begin{align*}
 \text{Re}(\L_1) \sim \eps (k+k')^3 \le O(\eps^{1-3\mu}).
\end{align*}
In particular, there exists a universal constant $C>0$ such that
$$1-e^{-\text{Re}(\L_1) y} \le 1-e^{-C\eps^{1-3\mu} y}.$$
Therefore, we have
\begin{align*}
 \|1-e^{-\text{Re}(\L_1) y}\|_{L^{2}_{y \le \eps^{-\lambda}}} \le \|1-e^{-C\eps^{1-3\mu} y}\|_{L^2_{y \le \eps^{-\lambda}}}=O(\eps^{1-3\mu - \frac 32 \lambda})=o(1) \quad \text{for} \quad \lambda < \frac 23 (1-3\mu).
\end{align*}
This gives that
\begin{align*}
 \|\mathcal{F}_{(x-x_0)\rightarrow \zeta_1}(\W_{\rm{nodecay}}-\W^{1, \rm{lin}, \L_1}_{\rm{(a\ell)}; \rm{MF} })\|_{L^2_{y \le \eps^{-\lambda}}}& \le C \frac{\delta \sigma^{2}}{\eps^{2/3}} \times \eps^{1-3\mu - \frac 32 \lambda}  \int O((|\widetilde \k|^{\frac 23}+ |\zeta|^{\frac 23}) |\zeta| )    \widehat{\Psi}_{\eps, \sigma}\widehat{\Psi}'_{\eps, \sigma}\, dk  \, d\widetilde m \, d\zeta_2.
\end{align*}
Now we use Plancherel identity, so that
\begin{align*}
\|\W_{\rm{nodecay}}-\W^{1, \rm{lin}, \L_1}_{\rm{(a\ell)}; \rm{MF} }\|_{L^2_{x,y| y \le \eps^{-\lambda}}} & \le  \|\mathcal{F}_{(x-x_0)\rightarrow \zeta_1}(\W_{\rm{nodecay}}-\W^{1, \rm{lin}, \L_1}_{\rm{(a\ell)}; \rm{MF} })\|_{L^2_{\zeta_1} L^2_{y \le \eps^{-\lambda}}} \\
&\le C\delta  \sigma^2 \eps^{\frac 13-3\mu - \frac 32 \lambda}\left\|  \int O((|\widetilde \k|^{\frac 23}+ |\zeta|^{\frac 23}) |\zeta| )  \widehat{\Psi}_{\eps, \sigma}\widehat{\Psi}'_{\eps, \sigma} \, dk  \, d\widetilde m \, d\zeta_2\right\|_{L^2_{\zeta_1}} \\
& \le  C \delta \sigma^{-\frac 23} \eps^{\frac 23-3\mu - \frac 32 \lambda}=O(\delta \eps^{\frac 23-\frac{11}{3}\mu - \frac 32\lambda}),
\end{align*}
where the additional $\eps^\frac 13$ factor comes from the angular localization of the vector $(k, \widetilde m)$. Choosing for instance $\lambda=\frac 23 (1-4\mu)$ with $\mu $ small enough, since from \eqref{eq:normWnodec} we have $\|\W_{\rm{nodecay}}\|_{L^2} =O(\delta \eps^{-\frac 16-\frac{10}{3} \mu})$, under the condition $O(\eps^{-\mu}) \le O (\eps^{-\frac 16 - \frac{10}{3}\mu})$ (which is always satisfied for $\mu>0$), it follows that
\begin{align*}
\|\W^{1, \rm{lin}, \L_1}_{\rm{(a\ell)}; \rm{MF} }\|_{L^2_{x,y| y \le \eps^{-\lambda}}} \le 2 \|\W_{\rm{nodecay}}\|_{L^2_{x,y| y \le \eps^{-\lambda}}}.
\end{align*}
Coming back to the variables $( \widetilde m,  \widetilde m')$, integrating by parts in $\widetilde m$, we have that
\begin{align*}
i y \mathcal{F}_{(x-x_0) \rightarrow \zeta_1}(\W^{1, \rm{lin}, \L_1}_{\rm{(a\ell)}; \rm{MF} }) &=  -\frac{\delta\sigma^2}{\eps^{2/3}} \int e^{i(\widetilde m+\widetilde m') y} \de_{\widetilde m} (O(|\widetilde \k|^{\frac 23}+ |\widetilde \k'|^{\frac 23} ) \zeta_1\widehat{\Psi}_{\eps, \sigma}\widehat{\Psi}'_{\eps, \sigma}
 e^{-i (\omega+\omega') t} e^{-\text{Re}(\L_1) y}) \,  dk  \, d\widetilde m \, d\widetilde m'\\
 &= - \frac{\delta\sigma^2}{\eps^{2/3}}  \int  e^{i(\widetilde m+\widetilde m') y} O((1+t)|\widetilde \k|^{-1}(|\widetilde \k|^{\frac 23}+ |\widetilde \k'|^{\frac 23}) )\zeta_1 \widehat{\Psi}_{\eps, \sigma}\widehat{\Psi}'_{\eps, \sigma}
 e^{-i (\omega+\omega')  t} e^{-\text{Re}(\L_1) y} \,  dk  \, d\widetilde m \, d\widetilde m'.
\end{align*}
This way
\begin{align*}
|\mathcal{F}_{(x-x_0) \rightarrow \zeta_1}(\W^{1, \rm{lin}, \L_1}_{\rm{(a\ell)}; \rm{MF} })| & =y^{-1} |\mathcal{F}_{(x-x_0) \rightarrow \zeta_1}(\W^{1, \rm{lin}, \L_1}_{\rm{(a\ell)}; \rm{MF} })| O(1+t),
\end{align*}
and 
\begin{align*}
\|\mathcal{F}_{(x-x_0) \rightarrow \zeta_1}(\W^{1, \rm{lin}, \L_1}_{\rm{(a\ell)}; \rm{MF} })\|_{L^2_{\zeta_1, y| y > \eps^{-\lambda}}} & = O(\eps^{ \lambda }(1+t))\|\W^{1, \rm{lin}, \L_1}_{\rm{(a\ell)}; \rm{MF} }\|_{L^2(\R^2_+)}.
\end{align*}
Now, as $\lambda=\frac 23 (1-4\mu)$, using that
$$\|\W^{1, \rm{lin}, \L_1}_{\rm{(a\ell)}; \rm{MF} }\|_{L^2(\R^2_+)}\le \|\W^{1, \rm{lin}, \L_1}_{\rm{(a\ell)}; \rm{MF} }\|_{L^2_{x, y| y > \eps^{-\lambda}}}+\|\W^{1, \rm{lin}, \L_1}_{\rm{(a\ell)}; \rm{MF} }\|_{L^2_{x, y| y \le \eps^{-\lambda}}},$$ we have that
\begin{align*}
\|\W^{1, \rm{lin}, \L_1}_{\rm{(a\ell)}; \rm{MF} }\|_{L^2_{x, y| y > \eps^{-\lambda}}} (1-O(\eps^{\frac 23 - \frac 83 \mu}(1+t))) \le \|\W^{1, \rm{lin}, \L_1}_{\rm{(a\ell)}; \rm{MF} }\|_{L^2_{x, y| y \le \eps^{-\lambda}}} \le 2 \|\W_{\rm{nodecay}}\|_{L^2(\R^2_+)}.
\end{align*}
As $\mu < \frac 14$, we get the desired bound until $t=O(\eps^{-\frac 23 + \frac 83 \mu})$, which is already longer than our stability time-scale in Remark \ref{rmk:timescale}, but the same type of estimate for all times can be obtained by integrating by parts in time (recall that 
$\W^{1, \rm{lin}, \L_1}_{\rm{(a\ell)}; \rm{MF} } \times \textbf{1}_{(\eps^{-\lambda}, + \infty)}$ 
at $\eps$ fixed is the tail of an $L^2(\R^2_+)$ function in the Schwartz space).

\vspace{3mm}

Now we compute the norms of the remaining $\W_{\rm{(a\ell); MF} }^{1, \rm{lin},\L_j }$ for $j=2,3$.
For the boundary layer part, recall from \eqref{eq:L1mean} (and the above discussion) that for $j=2,3$ we have $\text{Re}(\lambda_j)\ge \frac{C_j}{\eps^{1/2}}$ for some universal constant $C_j>0$ independent of $k,k'$. Therefore, $\W_{\rm{(a\ell); MF} }^{1, \rm{lin},\L_j} $ can be seen as a product of a \blbw of order $(\frac 12, 0, - \frac 13, - \frac 13)$ and a \blbw of order $(\frac 12, 0, - \frac 13, \frac 23)$. Using Lemma \ref{estbbl} yields
$$\|\W_{\rm{(a\ell); MF} }^{1, \rm{lin},\L_j }\|_{L^2} =O(\delta \eps^{-\frac{1}{12}} \sigma^{-\frac{11}{6}}), \quad j=2,3.$$
\begin{remark}
Note that in the \emph{nonlinear mean flow corrector} \eqref{eq:W1UB}-\eqref{eq:W1} where
 $\omega+\omega' \sim  \eps^\frac 13$, we have that  $\eps^{\frac 13} \lesssim  k+k'  \lesssim \eps^{-\mu} $ for $\mu>0$ small. This is different from the setting of \cite{BDSR19}, where $k+k' \sim \eps^\frac 13$ if and only if $\omega + \omega\sim  \eps^\frac 13$ (namely, a time resonance is also a spatial resonance). In fact, in our case time resonances do not correspond necessarily to spatial resonances thanks to the stronger spatial localization of our beam waves with respect to the packets of plane waves in \cite{BDSR19}.

\end{remark}

\item[\rm{(ii)}] \textbf{The second harmonic corrector.} 
Similarly to what has been done before, now we seek for a linear solution of system \eqref{eq:linsystem} that balances the boundary contribution due to the (nonlinear) second harmonic in \eqref{eq:W1UB1}.
Then we appeal to Lemma \ref{lem:asym3}. We see that the asymptotics of the roots $\L_j, j=1,2,3$ are the same as for the mean flow in \eqref{eq:L1mean}, but in the present case $\omega+\omega' = \pm 2\sin \gamma +O(\eps^\frac 13)$ and $k+k'$ has a uniform lower bound given by $\eta>0$ for all $\eps>0$. Note in particular that as Lemma \ref{lem:asym3} we know that $$\L_1 \sim \frac{i(k+k')}{((\omega+\omega')^2-\sin^2\gamma)} (\sin \gamma \cos \gamma + (\omega+\omega') \sqrt{1-(\omega+\omega')^2}),$$ then there is a contribution with with strictly positive real part if $(\omega+\omega')^2>1$. Since $\omega+\omega' \sim \pm 2\sin \gamma$, this real contribution appears if $\sin \gamma > \frac 12$. In this case, the second harmonic as a decay of $O(1)$ in $y$ and it is called \emph{evanescent second harmonic} (see \cite{BDSR19, DY1999} for further details).

Now, exactly as before, we introduce 
the functions $c_j$ for $j=1,2,3$ to be determined by solving the analogous of \eqref{eq:algebraic-system}, and from \eqref{eq:W1nonlin} with $\omega+\omega' \sim \pm 2 \sin \gamma$ for $\ell=1,2$ we have
\begin{align}
\W^{1, \rm{lin}, \L_j}_{\rm{(a\ell)}; \rm{(II)} }:&= \frac{\delta \sigma^2}{ \eps^{1/3}} \int_{\mathbf{1}_{\mathcal{M}^{\rm{II}}}}  c_j(k,k') \widehat{\Psi}_{\eps, \sigma}\widehat{\Psi}'_{\eps, \sigma}
\begin{pmatrix}
U_{\rm{(a\ell); {(II)}}}^{\L_j}\\
 W_{\rm{(a\ell); {(II)}}}^{\L_j}\\
 B_{\rm{(a\ell); {(II)}}}^{\L_j}
\end{pmatrix}
 e^{-i (\omega+\omega') t + i (k+k')(x-x_0) - \L_j(k+k') y} \, dk \, dk' \, dm \, dm', \label{eq:2linear} 
\end{align}
\begin{align}\label{def:MII}
\mathcal{M}^{\rm{II}}:=\{(k, k') \in \R^2 \, : \,  \omega+\omega' =\pm 2\sin \gamma + O(\eps^\frac 13), \; \eta \le |k+k'| \le  \eps^{-\mu}, \; \mu>0\}.
\end{align}
As before, the functions $c_j$ solve the linear algebraic system 
\begin{align}\label{eq:BCmf}
\sum_{\ell=1}^2 (\W^{1, \rm{nonlin}}_{\rm{(a\ell )}; \rm{(II)} }|_{y=0}+\W^{1, \rm{lin}}_{\rm{(a \ell)}; \rm{(II)} }|_{y=0})=0,
\end{align}
where 
\begin{align*}
\W^{1, \rm{lin}}_{\rm{(a\ell)}; \rm{(II)} }
=\sum_{j =1}^3 \W^{1, \rm{lin}, \L_j}_{\rm{(a\ell)};\text{(II)} },
\end{align*}
and we recall that $\W^{1, \rm{nonlin}}_{\rm{(a\ell )}; \rm{(II)} }$ is given exactly by the same expression in \eqref{eq:Wnonlin}, with $\mathcal{M}^{0}$ replaced by $\mathcal{M}^{\rm{II}}.$ One can repeat the computations of the previous section to obtain again that
\begin{align}\label{eq:amplitudeII}
c_1 = O(\eps^{-\frac 13} \kk^{-\frac 13} |\k'|^{\frac 23}) \times  (k+k'), \quad c_{2,3}   =O(\eps^{-\frac 23} \kk^{-\frac 13}|\k'|^\frac 23).
\end{align}
An application of Lemma \ref{estbbl} gives immediately that for the boundary layer part (for $j=2,3$) we have
$$\|\W_{\rm{(a1); (II)} }^{1, \rm{lin},\L_j }\|_{L^2} =O(\delta \eps^{-\frac{1}{12}} \sigma^{-\frac{11}{6}}), \quad j=2,3.$$
It remains to estimate the $L^2(\R^2_+)$ norm of $\W_{\rm{(a\ell); (II)} }^{1, \rm{lin},\L_1}$, i.e. the second harmonic \bw
We have that
\begin{align*}
\W^{1, \rm{lin}, \L_1}_{\rm{(a\ell)}; \rm{(II)} }&= \frac{\delta \sigma^2}{ \eps^{1/3}} \int_{\mathbf{1}_{\mathcal{M}^{\rm{II}}}}  c_1\widehat{\Psi}_{\eps, \sigma}\widehat{\Psi}'_{\eps, \sigma}
\begin{pmatrix}
U_{\rm{(a\ell); {(II)}}}^{\L_1}\\
 W_{\rm{(a\ell); {(II)}}}^{\L_1}\\
 B_{\rm{(a\ell); {(II)}}}^{\L_1}
\end{pmatrix}
 e^{-i (\omega+\omega') t + i (k+k')(x-x_0)  + i\frac{(k+k')}{31\sin^2\gamma}(\sin \gamma \cos \gamma +(\omega + \omega')\sqrt{1-(\omega+\omega')^2}+ o(\eps^{1/3})) y}\\
 & \quad  \times e^{-\text{Re}(\L_1) y} \, dk \, dk' \, dm \, dm'\\
 &= \frac{\delta \sigma^2}{ \eps^{2/3}} \int_{\mathbf{1}_{\mathcal{M}^{\rm{II}}}}  O(\kk^{-\frac 13}) O(|\k'|^{\frac 23})(k+k')\widehat{\Psi}_{\eps, \sigma}\widehat{\Psi}'_{\eps, \sigma}
 e^{-i (\omega+\omega') t + i (k+k')(x-x_0)  + i\frac{(k+k')}{3\sin^2\gamma}(\sin \gamma \cos \gamma + (\omega + \omega')\sqrt{1-(\omega+\omega')^2} + o(\eps^{1/3})) y}\\
 & \quad  \times e^{-\text{Re}(\L_1) y} \, dk \, dk' \, dm \, dm'.
\end{align*}
As before, we can ignore the decay in $y$. We apply again the change of variables \eqref{def:tildem} and thanks to the angular localization $$k \sin \gamma + m \cos \gamma \sim k' \sin \gamma + m' \cos \gamma   \sim \cos (\gamma), $$
so that we have again
$$dkdk'd m d  m' = O(1) dkdk' d\widetilde m d \widetilde m'.$$

Since now $\omega + \omega'=\pm 2 \sin \gamma + O(\eps^\frac 13)$, in this case the change of variables \eqref{def:tildem} gives that
$$\widetilde \theta \sim \arctan (\cot \gamma \pm \frac{2}{ \sin \gamma}); \quad \widetilde \theta' \sim \arctan (\cot \gamma \pm \frac{2}{ \sin \gamma}),$$
so that the angular localization (with a different angle) still holds for the vectors $(k, \widetilde m), (k', \widetilde m')$.
Therefore, exactly as before, we have that
\begin{align*}
\|\W^{1, \rm{lin}, \L_1}_{\rm{(a\ell)}; \rm{(II)} }\|_{L^2} & =O(\delta \eps^{-\frac 16} \sigma^{-\frac{10}{3}}). 
\end{align*}
\end{itemize}
The proof of Proposition \ref{prop:inta} is now complete.


\subsection{Interactions of type (b)}
Consider the interactions $\rm{(b\ell)}$, $\ell=1,2$ in Table \ref{table1}.
\begin{proposition}\label{prop:intb}
There exists a corrector $\W_{\rm{(b)}}^1=(u_{\rm{(b)}}^1, w_{\rm{(b)}}^1, b_{\rm{(b)}}^1)^T$ which solves
\begin{align*}
\de_t \W_{\rm{(b)}}^1 + \mathcal{L}_\eps \W_{\rm{(b)}}^1 = - \delta \times (\rm{(b1)}+\rm{(b2)}) + r_{\rm{(b)}}^1,\\
u_{\rm{(b)}}^1|_{y=0}=w_{\rm{(b)}}^1|_{y=0}=\de_y b_{\rm{(b)}}^1|_{y=0}=0,
\end{align*}
with $\|r_{\rm{\rm{(b)}}}^1\|_{L^2} = O(\delta \eps^{\frac 14} \sigma^{-\frac{11}{6}\mu})$.
The corrector is composed of
$$\W_{\rm(b)}^1=\W_{\rm{(b)}; \rm{II}}^1 + \W_{\rm{(b)}; \rm{MF}}^1+ \W^1_{\rm{(b)}; \rm{BL}, \eps^{1/2}},$$
where for the boundary layer part, we have
$$\|\W^1_{\rm{(b)}; \rm{BL}, \eps^{1/2}}\|_{L^2(\mathbb{R}^2)} =O(\delta \eps^{-\frac{11}{6}\mu }),$$
while 
$$ \|\W_{\rm{(b)}; \rm{MF}}^1\|_{L^2} \sim \|\W_{\rm{(b)}; \rm{II}}^1\|_{L^2} =O(\delta \eps^{-\frac{1}{12}-\frac{10}{3}\mu }).$$
\end{proposition}
The proof is identical to the proof of Proposition \ref{prop:inta} as the interactions between \bw and/or \blbw are exactly the same as before modulo a re-scaling in $\eps$. In fact, in this case every object has an additional smallness factor in terms of positive powers of $\eps$ (that is precisely $\eps^{\frac{1}{12}}$), so that it is smaller than the corresponding object in the proof of Proposition \ref{prop:inta}, but it has the same structure. Therefore, we omit the proof of this proposition as it does not add any new idea and it is already detailed in \cite{BDSR19} in the case of strongly (frequency-) localized packets of waves. 

\subsection{Proof of Proposition \ref{prop:nonlin}}
We immediately have from Proposition \ref{prop:inta} and Proposition \ref{prop:intb} that, denoting $\W^1:=\W^1_{\rm{(a)}}+\W^1_{\rm{(b)}}$, it holds 
\begin{align*}
\de_t \W^1+\mathcal{L}_\eps \W^1&=- \delta (\rm{(a1)+(a2) + (b1) + (b2)}) + r^1_{\rm{(a)}}+r^1_{\rm{(b)}}\\
&=-\delta Q (\W^0, \W^0) + \delta (\rm{(c1)+(c2) + (d1) + (d2)+(d3)} ) + r^1_{\rm{(a)}}+r^1_{\rm{(b)}}\\
&= -\delta Q (\W^0, \W^0)+r^1,
\end{align*}
where
\begin{align}\label{def:r1}
r^1:=  \delta (\rm{(c1)+(c2) + (d1) + (d2)+(d3)} ) + r^1_{\rm{(a)}}+r^1_{\rm{(b)}}. 
\end{align}
From Table \ref{table1}, Proposition \ref{prop:inta} and Proposition \ref{prop:intb}, we deduce that 
\begin{align*}
\|r^1\|_{L^2} \le O(\delta \|\rm{(c1)}\|_{L^2}) = O(\delta\eps^\frac 16\sigma^{-2})=O(\delta \eps^{\frac 16-2\mu}). 
\end{align*}

\section{Higher-order Approximation}\label{ltc}
A nice advantage of space-localized beam waves with respect to packets of waves in \cite{BDSR19} is the possibility to construct an approximate solution $\W^{\rm{app}}$ (in the previous sections) that is the some of a finite number of terms, \emph{each of them being physically relevant}. Every term of $\W^{\rm{app}}$ in Theorem \ref{thm:main1} is indeed an approximate solution of \eqref{eq:system}, with a remainder in $L^2$ which is strictly smaller than its $L^2$ size. 
We recall that in \cite{BDSR19}, the linear mean flow corrector $\W_{\rm{MF}}^1$ (Theorem \ref{thm:main1}) had a too singular amplitude combined with a too slow decay, in such a way that its $L^2$ norm had a priori a bad dependency on the small physical parameter $\eps$. Thus, in \cite{BDSR19}, that mean flow contribution was replaced by a non-physical term, which balances the error of the approximate solution on the boundary but it is not (even an approximate) solution to \eqref{eq:system}. 
In this work, in the previous sections, we succeeded to overcome this problem using (almost) axisymmetric beam waves which are spatially localized, in such a way that we can actually use the true mean flow to contruct $\W^{\rm{app}}$ and we do not introduce any non-physical corrector.
This allows us to push the approximation at next orders, using a nice cancellation of the triadic interactions that was discovered in \cite{DY1999} (see also \cite{Akylas}), and it is due to the \emph{null form structure} of the convection term for incompressible fluids, \cite{Klainermann}. 

This is the content of this last section, whose aim is to prove Theorem \ref{thm:main2}. To this end, we have to identify the next order (triadic) contributions to be corrected, which are originated by the interactions in $Q(\W^1, \W^0), Q(\W^0, \W^1)$, where $\W^0$ is the linear solution given by Proposition \ref{prop:BL}, and $\W^1$ is the corrector provided by Proposition \ref{prop:nonlin}. We recall from Proposition \ref{prop:nonlin} that $\W^1=\W^1_{\rm{MF}} + \W^1_{\rm{II}},$ where $\W^1_{\rm{MF}}$ is a mean flow, so that it has nearly zero time frequency, while $\W^1_{\rm{II}}$ is a second harmonic, so its time frequency is approximately twice the time frequency of the incident wave. Now, as showed in \eqref{eq:matrix-L}, the linear operator $L$ (i.e. the approximation of the linear operator $\mathcal{L}_\eps$ in the scaling of the boundary layer neglecting viscosity and dissipation) of system \eqref{eq:system} has pure imaginary eigenvalues $\pm i \omega_0 = \pm i \sin \gamma$, where $\omega_0$ is nearly the frequency of the incident beam wave. Therefore, when the triadic terms in  $Q(\W^1, \W^0),  Q(\W^0, \W^1)$ oscillate approximately with the same time frequency 
$$\de_t \W^{\rm{app}} + \mathcal{L}_\eps  \W^{\rm{app}} = - \delta ( Q(\W^1, \W^0)+ Q(\W^0, \W^1)),$$
the above source term is \emph{resonant} with respect to the oscillations of the linear part of the system and linear growths in time (secular growths) can appear \cite{DY1999, lannes}. Let us analyze the next order triadic term. We recall from Proposition \ref{prop:nonlin} that 
$$\W^1=\W^1_{\rm{II}} + \W^1_{\rm{MF}} + \W^1_{\rm{II, BL, \eps^{1/2}}}  + \W^1_{\rm{MF, BL, \eps^{1/2}}} + \W^1_{\rm{II, BL, \eps^{1/3}}} + \W^1_{\rm{MF, BL, \eps^{1/3}}}.$$
\begin{table}[h]
\caption{List of triadic interactions}\label{table2}
\centering
\begin{tabular}
{|l|l|l|l|}
\hline 
{} &\scriptsize {\bf type of interaction} &\scriptsize {\bf time frequency } &\scriptsize {\bf typical decay rate  }\\  \hline

{(A1)} &\scriptsize $Q(\W^1_{\rm{II, BL, \eps^{1/3}}}, \W_{\rm{BL}, \eps^{1/3}}^0)$ &\scriptsize $\pm \omega_0; \pm 3\omega_0$ &\scriptsize {$\eps^{-1/3}$ } \\ \hline

{(A2)} &\scriptsize $Q(\W^1_{\rm{MF, BL, \eps^{1/3}}}, \W_{\rm{BL}, \eps^{1/3}}^0)$ &\scriptsize $\pm \omega_0$ &\scriptsize {$\eps^{-1/3}$ }\\ \hline

{(B1)} &\scriptsize $Q(\W^1_{\rm{II}}, \W_{\rm{BL}, \eps^{1/3}}^0)$ &\scriptsize $\pm \omega_0; \pm 3\omega_0$ &\scriptsize {$\eps^{-1/3}$ }\\ \hline

{(B2)} &\scriptsize $Q(\W^1_{\rm{MF}}, \W_{\rm{BL}, \eps^{1/3}}^0)$ &\scriptsize $\pm \omega_0$ &\scriptsize {$\eps^{-1/3}$ }\\ \hline

{(C1)} &\scriptsize $Q(\W^1_{\rm{II, BL, \eps^{1/3}}}, \W_{\rm{BL}, \eps^{1/2}}^0)$ &\scriptsize $\pm \omega_0; \pm 3\omega_0$ &\scriptsize {$\eps^{-1/2}$ }\\ \hline

{(C2)} &\scriptsize $Q(\W^1_{\rm{MF, BL, \eps^{1/3}}}, \W_{\rm{BL}, \eps^{1/2}}^0)$ &\scriptsize $\pm \omega_0$ &\scriptsize {$\eps^{-1/2}$ }\\ \hline

{(D1)} &\scriptsize $Q(\W^1_{\rm{II}}, \W_{\rm{BL}, \eps^{1/2}}^0)$ &\scriptsize $\pm \omega_0; \pm 3\omega_0$ &\scriptsize {$\eps^{-1/2}$ }\\ \hline

{(D2)} &\scriptsize $Q(\W^1_{\rm{MF}}, \W_{\rm{BL}, \eps^{1/2}}^0)$ &\scriptsize $\pm \omega_0$ &\scriptsize {$\eps^{-1/2}$ }\\ \hline

{(E1)} &\scriptsize $Q(\W^1_{\rm{II, BL, \eps^{1/2}}}, \W_{\rm{BL}, \eps^{1/3}}^0)$ &\scriptsize $\pm \omega_0; \pm 3\omega_0$ &\scriptsize {$\eps^{-1/2}$ }\\ \hline

{(E2)} &\scriptsize $Q(\W^1_{\rm{MF, BL, \eps^{1/2}}}, \W_{\rm{BL}, \eps^{1/3}}^0)$ &\scriptsize $\pm \omega_0$ &\scriptsize {$\eps^{-1/2}$ }\\ \hline

{(F1)} &\scriptsize $Q( \W_{\rm{BL}, \eps^{1/2}}^0, \W^1_{\rm{II}})$ &\scriptsize $\pm \omega_0; \pm 3\omega_0$ &\scriptsize {$\eps^{-1/2}$ }\\ \hline

{(F2)} &\scriptsize $Q(\W_{\rm{BL}, \eps^{1/2}}^0, \W^1_{\rm{MF}})$ &\scriptsize $\pm \omega_0$ &\scriptsize {$\eps^{-1/2}$ }\\ \hline

{(G1)} &\scriptsize $Q(\W^1_{\rm{II, BL, \eps^{1/2}}}, \W_{\rm{BL}, \eps^{1/2}}^0)$ &\scriptsize $\pm \omega_0; \pm 3\omega_0$ &\scriptsize {$\eps^{-1/2}$ }\\ \hline

{(G2)} &\scriptsize $Q(\W^1_{\rm{MF, BL, \eps^{1/2}}}, \W_{\rm{BL}, \eps^{1/2}}^0)$ &\scriptsize $\pm \omega_0$ &\scriptsize {$\eps^{-1/2}$ }\\ \hline

{(H1)} &\scriptsize $Q( \W_{\rm{BL}, \eps^{1/3}}^0, \W^1_{\rm{II}})$ &\scriptsize $\pm \omega_0; \pm 3\omega_0$ &\scriptsize {$\eps^{-1/3}$ }\\ \hline

{(H2)} &\scriptsize $Q(\W_{\rm{BL}, \eps^{1/3}}^0, \W^1_{\rm{MF}})$ &\scriptsize $\pm \omega_0$ &\scriptsize {$\eps^{-1/3}$ }\\ \hline

{(I1)} &\scriptsize $Q(\W_{\rm{inc}}^0, \W^1_{\rm{II}})$ &\scriptsize $\pm \omega_0; \pm 3\omega_0 $ &\scriptsize  {no decay }\\ \hline

{(I2)} &\scriptsize $Q(\W_{\rm{inc}}^0, \W^1_{\rm{MF}})$ &\scriptsize $\pm \omega_0$ &\scriptsize  {no decay }\\ \hline

{(L1)} &\scriptsize $Q(\W^1_{\rm{II}}, \W_{\rm{inc}}^0)$ &\scriptsize $\pm \omega_0; \pm 3\omega_0 $ &\scriptsize  {no decay }\\ \hline

{(L2)} &\scriptsize $Q(\W^1_{\rm{MF}}, \W_{\rm{inc}}^0)$ &\scriptsize $\pm \omega_0$ &\scriptsize  {no decay }\\ \hline
\hline

\end{tabular}
\label{Tab:interactions}
\end{table}

All the terms in Table \ref{table2} could potentially generate secular growths. From Proposition \ref{prop:nonlin} and Proposition \ref{prop:BL}, we know that 
\begin{align}
\|\W^1_{\rm{II, BL}, \eps^{1/3}}\|_{L^2(\mathbb{R}_+^2)} & =O( {\delta}\eps^{-\frac 16} \sigma^{-\frac 76}); \quad \|\W^1_{\rm{MF, BL}, \eps^{1/3}}\|_{L^2(\mathbb{R}_+^2)} = O( {\delta}\eps^{-\frac 16} \sigma^{-\frac 76}); \notag \\
 \|\nabla \W^0_{\rm{BL}, \eps^{1/3}}\|_{L^\infty(\mathbb{R}_+^2)} & =O(  \sigma^{-\frac 23}\eps^{-\frac 12}).\label{eq:order} 
 \end{align}
Therefore, we can deduce from Lemma \ref{estbbl} that
\begin{align}\label{ineq:big-rem}
\delta \|(A1)\|_{L^2} \sim \delta  \|(A2)\|_{L^2} = O({\delta^2} \eps^{-\frac 23} \sigma^{-\frac{11}{6}}).
\end{align}
The triadic interactions of Table \ref{table2} have decreasing $L^2$ norm. We look at the last one. 
From Proposition \ref{prop:nonlin}, we know that 
$$ \|\W^1_{\rm{II}}\|_{L^2(\mathbb{R}^2)} \sim \|\W^1_{\rm{MF}}\|_{L^2(\mathbb{R}^2)}  =O(\delta \eps^{-\frac 16} \sigma^{-\frac{10}{3}}); \quad  \|\nabla \W^1_{\rm{MF}}\|_{L^2(\mathbb{R}^2)} \sim \|\nabla \W^1_{\rm{II}}\|_{L^2(\mathbb{R}^2)} = O(\delta \eps^{-\frac 16} \sigma^{-\frac{13}{3}}).$$
We also notice that
$$\|\nabla \W_{\rm{inc}}\|_{L^2} =O(\sigma^{-1}).$$

We can deduce that
$$\delta \|(L1)\|_{L^2} \sim \delta \|(L2)\|_{L^2} = O(\delta^2 \eps^{-\frac 16} \sigma^{-\frac{13}{3}}).$$
We show in the following that we can correct all the terms of the above table. Then the error $\|\widetilde{R}_{\rm{app}} \|_{L^2}$ of our consistent  approximation $\widetilde W^{\rm{app}}$ will be of the size 
stated in Theorem \ref{thm:main2}. 
We finally remark that once the terms of Table \ref{table2} will be corrected in the following, the correctors will be trilinear with time oscillation $\sim 0, \pm 2\omega_0$, and they could give rise to secular growths due to their interaction with $\W^1$ in Proposition \ref{prop:nonlin}.

Lastly, we show how to exploit the nice cancellation in the convection term to correct the first triadic term $(A1)$. Since it would not add anything but long and tedious computations, we omit the treatment of the other terms of Table \ref{table2}.
For $(A1)$, whose decay in $y$ is order $\eps^{\frac 13}$, we exploit the scaling of the boundary layer.
As done in Section 3, we reduce the system \eqref{eq:system} to the equations for $u, b$ ($(U, B)$ denote their Fourier transform). For $\mathcal{V}=(\V_1, \V_2)=(U, B)^T$, the equations assume the general structure
\begin{align*}
\de_t \mathcal{V} + L  \mathcal{V} &= \delta \exp (ilx-i\alpha t - \L y) \mathcal{V}', 
\end{align*}
where 
\begin{align}\label{eq:matrix-L}
L=\sin \gamma \begin{pmatrix}
0 & -1 \\
1 & 0
\end{pmatrix}.
\end{align}

We now introduce (and explain the formal motivation of) a scaling in time. Notice that the boundary layer $\W^0_{\rm{BL}, \eps^{1/3}}$ in Proposition \ref{prop:BL} has a vertical component of the velocity field that is approximately of order $\eps^{\frac 13}$ in $L^2$ or in $L^\infty$. Since the vertical component of the velocity of the boundary layer is $O(\eps^\frac 13)$ in $L^\infty$ and has to balance the boundary contribution of the incident wave of ${O}(1)$, this balance will be effective after a time of the order $\eps^{-\frac 13}$ as the group velocity $c_g=\nabla_{k,m} \omega = O(\kk^{-1})=O(1)$.
This is the formal reason behind the time scaling $t={\tau}{\eps^{-\frac 13}}$, as suggested by \cite{DY1999}.
\bcb{\subsection{Proof of Theorem \ref{thm:main2}}\label{sec:proof2}
We define a scaled time variable $\tau=\eps^{\frac 13}t$, so that $\partial_t = \eps^{\frac 13} \partial_\tau$ and we rewrite the previous compact system as in \eqref{eq:compactsystem}
\begin{equation}
\label{BL2_scaled_compact_system}
\partial_\tau \V+\eps^{-\frac 13} {L}\V=-{\delta}{\eps^{-\frac 13}} Q(\V, \V).
\end{equation}
We recall from Section 3 that $\exp({\eps^{-\frac 13}} L t)=\exp(i \eps^{-\frac 13} {\sin \gamma} \tau ) \Pi_+  + \exp(-i\eps^{-\frac 13} {\sin \gamma} \tau ) \Pi_-$.
Now define the \emph{filtered variable} 
$$\widetilde \V:=\exp(\eps^{-\frac 13} {L \tau }) \V ,$$ so that the system reads 
$$\partial_\tau  \widetilde \V +\eps^{-\frac 13} {\delta}\exp(\eps^{-\frac 13} {Lt})  Q(\exp(-\eps^{-\frac 13}{Lt} }) \widetilde \V , \exp(-\eps^{-\frac 13} {Lt}) \widetilde \V )=0.$$
We also use the notation $\widetilde \V^+=\Pi_+ \widetilde \V, \quad \widetilde \V^-=\Pi_- \widetilde \V=\widetilde \V^+$, where 
$$\Pi_+=\frac 12 \begin{pmatrix} 1 & i \\
-i & 1
\end{pmatrix}; \quad \Pi_-=\bar \Pi_+.$$
Exploiting the mutual orthogonality of the projectors, we have 
\begin{equation}
\label{Eq_V_pm}
\begin{aligned}
\partial_\tau \widetilde \V^\pm & + \frac{\delta}{\eps^{1/3}} \exp(i \frac{i\sin \gamma }{\eps^{1/3}}\tau ) (\widetilde \V _1^+, \, -\partial_y^{-1}\partial_x \widetilde \V _1^+) \cdot \nabla \widetilde \V^\pm\\
&+ \frac{\delta}{\eps^{1/3}} \exp(-i \frac{\sin \gamma}{\eps^{1/3}}\tau ) (\widetilde \V_1^-, \, -\partial_y^{-1}\partial_x \widetilde\V_1^-) \cdot \nabla\widetilde \V^\pm=0.
\end{aligned}
\end{equation}
At this stage, notice that the approximate solution that we built in the previous section, i.e. $\W^{\rm{app}}=\W^0+\W^1$, with $\W^0$ given by Proposition \ref{prop:BL} and $\W^1$ by Proposition \ref{prop:nonlin}, can be written as $\W^{\rm{app}}= \W^0+ \delta \W^{1, \delta}$, since the expression of $\W^1$ in Proposition \ref{prop:nonlin} is multiplied by the weakly nonlinear parameter $\delta$. Notice that quadratic terms in \eqref{Eq_V_pm} are multiplied by $\eps^{-\frac 13} \delta$. This means that our first-order corrector $\delta \W^{1, \delta}$, where the singular factor $\eps^{-\frac 13}$ does not appear, is the result of time integration of the quadratic terms in \eqref{Eq_V_pm}. Thus, in order to push the approximation at the next order and decrease its error, we expect to add a corrector of the form $\delta^2 \W^2$. We denote the approximate solution of the previous section in filtered form as $$\widetilde \V^{\rm{app}}_0=\widetilde \V^0+\delta \widetilde \V^1.$$
Now, we want an approximate solution of the form $$\widetilde \V^{\rm{app}}_1=\widetilde \V^0+\delta \widetilde \V^1+ \delta^2\widetilde \V^2,$$ where the last term has to be determined, while we recall from the previous section that $\widetilde \V^1$ solves the equation
$$\de_\tau \widetilde \V^1+\frac{\delta}{\eps^{1/3}} \exp(\frac{Lt}{\eps^{1/3}})  Q(\exp(-\frac{Lt}{\eps^{1/3} }) \widetilde \V^0 , \exp(-\frac{Lt}{\eps^{1/3}}) \widetilde \V^0 )=0.$$
To shorten the notation, we denote 
$$\B(\V, \V):=  \exp(\frac{Lt}{\eps^{1/3}})  Q(\exp(-\frac{Lt}{\eps^{1/3} })  \V , \exp(-\frac{Lt}{\eps^{1/3}}) \V ),$$
so that the system reads
$$\de_\tau \V + \frac{\delta}{\eps^{1/3}} \B(\V, \V)=0.$$
As we look for the next order corrector $\widetilde \V^2$, we plug the ansatz $\widetilde \V^{\rm{app}}_1=\widetilde \V^{\rm{app}}_0+ \delta^2 \widetilde \V^2$ inside the compact system, and we obtain  
\begin{align}
\de_\tau \widetilde \V^{\rm{app}}_0 &+ \delta^2 \de_\tau \widetilde \V^2 + \frac{\delta}{\eps^{1/3}}\B(\widetilde \V^0, \widetilde \V^0) +  \frac{\delta^2}{\eps^{1/3}}(\B(\widetilde \V^0, \widetilde \V^1)+ \B(\widetilde \V^1, \widetilde \V^0)) \notag \\
& + \frac{\delta^3}{\eps^{1/3}}(\B(\widetilde \V^0, \widetilde \V^2)+ \B(\widetilde \V^2, \widetilde \V^0)) + \frac{\delta^4}{\eps^{1/3}}(\B(\widetilde \V^2, \widetilde \V^1)+ \B(\widetilde \V^1, \widetilde \V^2)) + \frac{\delta^5}{\eps^{1/3}}\B(\widetilde \V^2, \widetilde \V^2)=0.\label{eq:app}
\end{align}
As already pointed out above, $\widetilde \V^{\rm{app}}$, constructed in the previous section, solves $ \de_\tau \widetilde \V^{\rm{app}}_0  + \frac{\delta}{\eps^{1/3}}\B(\widetilde \V^0, \widetilde \V^0)=0$. 

\noindent The next order approximation is obtained by requiring that $\widetilde \V^2$ solves
\begin{align}\label{eq:tV2}
\delta^2 \de_\tau \widetilde \V^2+  \frac{\delta^2}{\eps^{1/3}}(\B(\widetilde \V^0, \widetilde \V^1)+ \B(\widetilde \V^1, \widetilde \V^0))=0.
\end{align}
Recalling the definition of the bilinear form $\B(\cdot, \cdot)$ and equation \eqref{Eq_V_pm}, we write more explicitly, with the notation $\widetilde \V^{\iota, \pm}=(\widetilde \V^{\iota, \pm}, \widetilde \V^{\iota, \pm})$, $\iota \in \{0, 1\}$, that
\begin{align}
\delta^2 \de_\tau \widetilde {\V}^{2, \pm} 
& + \frac{\delta^2}{\eps^{1/3}} \exp(i \frac{\sin \gamma }{\eps^{1/3}}\tau ) (\widetilde \V _1^{0,+}, \, -\partial_y^{-1}\partial_x \widetilde \V_1^{0,+}) \cdot \nabla \widetilde \V^{1,\pm} \quad \rm{(a)} \notag\\
&+ \frac{\delta^2}{\eps^{1/3}} \exp(-i \frac{\sin \gamma}{\eps^{1/3}}\tau ) (\widetilde \V_1^{0,-}, \, -\partial_y^{-1}\partial_x \widetilde \V_1^{0,-}) \cdot \nabla\widetilde \V^{1,\pm}\quad \rm{(b)}\notag\\
& + \frac{\delta^2}{\eps^{1/3}} \exp(i \frac{\sin \gamma }{\eps^{1/3}}\tau ) (\widetilde \V _1^{1,+}, \, -\partial_y^{-1}\partial_x \widetilde \V _1^{1,+}) \cdot \nabla \widetilde \V^{0,\pm}\quad \rm{(c)}\notag\\
&+ \frac{\delta^2}{\eps^{1/3}} \exp(-i \frac{\sin \gamma}{\eps^{1/3}}\tau ) (\widetilde \V_1^{1,-}, \, -\partial_y^{-1}\partial_x \widetilde \V_1^{1,-}) \cdot \nabla\widetilde \V^{0,\pm} =0.\quad \rm{(d)} \label{eq:V2}
\end{align}
Now observe that, while $\widetilde \V^{0, \pm}$ does not oscillate in time, since its time oscillations have been filtered by means of the change of variable, the next order $\widetilde \V^{1, \pm}$ has time oscillations. More precisely, coming back to the non-filtered variable $ \V$, it holds 
\begin{align}\label{eq:V1next}
\widetilde \V^{1, \pm} = \exp(\pm i \eps^{-\frac 13} ({\sin \gamma}) \tau) \V^{1, \pm},
\end{align}
where $ \V^{1, \pm}= \V_{\rm{II}}^{1, \pm}+ \V_{\rm{MF}}^{1, \pm}$ constructed in the previous section contains a \emph{second hamornic} and a \emph{mean flow} contribution, which oscillate in time with frequency $\omega_{\rm{II}} \sim \pm 2 \sin \gamma, \; \omega_{\rm{MF}}  \sim 0$, thanks to the angular localization of our wave beam (due to the angular localization given by $\widehat{\Psi}_{\eps, \sigma}$ in \eqref{def:psi}). Therefore we deduce that $\widetilde \V^{1, \pm}$ in \eqref{eq:V1next} contains contributions which oscillate in time with frequency $\sim \pm {(\sin \gamma)}{\eps^{-\frac 13}}, \, \pm 3{(\sin \gamma)}{\eps^{-\frac 1 3}}.$ The dangerous point is that in \eqref{eq:V2} the oscillations $\exp(\pm \eps^{-\frac 13}{\sin \gamma})$ hit $\widetilde \V^{1, \pm}$. As some terms of $\widetilde \V^{1, \pm}$ oscillate with frequency $\pm {(\sin \gamma)}{\eps^{-\frac 13}}$, there are some non-oscillating terms (resonances) in equation \eqref{eq:V2}. When this is the case, $\widetilde \V^2$ satisfies an equation of the form
$$\de_\tau \widetilde \V^2 + f(\widetilde \V^{\rm{app}})_{\rm{MF}} + (\text{oscillating terms})=0,$$
where $f(\widetilde \V^{\rm{app}})_{\rm{MF}}$ is generic a function of a non-oscillating contribution in time. 
The solutions to this equations has a linear growth in time, which is often named \emph{secular growth} \cite{DY1999}. The only hope to avoid this growth, which would prevent ourselves to be able to obtain an improved approximate solution, is that the non-oscillating terms cancel thanks to the structure of the quadratic/bilinear form. This is precisely what happens in this context, as we show below. We also point out that, even though at this stage this can appear as a miraculous cancellation, it is actually related to the \emph{null form structure} of the convection term for incompressible fluids, see for instance \cite{Klainermann}. 
Now, in equation \eqref{eq:V2}, we recall that, in view of \eqref{Eq_V_pm}, 
\begin{align*}
\widetilde \V^{1, \pm}\sim & -\delta \exp(i \frac{\sin \gamma}{\eps^{1/3}} \tau) (\widetilde \V_1^{0, +}, -\de_{y}^{-1} \de_x \widetilde \V_1^{0, +}) \cdot \nabla \widetilde \V^{0, \pm}\\
&  -\delta \exp(-i \frac{\sin \gamma}{\eps^{1/3}} \tau) (\widetilde \V_1^{0, -}, -\de_{y}^{-1} \de_x \widetilde \V_1^{0, -}) \cdot \nabla \widetilde \V^{0, \pm}.
\end{align*}
Plugging this expression inside \eqref{eq:V2}, after some computations we obtain that 
\begin{align*}
\rm{(a)+(c)}&= [\de_{y}^{-1} (\nabla^\perp \widetilde V_1^{0, -} \cdot \nabla  \widetilde V_1^{0, +}), -  \de_{y}^{-1}\de_x \de_y^{-1}(\nabla^\perp \widetilde V_1^{0, -} \cdot \nabla  \widetilde V_1^{0, +})] \\
& \quad + [((\nabla^\perp \de_{y}^{-1}\widetilde V_1^{0, -} \otimes \nabla^\perp  \de_{y}^{-1} \widetilde V_1^{0, +}) \cdot \nabla] \cdot \nabla \widetilde V^{0, \pm} + \rm{(oscillating\,  terms)}; \\
\rm{(b)+(d)}&=-(\rm{(a)+(c)})+ \rm{(oscillating\,  terms)}.
\end{align*}
Therefore the non-oscillating (resonant) terms cancel out and we can correct the oscillating ones solving the equation for $\widetilde \V^2$. This implies that we can actually construct 
$$\widetilde \V^{\rm{app}}_1=\widetilde \V^0+\delta \widetilde \V^1+ \delta^2\widetilde \V^2,$$
with $\widetilde \V^2$ solving \eqref{eq:tV2}. Looking at \eqref{eq:app}, the potential secular growths are in the second non-corrected term, i.e. $ {\delta^4}{\eps^{-\frac 13}}(\B(\widetilde \V^2, \widetilde \V^1)+ \B(\tilde \V^1, \tilde \V^2))$. In fact, adding $\widetilde \V^2$ to the approximate solution, the first non-corrected term, given by $ {\delta^3}{\eps^{-\frac 13}}(\B(\widetilde \V^0, \widetilde \V^2)+ \B(\widetilde \V^2, \widetilde \V^0)) ,$ oscillates in time (this can be easily verified recalling that $\widetilde \V^2$ contains the oscillations $0, \pm 2, \pm 4$ and therefore cannot generate any secular growth). Finally, recalling that we corrected (A1) and (A2) in Table \ref{table2}, we claim that all the terms of Table \ref{table2} can be corrected exactly in the same way, so that using the estimates \eqref{eq:order}, we obtain that in the original time-scale $t=\eps^{-\frac 13} \tau$ the worst remainder is estimated by Lemma \ref{estbbl} as
\begin{align*}
\|\widetilde{R}_{\rm{app}} \|_{L^2} & =o( \delta^3 t \|w_{\rm{BL, \eps^\frac 13}} \de_y u_{\rm{BL, \eps^\frac 13}} \times w_{\rm{BL, \eps^\frac 13}} \de_{yy}u_{\rm{BL, \eps^\frac 13}}\|_{L^2})\\
& =o(\delta^3 t \|w_{\rm{BL, \eps^\frac 13}} \de_y u_{\rm{BL, \eps^\frac 13}}\|_{L^\infty} \|w_{\rm{BL, \eps^\frac 13}} \de_{yy}u_{\rm{BL, \eps^\frac 13}}\|_{L^2})\\
&=o(\delta^3 \eps^{-\frac 56} \sigma^{-\frac{19}{6}}t),
\end{align*}
where $w_{\rm{BL, \eps^\frac 13}}$ is a \blbw of order $(\frac 13 , \frac 13, 0, 0)$, $\de_y u_{\rm{BL, \eps^\frac 13}}$ is a \blbw of order $(\frac 13, \frac 13, - \frac 23, - \frac 13)$ and $\de_{yy} u_{\rm{BL, \eps^\frac 13}}$ is a \blbw of order $(\frac 13, \frac 13, -1, 0)$.

\begin{appendix}
\bcr{
\section{Proof of Lemma \ref{estbbl}}\label{prooflemma}
By definition \eqref{defbeam}, $v_{\rm{beam}}^t$ is the Fourier transform of the function
\begin{align}\label{functiong}
g(\k)=\sqrt{2\pi}\frac{\sigma}{\eps^{1/6}} \left( \widehat{\Psi}_{\eps,\sigma}({\kk}, \theta) e^{-i\omega t}\right)|_{{\kk}\geq \eta}
\end{align}
where $\left(\kk, \frac \pi 2-(\theta+\gamma)\right)$ are the polar coordinates of $\k$ in the rotated reference system $(x,y)$ for $\theta\sim \gamma$ (i.e. the first addend inside the integral \eqref{defbeam}, given by the first term of $\widehat{\Psi}_{\eps, \sigma} (\kk, \theta)$ in Definition \ref{def:inc-beam}, while the second addend is just the complex conjugate of the first one, yielding  $\theta \sim \gamma+\pi $). We focus on the first term $\theta \sim \gamma$ and from now on, we use the shortened notation $$\widehat{\Psi}_{\eps, \sigma} (\kk, \theta)=\G_{\eps, \sigma} (\sigma \kk, \eps^{-1/3}(\theta-\gamma)) O(\eps^\nn \kk^\ta)$$ for $\widehat{\Psi}_{\eps, \sigma}$ in Definition \ref{def:inc-beam}. By the Plancherel equality, one has
\begin{eqnarray}
\|v_{\rm{beam}}^t\|_{L^2}^2&=& 2
\pi\frac{\sigma^2}{\eps^{1/3}}
\int_{\et}^\infty \int_0^{2\pi}  \G_{\sigma,\eps}^2 (\sigma \kk, \eps^{-1/3}(\theta-\gamma))O(\eps^{2\nn} {\kk}^{2\ta}){\kk} d{\kk} \nonumber\\
&=&
O(\eps^{2\nn})2
\pi\frac{\sigma^2}{\eps^{1/3}}
\int_{\et}^\infty \int_0^{2\pi}  \G_{\sigma,\eps}^2 (\sigma \kk, \eps^{-1/3}(\theta-\gamma)) {\kk}^{1+2\ta} d{\kk} d\theta,\nonumber
\nonumber\\
\end{eqnarray}
and the estimate easily follows. For the boundary layer term, one first notices that
\begin{align}\label{ineq:BLappendix}
\|v_{\rm{BL}}^t\|_{L^2}\leq \sqrt{2
\pi}\frac{\sigma}{\eps^{1/6}}
\int_0^{2\pi}
\left\|
\int_{\et}^\infty  { \G_{\sigma,\eps} (\sigma \kk, \eps^{-1/3}(\theta-\gamma))O(\eps^{\nn} {\kk}^{\ta})}
e^{- i\omega t +  i {\kk} \sin (\theta+\gamma) (x-x_0) {-\L(k)^{\alpha, \beta}y}}
{\kk} d{\kk}\right\|_{L^2_{x,y}} d\theta.
\end{align}

By the Plancherel equality again, one has
\begin{eqnarray}
&&
\left\|
\int_{\et}^\infty   \G_{\sigma,\eps} (\sigma \kk, \eps^{-1/3}(\theta-\gamma))O(\eps^{\nn} {\kk}^{\ta})
e^{- i\omega t +  i {\kk} \sin (\theta+\gamma) (x-x_0)  - {\L(k)^{\alpha, \beta}y}}
{\kk} d{\kk}\right\|_{L^2_{x}}^2\nonumber\\ 
&\leq&O(\eps^{2\nn})
\tfrac1{\sin (\theta+\gamma)^2}
\int_\et^\infty
\G' (\sigma \kk, \eps^{-1/3} (\theta-\gamma))^2
{\kk}^{2\ta+2}e^{-2{\Re({\L}(k)^{\alpha, \beta}})y}d{\kk}.\nonumber
\end{eqnarray}
Indeed, the function
\begin{align}
\widehat{f}(\kk):=\sqrt{2\pi}\int_{\et}^\infty   \G_{\sigma,\eps} (\sigma \kk, \eps^{-1/3}(\theta-\gamma))\kk^{\ta+1}
e^{- i\omega t +  i {\kk} \sin (\theta+\gamma) (x-x_0)  -\L(k)^{\alpha, \beta} y}
 d{\kk}
\end{align} 
is nothing but the Fourier transform in ${\kk}$ of the function
\begin{align}\label{eq:functionf}
 f(\kk):=\G_{\sigma,\eps} (\sigma \kk, \eps^{-1/3}(\theta-\gamma))\kk^{\ta+1}
e^{- i\omega t + i {\kk} \sin (\theta+\gamma) x_0-\L(k)^{\alpha, \beta} y}
\end{align}
evaluated at the point $\sin (\theta+\gamma) x$.
It is enough now to notice that
\begin{eqnarray}
 && \int_\et^\infty
{\G_{\eps, \sigma}(\sigma \kk, \eps^{-1/3} (\theta-\gamma))^2}
{\kk}^{2\ta+2}e^{-2\Re(\L(k)^{\alpha, \beta})y}d{\kk}\nonumber\\
&\leq&
e^{-2\inf\limits_{k}\Re(\L(k)^{\alpha, \beta})y}
 \int_\et^\infty
{\G_{\eps, \sigma} (\sigma \kk, \eps^{-1/3} (\theta-\gamma))^2}
{\kk}^{2\ta+2}d{\kk}\nonumber\\
&=&
e^{-2\inf\limits_{k}\Re(\L(k)^{\alpha, \beta})y} \sigma^{-2\ta-3}{\int_\et^\infty \G'(p, \eps^{-1/3}(\theta-\gamma))^2 \, p^{2\ta+2} \, dp}.\nonumber
\end{eqnarray}
{Thanks to the assumptions on $\L(k)^{\alpha, \beta}$ in Definition \ref{def:inc-beam}, there exists $\tilde \ell$ such that $\Re(\tilde \ell) >0$ and $\inf_{k} \Re (\L(k)^{\alpha, \beta}) \ge \Re(\tilde \ell) \frac{\kk^\beta}{\eps^{\alpha}} \ge  \Re(\tilde \ell) \frac{\et^\beta}{\eps^{\alpha}}$, so that, integrating in $y$, one has }

$$
\| e^{-2\Re(\L(k)^{\alpha, \beta}) y}\|_{L^2_y}^2 \le  \|e^{-2\Re{(\tilde \ell)}\frac{\et^\beta}{\eps^\alpha}y}\|^2_{L^2_y}
=
O(\eps^\alpha).
$$

Integrating in $\theta$, it yields 
$$
\int_0^{2\pi}
{ \frac{1}{\sin (\theta+\gamma)^2} \int_\et^\infty \G_{\eps, \sigma} (p, \eps^{-1/3}(\theta-\gamma))^2 \, p^{2\ta+2} \, dp} \, d\theta=O(\eps^\frac13),
$$
and finally, we obtain
$$
\|v_{\rm{BL}}^t\|_{L^2_{x,y}}=O(\eps^{\frac16+\nn+\frac\alpha2}\sigma^{-\ta-\frac12}).
$$

Note that 
when $\alpha=\tfrac13,\  \ta=-\frac23,\ \nn=-\frac13$ (the widest boundary layer of Proposition \ref{prop:BL}),  
we get
$$
\|v_{\rm{BL}}^t\|_{L^2}=O(\sigma^{\frac16}).
$$

}

For the $L^\infty$ norm of the wave beam, simply notice that
\begin{align*}
\|v_{\rm{beam}}^t\|_{L^\infty(\R^2_+)} & \le \eps^{\nn-\frac 16} \sigma \int_0^{2\pi} \int_\et^\infty \G_{\sigma,\eps} (\sigma \kk, \eps^{-1/3}(\theta-\gamma)) \, {\kk}^{1+\ta} \, d{\kk} \, d\theta= O(\eps^{\nn+\frac 1 6} \sigma^{- \ta-1}).
\end{align*}
Similarly, the $L^\infty$ norm of the wave beam boundary layer yields 
\begin{align*}
\|v_{\rm{BL}}^t\|_{L^\infty(\R^2_+)} & \le  \eps^{\nn-\frac 16} \sigma \int_0^{2\pi} \int_\et^\infty \G_{\sigma,\eps} (\sigma \kk, \eps^{-1/3}(\theta-\gamma))\, {\sup_{y\geq 0}(e^{-\Re (\L(k)^{\alpha, \beta}) y)}} \, {\kk}^{1+\ta} \, d{\kk} \, d\theta\\
& \le   \eps^{\nn-\frac 16} \sigma 
 \int_\et^\infty \G_{\sigma,\eps} (\sigma \kk, \eps^{-1/3}(\theta-\gamma))\,  {\kk}^{1+\ta} \, d{\kk} \, d\theta = O (\eps^{\nn+\frac 16} \sigma^{-\ta-1}). 
\end{align*}

We further prove the remaining claims.

(i) Notice from \eqref{defbeam} that 
$$\de_x v_{\rm{beam}}^t = \frac{\sigma}{\eps^{1/6}} \int_\eta^\infty \int_0^{2\pi} i {\kk}\sin(\theta+\gamma) \G_{\sigma,\eps} (\sigma \kk, \eps^{-1/3}(\theta-\gamma))O(\eps^\nn \kk^\ta) e^{-i\omega t + i{\kk} \sin(\theta+\gamma) (x-x_0)+i{\kk} \cos(\theta+\gamma) y} {\kk} \, d{\kk} d\theta.$$

Therefore $|i {\kk}\sin(\theta+\gamma)\G_{\sigma,\eps} (\sigma \kk, \eps^{-1/3}(\theta-\gamma)) O(\eps^\nn \kk^\ta)|=O(\eps^\nn\kk^{\ta+1})$ so that, by Definition \ref{def:inc-beam}, $\de_x \W_{\rm{beam}}^t$ is \bw of order $(\nn,\ta+1)$. The same argument works for  $\de_y \W_{\rm{beam}}^t$.

(ii) The conclusion follows by applying exactly the same reasoning as for (i).\\

(iii) To estimate the product of two \blbw of order $(\alpha, \beta, \nn, \ta)$ and $(\alpha', \beta', \nn', \ta')$ respectively, it is enough to notice that
\begin{align*}
\|v_{\rm{BL}}^t \times v_{\rm{BL}}'^t\|_{L^2(\R \times \R_+)} &\le \min\{ \|v_{\rm{BL}}^t\|_{L^2(\R \times \R_+)} \|v_{\rm{BL}}'^t\|_{L^\infty(\R \times \R_+)}, \,  \|v_{\rm{BL}}^t\|_{L^\infty(\R \times \R_+)} \|v_{\rm{BL}}'^t\|_{L^2(\R \times \R_+)}\}\\
& = O(\eps^{\nn+\nn'+\frac 13 + \frac{\max\{\alpha, \alpha'\}}{2}} \sigma^{-\ta-\ta'-\frac 32}),
\end{align*}
where the estimates in \eqref{eq:estimatenorm} have been applied. \\
The product of a \blbw $v_{\rm{BL}}^t$ and a \bw $v_{\rm{beam}}^t$, and of two \bw  $v_{\rm{beam}}^t,\  v_{\rm{beam}}'^t$ are estimated exactly in the same way.
\section{A useful lemma}\label{lemasypm}
The proofs of Lemmata \ref{lem:asym}, \ref{lem:asym1} and \ref{lem:asym3} rely on the following simple and useful result.
\begin{lemma}\label{lem:roots}
Let $g
$ be any complex polynomial. Let us suppose that, for some $\mu_0$,
$$
g(\mu_0)\leq\frac{|g'(\mu_0)|^2}{4\sup\limits_{|\nu|\leq 2\frac{|g(\mu_0)|}{|g'(\mu_0)|}}
|g''(\nu)|}.
$$
Then there exists a unique root $\mu
$ of $g$ such that
$$
|\mu-\mu_0|\leq 2\frac{|g(\mu_0)|}{|g'(\mu_0)|}.
$$
\end{lemma}

\begin{proof}
Let $\eps_0:=g(\mu_0),\ \alpha_0:=g'(\mu_o)$ and
$$f(x)=\frac{g(x+\mu_0)-\eps_0}{\alpha_0}.
$$
Then, 
$$
f(0)=0,\ f'(0)=1.
$$
Therefore $f$ satisfies the hypothesis of \cite[Lemma 1.3\ p. 130]{lang}
which can be formulated as:
\vskip 0.5cm
``if $\sup\limits_{|x|,|z|\leq r}|f'(x)-f'(z)|\leq s(r)<1$ and $|y|\leq (1-s(r))r$, then $\exists ! $x$,\ |x|\leq r$ such that $y=f(x)$." 
\vskip 0.5cm
 We have
$$
f(x)=y\Longleftrightarrow 
g(\mu_0+x)=y\alpha_0+\eps_0
$$
so that
$$
g(\mu_0+x)=0\Longleftrightarrow y=-\frac{\eps_0}{\alpha_0}.
$$
Therefore, the lemma is proved as soon as one can find 
$s_0<1$ such that
\be\label{thetruc}
\sup\limits_{|x|,|z|\leq 2\frac{|\eps_0|}{|\alpha_0|}}|f'(x)-f'(z)|\leq s_0<1\mbox{  and } \frac{|\eps_0|}{|\alpha_0|}\leq(1-s_0)2\frac{|\eps_0|}{|\alpha_0|}\Longleftrightarrow
 s_0\leq\tfrac12.
\ee
Defining $s_0$ by
$$
s_0:=\tfrac1{|\alpha_0|} 2r\sup\limits_{|z|
\leq 2\frac{|\eps_0|}{|\alpha_0|}}|g''(z)|
\geq
\tfrac1{|\alpha_0|}\sup_{|x|,|z|\leq 2\frac{|\eps_0|}{|\alpha_0|}}|g'(x+\mu_o)-g'(z+\mu_o)|
\geq 
\sup_{|x|,|z|\leq 2\frac{|\eps_0|}{|\alpha_0|}}|f'(x)-f'(z)|,
$$
we find that \eqref{thetruc} is satisfied as soon as $\alpha_0\leq\tfrac12$ 
that is
$$
\tfrac{2|\eps_0|}{|\alpha_0|^2} \sup\limits_{|z|
\leq 2\frac{|\eps_0|}{|\alpha_0|}}|g''(z)|\leq\tfrac12
\Longleftrightarrow \eps_0\leq
\frac{\alpha_0^2}{4\sup\limits_{|z|\leq 2\frac{\eps_0}{\alpha_0}}|g''(z)|}
\Longleftrightarrow
g(\mu_0)\leq\frac{|g'(\mu_0)|^2}{4\sup\limits_{|\nu|\leq 2\frac{|g(\mu_0)|}{|g'(\mu_0)|}}
|g''(\nu)|}.
$$
\end{proof}
\end{appendix}
\section*{Acknowledgment}
RB is partially supported by the GNAMPA group of INdAM and the Royal Society International Exchange Grant. This work has received funding from the PRIN Project (MUR, Italy) n. 20204NT8W4 \textit{Nonlinear evolution PDEs, fluid dynamics and transport equations: theoretical foundations and applications}. 
TP acknowledges Istituto per le Applicazioni del Calcolo (IAC-CNR) for hospitality and financial support. 
Some computations in Section \ref{sec:proof2} have been realized, in the framework of plane waves, when RB was holding a postdoc position under the mentorship of A.-L. Dalibard and L. Saint-Raymond: RB warmly thanks both of them for their support. 
Finally, we both thank Thierry Dauxois for several discussions concerning experiments related to this work.

\end{document}